\documentclass[leqno,11pt]{amsart}
\usepackage{fourier}
\usepackage[margin=1in]{geometry}
\usepackage{amssymb,amsfonts,amsmath,fancyhdr,subfig,graphicx}
\usepackage{caption,latexsym,url,hyperref,mathrsfs,color,enumitem,amsthm}
\usepackage{indentfirst}

\DeclareCaptionFormat{cont}{#1 (cont.)#2#3\par}

\theoremstyle{plain}
\newtheorem{theorem}{Theorem}[section]
\newtheorem{lemma}[theorem]{Lemma}

\newtheorem{proposition}[theorem]{Proposition}

\theoremstyle{definition}

\newtheorem{remark}[theorem]{Remark}

\numberwithin{equation}{section}
\allowdisplaybreaks

\newcommand\myeq{\mathrel{\stackrel{\makebox[0pt]{\tiny d}}{=}}}

\def\le{\leqslant}
\def\ge{\geqslant}

\newcommand{\eps}{\varepsilon}

\numberwithin{equation}{section}

\begin{document}

\title{ Numerical study of ergodicity for the overdamped Generalized Langevin Equation with fractional noise}

\author{Di Fang} \address{ Department of Mathematics, University of
  Wisconsin-Madison, Madison, WI 53706, USA} \email{di@math.wisc.edu}

  \author{Lei Li} \address{Institute of Natural Sciences and School of Mathematical Sciences, Shanghai Jiao Tong University, Shanghai, 200240, P. R. China }
\email{leili2010@sjtu.edu.cn}

\date{}

\thanks{ The work of Di Fang is supported in part by by National Science Foundation under the grant DMS-1522184, DMS-1619778 and DMS-1107291: RNMS KI-Net.}

\begin{abstract}
The Generalized Langevin Equation, in history, arises as a natural fix for the rather traditional Langevin equation when the random force is no longer memoryless. It has been proved that with fractional Gaussian noise (fGn) mostly considered by biologists, the overdamped Generalized Langevin equation satisfying fluctuation dissipation theorem can be written as a fractional stochastic differential equation (FSDE). While the ergodicity is clear for linear forces \cite{liliulu2017}, it remains less transparent for nonlinear forces. In this work, we present both a direct and a fast algorithm respectively to this FSDE model. The strong orders of convergence are proved for both schemes, where the role of the memory effects can be clearly observed. We verify the convergence theorems using linear forces, and then present the ergodicity study of the double well potentials in both 1D and 2D setups.
\end{abstract}
\maketitle

\section{Introduction}

Diffusion in statistical mechanics is a class of ubiquitous phenomena that appears commonly in nature and has been extensively studied from both physical and mathematical points of view. While the normal diffusion, typically formalized by random walk, is well understood through the rather classical Brownian Motion theory and the Langevin equation (LE), the so-called \textit{anomalous diffusion} processes, however, remain far from fully explored. Among them, subdiffusion, in which the mean square displacement $\left< \Delta x(t)^2 \right>$ scales as $t^\beta$ with $0<\beta <1$, has been found in many different physical contexts such as cytoplasmic macromolecules in living cells \cite{sub_cyto1,sub_cyto2}, the movement of lipids and single-molecule on membranes \cite{sub_membrane1, sub_membrane2,sub_membrane3}, the solute transport in porous media \cite{sub_porous}, the translocation of polymer solutions \cite{sub_trans1,sub_trans2_also_full,sub_trans3}, and the conformational dynamics and fluctuations of protein molecules \cite{sub_protein1,sub_protein2}. To better describe subdiffusive phenomena, the Generalized Langevin Equation (GLE) with fractional Gaussian noise (fGn) is therefore introduced \cite{kouxie04} from a model point of view. This particular GLE can be written as
\begin{equation} \label{gle}
m \ddot x(t) = - \nabla V(x) - \int_{t_0}^t \gamma(t-s) \dot x(s) ds + \eta(t),
\end{equation}
where a particle with mass $m$ and position $x(t)$ is considered. Besides the external force $- \nabla V(x) = b(x)$, the particle is also driven by the friction (dissipation) and a random force (fluctuation), in which the friction term depending on the history of the particle velocity -- instead of the instantaneous velocity -- has a memory effect with the memory kernel $\gamma(t)$, and the random fluctuation force $\eta(t) = \sigma \dot{B}_H$ is the fractional Gaussian noise with the Hurst index $H$ and some fluctuation constant $\sigma$. Here, $\dot{B}_H$ is understood as the distributional time derivative of the fractional Brownian motion $B_H$. See section \ref{sec_fbm} for the brief introduction and we refer the readers to \cite{shevchenko2015,dhp00} for more details. This equation is sometimes also referred to as the fractional Langevin equation in physical literature \cite{fle1,fle2,fle3,fle4}.

The intuitive picture of the GLE (or LE) description, as is provided by Kubo \cite{kubo66}, is to consider for example the colloidal particles floating in a liquid medium. The random impacts of surrounding particles are responsible for two effects -- the random force and the systematic friction, and hence the two parts must be related. To be specific, the energy restored by fluctuation must be balanced with the energy loss by dissipation so that the particle achieves the equilibrium with the correct temperature. This internal relationship linking both parts of the microscopic forces is described in general as the fluctuation-dissipation theorem (FDT) \cite{cw51,kubo66,mprv08}. To be specific,  $\eta(t)$ and the kernel $\gamma(t)$ satisfy
\begin{gather}\label{eq:fdt}
\mathbb{E}(\eta(t)\eta(t+\tau))=kT\gamma(|\tau|),~\forall \tau\in\mathbb{R},
\end{gather}
where $\mathbb{E}$ means `ensemble average' in physics or expectation in mathematics.

We point out that the GLE is, of course, by no means introduced merely for the sake of subdiffusion. Proposed by Mori \cite{mori65} and Kubo \cite{kubo66} in the sixties, the GLE is a well-appreciated object enjoying a long history. From a model viewpoint, GLE appears naturally according to the fluctuation-dissipation theorem as a generalization of the Langevin equation when the random force in considerations is no longer memoryless; In terms of derivation, it can be derived from the Mori-Zwanzig formulation \cite{mori1965transport, zwanzig73} as a powerful tool for dimension reduction in many different forms depending on applications, such as molecular dynamics \cite{coarsegrain1_eric,coarsegrain2_karniadakis,coarsegrain3_karniadakis,coarsegrain4_dashixiong} and recently uncertainty quantification \cite{uq1}. For different random forces, their corresponding memory kernels are different as a consequence of the fluctuation dissipation theorem \eqref{eq:fdt}. In the subdiffusion model with fractional Gaussian noise, the reasonable memory kernel turns out to behave as a power law
$\gamma(t) \propto t^{-\alpha}$ with certain constant $\alpha$ and the friction term becomes the fractional derivative, as is shown both formally and rigorously that in the contexts of the absence of external force and with quadratic potential in the \textit{overdamped} regime ($m \ll 1$) \cite{liliulu2017,kouxie04,kubo66,felderhof78}, respectively. It worths pointing out that due to the complicated memory effects, a rigorous proof of the fluctuation dissipation theorem is by no means an easy task, if not impossible. Fortunately in above cases, the solution can be formally written down explicitly, which makes the proofs feasible. For general potentials, however, a validation analysis of the fluctuation dissipation theorem in Equation \eqref{gle} remains unclear, which is also proposed as an open problem in \cite{liliulu2017}. Therefore, trying to understand this problem numerically serves as one main motivation of this work.

Following the analytical work, we also restrict ourselves to the overdamped GLE ($m \rightarrow 0$) as in \cite{liliulu2017}. The corresponding fractional stochastic differential equation (FSDE) reads \cite{liliulu2017}
\begin{equation} \label{fle}
D_c^{\alpha}x = - \nabla V(x) + \sigma \dot{B}_H,
\end{equation}
which is also known as the overdamped fractional Langevin equation with fractional noise, where
\[
\alpha=2-2H,~~\sigma=\frac{\sqrt{2}}{\sqrt{\Gamma(2H+1)}},
\]
as a result of FDT \eqref{eq:fdt} (see section \ref{sec_fbm} for more details). The fractional derivative in the Caputo sense \cite{gm97,kst06} is given by
\begin{equation} \label{caputo}
D_c^{\alpha}x(t) = \frac{1}{\Gamma (1-\alpha)} \int_0^t \frac{\dot x (s)}{(t-s)^\alpha} ds.
\end{equation}

Though motivated by the understanding of overdamped GLE in this FSDE model, we point out that the numerical analysis of FSDE is by itself an interesting mathematical problem. Although there has been some numerical simulations of the FLE \cite{fle3,fle_num}, the numerical analysis remain untouched. The main difficulties of our problem are of two folds: (a) In terms of numerical analysis, different from the usual SDE where the correlation between increments of the standard Wiener process is simply absent, in FSDE the increments of fBm depends on the history, resulting in the analysis of strong convergence much harder comparing to usual case. (b) From a computational viewpoint, due to the memory effect, a straightforward discretization will be very memory-consuming, as one needs access to all history values at each time step. This becomes particularly troublesome when computing a number of sample paths. Providing a rigorous strong convergence analysis, our paper also features a fast algorithm that can be used for general potentials with good efficiency. The rest of the paper is organized as follows. In this work, we first propose a direct numerical scheme for the FSDE with general parameters
\begin{gather}\label{eq:parameters}
\alpha\in (0,1),~H\in \left(\frac{1}{2}, 1\right),~~\sigma>0,
\end{gather}
based on the integral formulation of the FSDE and prove its convergence result in the strong sense. The optimal strong order for the overdamped GLE in the case $b(\cdot)$ is linear has been obtained by careful estimates of correlation for increments of fBm. However, due to the memory effects, the computational cost is rather formidable. In particular, a speeding up of the solver becomes crucial in the considerations of the ergodicity, and moreover the stochastic nature of our problem where multiple sample paths need to be computed.
To tackle the task, a fast numerical algorithm is then introduced in Section \ref{sec:fast}. The idea is to make use of the sum-of-exponentials (SOE) approximation of the algebraic memory kernel, which can be understood intuitively here as Markovian approximations of the non-Markovian process. The convergence of the fast algorithm is proved by establishing a stability lemma (Lemma \ref{lmm:discretecp}) based on comparison principles. We point out that SOE approximation has been applied in various situations for improving the computational efficiency of convolution integrals, for instance, \cite{soe1,soe2,soe3,soe4,soe5,soe_jiwei1,soe_jiwei2,soe_jiwei3}. Finally, the algorithms are tested on a number of numerical examples, including the verification test of FDT for quadratic and double-well potentials of both 1D and 2D cases.

\section{FSDE and Fractional Brownian Motion}

\subsection{Preliminaries and notations}

In this paper, we will fix the probability space $(\Omega, \mathcal{F}, \mathbb{P})$.  $x_0$ is a random variable defined on this space, while $B_H$ is a fractional Brownian motion defined on this space (see section \ref{sec_fbm} for brief introduction). We will use the filtration $(\mathcal{G}_t)$ with
\[
\mathcal{G}_t=\cap_{s>t} \sigma\Big(B_H(\tau), 0\le\tau\le s, x_0\Big),~~\forall t\in [0, T).
\]
The notation $\mathbb{E}$ represents the expectation (integral) under probability measure $\mathbb{P}$.

The FSDE model \eqref{fle} with general parameters given in \eqref{eq:parameters} is rigorously defined through the following integral formulation
\begin{gather}\label{eq:fsde1}
x(t)=x_0+\frac{1}{\Gamma(\alpha)}\int_0^{t}(t-s)^{\alpha-1}b(x(s))\,ds+\frac{\sigma}{\Gamma(\alpha)}\int_0^t (t-s)^{\alpha-1}dB_H,
\end{gather}
where $b(x)=-\nabla V$. In the discussion below, we will consider a general force field $b(x)$ that is not necessarily conservative.
As is shown in \cite[Theorem 1]{liliulu2017}, if $b$ is Lipschitz, the FSDE has a unique strong solution $x(t)$, which is a stochastic process adapted to the filtration $\{\mathscr{G}_t\}$.

For the convenience of discussion, we introduce the norm of a random variable $v\in L^2(\Omega; \mathbb{P})$
\begin{gather}
\|v\|=\sqrt{\mathbb{E}|v|^2},
\end{gather}
together with the associated inner product
\begin{gather}
\langle u, v\rangle=\mathbb{E} uv.
\end{gather}
Occasionally, we will drop the measure $\mathbb{P}$ and use $L^2(\Omega)$ to mean the space of square integrable random variables.

We denote by $J_{\alpha}$ the fractional integral operator
\begin{gather}
J_{\alpha}f=\frac{1}{\Gamma(\alpha)}\int_0^t(t-s)^{\alpha-1}f(s)\,ds.
\end{gather}
Moreover, we note that the solution to the fractional ODE \cite{li2018generalized}
\begin{gather}
D_c^{\alpha}u=Lu, ~~u(0)=A,
\end{gather}
is given by
\begin{gather}
u(t)=AE_{\alpha}(L t^{\alpha}),
\end{gather}
where $E_{\alpha}(\cdot)$ is the Mittag-Leffler function defined by
\begin{gather}\label{eq:mlfunc}
E_{\alpha}(z)=\sum_{n=0}^{\infty}\frac{z^n}{\Gamma(1+n\alpha)}.
\end{gather}

In the following subsections, we shall briefly revisit the basics of fractional Brownian motion, and then prove some basic estimates for the FSDE that prepares us for the numerical analysis in later sections.

\subsection{Fractional Brownian Motion}  \label{sec_fbm}

The fractional Brownian motion $B_H$ (see \cite{mv68,nualart06} for more detailed discussions) with Hurst parameter $H\in(0,1)$ is a Gaussian process defined on some probability space $(\Omega, \mathcal{F}, P)$ such that $B_H(0)=0$, with mean zero and covariance
\begin{gather}\label{fb:cov}
\mathbb{E}(B_t^HB_s^H)=R_H(s, t)=\frac{1}{2}\left(s^{2H}+t^{2H}-|t-s|^{2H}\right).
\end{gather}
By definition, $B_H$ has stationary increments
which are normal distributions with
$\mathbb{E}((B_H(t)-B_H(s))^2)=(t-s)^{2H}$. By the Kolmogorov
continuity theorem, $B_H$ is H\"older continuous with order
$H-\epsilon$ for any $\epsilon\in (0, H)$. $B_H$ has finite
$1/H$-variation. Besides, it is self similar:
$B_H(t)\myeq a^{-H}B_H(at)$ where `$\myeq$' means they have the same distribution. It is non-Markovian except for $H=1/2$
when it is reduced to the Brownian motion (i.e., Wiener process).
With this definition and the fact that $(B_H(t+h)-B_H(t))/h$ converges in distribution (i.e. under the topology of the dual of $C_c^{\infty}(0,\infty)$) to $\dot{B}_H(t)$, we have
\begin{multline}
\mathbb{E}(\dot{B}_H(t)\dot{B}_H(\tau+t))=\lim_{h\to 0, h_1\to 0}\mathbb{E}\left(\frac{B_H(t+h_1)-B_H(t)}{h_1}\frac{B_H(t+\tau+h)-B_h(t+\tau)}{h}\right)\\
  =\lim_{h\to 0, h_1\to 0}
  \frac{1}{2hh_1}\bigl((\tau+h_1)^{2H}-(t+h-h_1)^{2H}-\tau^{2H}+(\tau-h)^{2H}\bigr)=H(2H-1) \tau^{2H-2}.
\end{multline}
This explains why the fractional noise leads to power law kernel by the \eqref{eq:fdt} and why we have the FSDE in the over-damped limit as mentioned in the introduction.

Also, for the convenience, we denote
\begin{gather}
G(t)=\frac{\sigma}{\Gamma(\alpha)}\int_0^t (t-s)^{\alpha-1}dB_H.
\end{gather}
 The process $G(t)$ is clearly a Gaussian process because $B_H(t)$ is a Gaussian process. In \cite{liliulu2017}, it has been shown that
 \begin{lemma}[\cite{liliulu2017}]\label{eq:processG}
The increments of $G$ satisfies
\begin{gather}
\mathbb{E}|G(t_2)-G(t_1)|^2\le C|t_2-t_1|^{2H+2\alpha-2}.
\end{gather}
Consequently, $G(t)$ is $H+\alpha-1-\epsilon$ H\"older continuous for any $\epsilon>0$. Moreover, if $\alpha=2-2H$
and $\sigma=\frac{\sqrt{2}}{\sqrt{\Gamma(2H+1)}}$, then
\begin{gather}
G(t)\myeq \beta_H B_{1-H}
\end{gather}
 is a fractional Brownian motion up to a factor. Here, $\beta_H$ is a constant given by
 \begin{gather}
 \beta_H=\frac{\sqrt{2}}{\sqrt{\Gamma(3-2H)}}.
 \end{gather}
 \end{lemma}

\subsection{Some Estimates of the FSDE}
In this section, we prove some basic estimates of the FSDE, which helps the understanding of the equation and prepares us for the numerical analysis in the later sections.

For the FSDE \eqref{eq:fsde1}, assume that
\begin{gather}\label{eq:blip}
|b(x)-b(y)|\le L|x-y|.
\end{gather}
In \cite{liliulu2017}, it has been shown that \eqref{eq:fsde1} has a unique continuous strong solution. Moreover, we have the following moment control:
\begin{lemma} \label{Ex^2}
If $\ \mathbb{E}(|x_0|^2+|b(x_0)|^2)<\infty$, then
\begin{gather}
\sup_{0\le t\le T}\|x(t)\|^2\le C(T).
\end{gather}
\end{lemma}
{}
\begin{proof}
Using \eqref{eq:blip}, the strong solution $x(t)$ satisfies
\[
\lvert x(t)-x_0 \rvert \le \frac{t^{\alpha}}{\Gamma(\alpha + 1)}|b(x_0)|+\frac{L}{\Gamma(\alpha)}\int_0^t(t-s)^{\alpha-1}|x-x_0|(s)\,ds
+|G(t)|.
\]
Taking square and using the elementary inequality
$(a+b+c)^2\le 3(a^2+b^2+c^2)$,
we have
\begin{gather}\label{eq:aux1}
\lvert x(t)-x_0\rvert^2\le 3 \left[ \frac{t^{2\alpha}}{\Gamma(\alpha + 1)^2} |b(x_0)|^2
+ \frac{L^2}{\Gamma(\alpha)^2} \left(\int_0^t(t-s)^{\alpha-1}
 |x-x_0|(s)\,ds\right)^2+\lvert G(t)\rvert^2 \right].
\end{gather}
For the second term, H\"older inequality yields
\begin{multline}\label{eq:aux2}
\left(\int_0^t(t-s)^{\alpha-1}
 |x-x_0|(s)\,ds\right)^2 \\
 \le \int_0^t(t-s)^{\alpha-1}\,ds\int_0^t(t-s)^{\alpha-1}|x-x_0|^2(s)\,ds
 =\frac{t^{\alpha}}{\alpha}\int_0^t(t-s)^{\alpha-1}|x-x_0|^2(s)\,ds.
\end{multline}
Further, by the result in \cite[Prop. 1]{liliulu2017}
\begin{gather}\label{eq:aux3}
\mathbb{E}|G(t)|^2\le C(T).
\end{gather}
Combining \eqref{eq:aux1}-\eqref{eq:aux3}, we have
\begin{gather}
\mathbb{E}|x(t)-x_0|^2\le
C_1(T)+C_2(T)\frac{1}{\Gamma(\alpha)}\int_0^t(t-s)^{\alpha-1}\mathbb{E}|x-x_0|^2(s)\,ds.
\end{gather}
Using the Gr\"onwall inequality in \cite[Prop. 5]{feng2017}, we have
\begin{gather}
\mathbb{E}|x(t)-x_0|^2\le C(T),\forall t\in [0, T].
\end{gather}
The claim therefore follows.
\end{proof}

Now we are ready to estimate the increments of the solution.
\begin{lemma}\label{lmm:processestimate}
There exists a constant $C(T)$ such that for all $\delta \in (0,1)$, we have
\begin{equation} \label{incre}
\|x(t+\delta)-x(t)\|^2\le C(T)\delta^{2H+2\alpha-2}
\end{equation}
for all $t\le T, t+\delta\le T$.
\end{lemma}

\begin{proof}
By \eqref{eq:fsde1}, we have
\begin{multline*}
x(t+\delta)-x(t)
=\frac{1}{\Gamma(\alpha)}\int_0^t\left[(t+\delta-s)^{\alpha-1}
-(t-s)^{\alpha-1}\right]b(x(s))\,ds \\
+\frac{1}{\Gamma(\alpha)}\int_t^{t+\delta}(t+\delta-s)^{\alpha-1}b(x(s))\,ds
+(G(t+\delta)-G(t))=:I_1+I_2+I_3.
\end{multline*}

Using again $(a+b+c)^2\le 3(a^2+b^2+c^2)$, we have
\[
\|x(t+\delta)-x(t)\|^2\le 3(\|I_1\|^2+\|I_2\|^2+\|I_3\|^2).
\]

Using the H\"older inequality,
\begin{multline*}
\mathbb{E}\left(\int_0^t[(t+\delta-s)^{\alpha-1}
-(t-s)^{\alpha-1}]b(x(s))\,ds\right)^2 \\
\le \int_0^t\left[(t-s)^{\alpha-1}-(t+\delta-s)^{\alpha-1}\right]\,ds\,\mathbb{E}
\int_0^t\left[(t-s)^{\alpha-1}-(t+\delta-s)^{\alpha-1}\right] |b(x(s))|^2\,ds
\end{multline*}
Since
\[
0\le \int_0^t(t-s)^{\alpha-1}-(t+\delta-s)^{\alpha-1}\,ds=\frac{1}{\alpha}(t^{\alpha}+\delta^{\alpha}-(t+\delta)^{\alpha})\le \frac{1}{\alpha}\delta^{\alpha},
\]
we have
\[
\|I_1\|^2 \le C(\alpha)\delta^{2\alpha}\sup_{0\le s\le T}\mathbb{E}|b(x(s))|^2.
\]
Similarly, one can apply H\"older inequality for $\mathbb{E}(\int_t^{t+\delta}(t+\delta-s)^{\alpha-1}b(x(s))\,ds)^2$. Finally, we have
\begin{gather} \label{incre_pf}
\lVert x(t+\delta)-x(t) \rVert^2
\le C\delta^{2\alpha}\sup_{0\le s\le T}\mathbb{E}\lvert b(x(s))\rvert^2+\mathbb{E} \lvert G(t+\delta)-G(t)\rvert^2
\le C(T) \left(\delta^{2\alpha}+\delta^{2H+2\alpha-2} \right),
\end{gather}
where we used Lemma \ref{eq:processG} and Lemma \ref{Ex^2}. Apparently $2 \alpha > 2H + 2 \alpha-2$, and the claim follows.
\end{proof}
\begin{remark}
When FDT is satisfied, $\alpha = 2 - 2H$, the order in the right hand side of Equation \eqref{incre_pf} becomes
\[
\|x(t+\delta)-x(t)\|^2\le C\delta^{2-2H}.
\]
If $H=\frac{1}{2}$ and $B_H=W$ is the standard Wiener process, this is a well-known result for diffusion processes.
\end{remark}

\section{Direct Discretization}

For the fixed terminal time $T$, we introduce the time step
\begin{gather}
k=\frac{T}{N},
\end{gather}
where $N$ is a positive integer and we define
\begin{gather}
t_j=jk.
\end{gather}
We will use the notation $\mathcal{C}$ to represent the complexity, or cost, of an algorithm.

We approximate $b(x(t))$ with a function $\tilde{b}(t)$ such that
\begin{gather}\label{eq:apprb}
\tilde{b}(t)=b(x_{j-1}),~t\in [t_{j-1}, t_j).
\end{gather}
This then gives a numerical scheme for the FSDE \eqref{eq:fsde1}:
\begin{multline}\label{eq:scheme}
x_n=x_0+\frac{1}{\Gamma(\alpha)}\sum_{j=1}^n b(x_{j-1})\int_{t_{j-1}}^{t_j}(t_n-s)^{\alpha-1}ds+G(t_n)\\
=x_0+\frac{k^{\alpha}}{\Gamma(1+\alpha)}\sum_{j=1}^n b(x_{j-1})((n-j+1)^{\alpha}-(n-j)^{\alpha})+G(t_n).
\end{multline}

Let $N_G$ be the complexity for sampling process $G(t)$. Then, a simple estimate gives the following claim regarding the numerical scheme \eqref{eq:scheme}:
\begin{proposition}\label{pro:directmethoderr}
The scheme \eqref{eq:scheme} has time complexity $\mathcal{C}=O(N^2+N_G)$ and it converges to the solution of the FSDE strongly in the following sense:
\begin{gather}
\sup_{n\le T/k}\|x_n-x(t_n)\|\le C(T)k^{H+\alpha-1}.
\end{gather}
\end{proposition}

\begin{proof}
To compute the the fractional integral, we need $O(N^2)$ operations.
Hence, the complexity is clearly
\[
\mathcal{C}=O(N^2+N_G).
\]

Using \eqref{eq:fsde1}  and \eqref{eq:scheme}, we have
\begin{multline}\label{eq:erraux}
x_n-x(t_n)=\frac{1}{\Gamma(\alpha)}\sum_{j=1}^n \int_{t_{j-1}}^{t_j}(t_n-s)^{\alpha-1}(b(x_{j-1})-b(x(s)))\,ds\\
=\frac{1}{\Gamma(\alpha)}\sum_{j=1}^n \int_{t_{j-1}}^{t_j}(t_n-s)^{\alpha-1}(b(x_{j-1})-b(x(t_{j-1})))\,ds
+\frac{1}{\Gamma(\alpha)}\sum_{j=1}^n \int_{t_{j-1}}^{t_j}(t_n-s)^{\alpha-1}(b(x(t_{j-1}))-b(x(s)))\,ds.
\end{multline}

Denote
\begin{gather}\label{eq:Rn}
R_n :=\frac{1}{\Gamma(\alpha)}\sum_{j=1}^n \int_{t_{j-1}}^{t_j}(t_n-s)^{\alpha-1}(b(x(t_{j-1}))-b(x(s)))\,ds.
\end{gather}
It follows from \eqref{eq:blip} and Lemma \ref{lmm:processestimate} that for all $n\le T/k$
\begin{gather}
\|R_n\|  \le C_1(T)k^{H+\alpha-1}.
\end{gather}

Hence, for any $n\le T/k$, we have
\[
\|x_n-x(t_n)\|\le \frac{L}{\Gamma(\alpha)}\sum_{j=1}^n\int_{t_{j-1}}^{t_j}(t_n-s)^{\alpha-1}ds \|x_{j-1}-x(t_{j-1})\|
+C_1(T)k^{H+\alpha-1}.
\]
Applying \cite[Lemma 6.1]{feng2017}, we find that
\[
\|x_n-x(t_n)\|\le u(t_n)\le C(T)k^{\alpha+H-1},
\]
where $u$ solves
\[
D_c^{\alpha}u=Lu,~~u(0)=C_1(T)k^{H+\alpha-1}.
\]
This then finishes the proof.
\end{proof}

The rate in Proposition \ref{pro:directmethoderr} is only optimal for multiplicative noise and we expect better bounds for the strong order since we have additive noise. As is well known, the Euler-Maruyama scheme for usual SDE has strong order $O(k)$ for additive noise, which can be proved using the fact that $W(t_2)-W(t_1)$ is independent of the sigma algebra $\sigma(W(s): s\le t_1)$. Unfortunately, for the fractional Brownian motion,  $B_H(t_2)-B_H(t_1)$ is not independent of the history. However, we note that the correlation decays and we may use this fact to improve the strong order. In fact, we are able to improve the result for the case $b(x)=Bx$ and $\alpha=2-2H$.
\begin{theorem}\label{thm:directmethoderr}
Let $\alpha=2-2H$ and $b(x)=Bx$ where $B$ is a $d\times d$ constant matrix. The scheme \eqref{eq:scheme} has time complexity $\mathcal{C}=O(N^2+N_G)$ and the strong error of the scheme can be controlled as
\begin{gather} \label{inthm:direct_order}
\sup_{n\le T/k}\|x_n-x(t_n)\|\le
\begin{cases}
C(T,H)k^{3-3H}, & H\in (3/4, 1) \\
C(T,H)\sqrt{|\ln k|} k^{3/4}, & H=3/4,\\
C(T,H) k^{3/2-H}, & H\in (1/2, 3/4).
\end{cases}
\end{gather}
\end{theorem}

We need the following to prove this theorem.
\begin{lemma}[\cite{skm93}, Theorem 3.1]\label{lmm:holderregu}
Suppose $f\in C^{0, \beta}([0, T]; B)$ for a Banach space $B$ and $0\le \beta\le 1$. Let $\alpha\in (0, 1)$ and
\[
u(t)=u_0+\frac{1}{\Gamma(\alpha)}\int_0^t(t-s)^{\alpha-1}f(s)\, ds.
\]
Then,
\begin{gather}
u(t)=u_0+\frac{f(0)}{\Gamma(1+\alpha)}t^{\alpha}+\psi(t),
\end{gather}
where
\begin{gather}
\psi \in
\begin{cases}
C^{0, \beta+\alpha}([0, T]; B), & \beta+\alpha<1,\\
C^{1, \beta+\alpha-1}([0, T]; B), & \beta+\alpha>1,\\
C^{0,1; 1}([0, T]; B), & \beta+\alpha=1.
\end{cases}
\end{gather}
\end{lemma}

Now, we are ready to prove Theorem \ref{thm:directmethoderr}:

\begin{proof}[Proof of Theorem \ref{thm:directmethoderr}]
To simplify the notation, we denote
\begin{equation}
R(n,k) :=
\begin{cases}
k^{6-6H} & H>\frac{3}{4}, \\
|\ln n| k^{6-6H} & H=\frac{3}{4},\\
 t_n^{3-4H}k^{3-2H} & H\in (\frac{1}{2}, \frac{3}{4}).
\end{cases}
\label{def_rnk}
\end{equation}
As in the proof of Proposition \ref{pro:directmethoderr}, we only need to estimate
\begin{gather}
R_n=\frac{1}{\Gamma(\alpha)}\sum_{j=1}^n \int_{t_{j-1}}^{t_j}(t_n-s)^{\alpha-1}(b(x(t_{j-1}))-b(x(s)))\,ds.
\end{gather}

Recall
\[
x(t)=\Big(x_0+\frac{1}{\Gamma(\alpha)}\int_0^t(t-s)^{\alpha-1}b(x(s))\,ds\Big)
+G(t)=: \zeta(t)+\beta_HB_{1-H}(t),
\]
where we have identified $G(t)$ with $\beta_HB_{1-H}$ since they
have the same distribution.

Since $b(x)=Bx$, we have
\begin{multline}
R_n=\frac{B}{\Gamma(\alpha)}\sum_{j=1}^n \int_{t_{j-1}}^{t_j}(t_n-s)^{\alpha-1}(\zeta(t_{j-1})-\zeta(s))\,ds \\
+\frac{B\beta_H}{\Gamma(\alpha)}\sum_{j=1}^n \int_{t_{j-1}}^{t_j}(t_n-s)^{\alpha-1}(B_{1-H}(t_{j-1})-B_{1-H}(s))\,ds
=:R_{n,1}+R_{n,2}.
\end{multline}

{\bf Step 1}

We first of all estimate $R_{n,1}$. We denote
\[
I_{n,1}^j:=\frac{B}{\Gamma(\alpha)}\int_{t_{j-1}}^{t_j}(t_n-s)^{\alpha-1}(\zeta(t_{j-1})-\zeta(s))\,ds,
\]
and
\begin{gather}\label{eq:Rnm}
R_{n,1}^m:=\sum_{j=1}^m I_{n,1}^j.
\end{gather}
Clearly,  $R_{n,1}=R_{n,1}^n$. By the definition of $R_{n,1}^m$, we easily get that
\[
\|R_{n,1}^m\|^2=2\sum_{j=1}^{m-1}\mathbb{E}(R_{n,1}^j I_{n,1}^{j+1})+\sum_{j=1}^m \|I_{n,1}^j\|^2,
\]
since $R_{n,1}^1=I_{n,1}^1$.

\begin{equation}
\|I_{n,1}^j\|\le C\int_{t_{j-1}}^{t_j}(t_n-s)^{\alpha-1}\|\zeta(s)-\zeta(t_{j-1})\|\,ds
\le Ck^{2-2H}\int_{t_{j-1}}^{t_j}(t_n-s)^{\alpha-1}\,ds,
\label{comment1}
\end{equation}
where the estimate of $\|\zeta(s)-\zeta(t_{j-1})\|$ is due to that $\zeta$ is $\alpha$-H\"older continuous by Lemma \ref{lmm:holderregu} (see also the estimates of $I_1, I_2$ in Lemma \ref{lmm:processestimate}).

Hence, we have
\begin{align*}
\sum_{j=1}^n \|I_{n,1}^j\|^2
& \le C k^{4(1-H)}\left(k^{2\alpha}+\sum_{j=1}^{n-1}(\int_{t_{j-1}}^{t_j}(t_n-s)^{\alpha-1}\,ds)^2 \right)
 \le Ck^{4(1-H)}(k^{2\alpha}+k^2\sum_{j=1}^{n-1}(t_n-t_j)^{2\alpha-2})\\
 & =Ck^{4(1-H)}(k^{2\alpha}+k^{2\alpha}\sum_{m=1}^{n-1}m^{2\alpha-2}) = Ck^{8-8H}(1+\sum_{m=1}^{n-1}m^{2-4H}).
\end{align*}
Depending on the behavior of $\displaystyle{\sum_{m=1}^{n-1}}m^{2-4H}$, the estimates branches into three cases. Clearly, when $4H-2>1$, i.e. $H>3/4$, the finite sum is bounded by a constant, while $H = 3/4$, it is bounded by $\ln n$. When $H < 3/4$,
$\sum_{m=1}^{n-1}m^{2-4H} \le \int_0^n x^{2-4H} \,dx = \frac{n^{3-4H}}{3-4H}$. Hence, it is easily bounded by $C R(n,k)$.

Now turning to the first term in $R_{n,1}$ which is present for $n\ge 2$. Before starting the estimate, we point out that $x \in C^{0,1-H}[0, T; L^2(\Omega)]$, so is $b(x(s))$. Hence one could apply Lemma \ref{lmm:holderregu} to $\zeta$ with $f(s) = b(x(s))$ and $B = L^2(\Omega)$, and get
$$\zeta(s) = x_0 + \frac{b(x(0))}{\Gamma(1+\alpha)} s^{2-2H}+ \psi(s),$$
where $\psi \in C^{0,3-3H}$ for $H>2/3$ or $C^{1, 2-3H}$ for $H<2/3$. It follows that
$\|\psi(s) - \psi(t_{j})\|\le Ck^{\min(1, 3-3H)},~H\neq 2/3$.
For $H=2/3$, there is $|\log k|$ factor but overall it is bounded by $k^{\min(3/2-H, 3-3H)}$.
Note that the same bound as in \eqref{comment1} is not enough for the desired results, so we must split $\zeta$ as $\psi$ and $s^{2-2H}$.
Now we are ready for the estimation of $2\sum_{j=1}^{m-1}\mathbb{E}(R_{n,1}^j I_{n,1}^{j+1})$. We first apply Cauchy-Schwarz Inequality inequality.
\begin{align*}
2\sum_{j=1}^{m-1}\mathbb{E}(R_{n,1}^j I_{n,1}^{j+1})
& \le C\sum_{j=1}^{m-1}\|R_{n,1}^j\|\int_{t_j}^{t_{j+1}}(t_n-s)^{\alpha-1}(k^{\min(3/2-H, 3-3H)}
+(s^{2-2H}-t_{j}^{2-2H})) \\
&\le C\sum_{j=1}^{m-1}\int_{t_j}^{t_{j+1}}(t_n-s)^{\alpha-1}
\|R_{n,1}^j\|^2\,ds
+C\sum_{j=1}^{m-1}\int_{t_j}^{t_{j+1}}(t_n-s)^{\alpha-1}k^{2\min(3/2-H, 3-3H)}\,ds\\
&+C\sum_{j=1}^{m-1}\int_{t_j}^{t_{j+1}}(t_n-s)^{\alpha-1}(s^{2-2H}-t_j^{2-2H})^2\,ds
=:J_1+J_2+J_3.
\end{align*}

For $J_1$, we simply estimate it as
\[
J_1\le C\int_0^{t_m}(t_n-s)^{\alpha-1}u_n(s)\,ds\le C\int_0^{t_m}(t_m-s)^{\alpha-1}u_n(s)\,ds
\]
where
\begin{gather}
\begin{split}
u_n(s)&:=\|R_{n,1}^j\|^2,~s\in [t_j, t_{j+1}),\\
R_{n,1}^0&=0.
\end{split}
\end{gather}

$J_2$ is easily bounded:
\[
J_2\le Ck^{2\min(3/2-H, 3-3H)}t_n^{\alpha}
\le C(T)R(n,k).
\]

For $J_3$, we first of all note
\[
(s^{2-2H}-t_j^{2-2H})^2\le Ck^2 t_j^{2-4H}.
\]
Hence,
\begin{gather*}
J_3\le Ck^2\sum_{j=1}^{m-1}\int_{t_j}^{t_{j+1}}(t_n-s)^{\alpha-1}t_j^{2-4H}\,ds
\le Ck^2\left(k^{2-4H}k^{2-2H}
+\sum_{j=2}^{n-1}\int_{t_j}^{t_{j+1}}(t_n-s)^{\alpha-1}t_j^{2-4H}\,ds\right).
\end{gather*}
For the last term, we do a simple change of variables $s-k\to s$, and obktain
\begin{gather*}
\sum_{j=2}^{n-1}\int_{t_j}^{t_{j+1}}(t_n-s)^{\alpha-1}t_j^{2-4H}\,ds
\le \sum_{j=2}^{n-1}\int_{t_{j-1}}^{t_j}(t_{n-1}-s)^{\alpha-1}t_j^{2-4H}ds
\le \int_k^{t_{n-1}}(t_{n-1}-s)^{\alpha-1}s^{2-4H}\,ds.
\end{gather*}
If $H< 3/4$, this integral is bounded by $t_{n-1}^{4-6H}$.
If $H=3/4$, it is bounded by $\ln(n) t_{n-1}^{4-6H}$.
If $H>3/4$, we have it bounded by
\[
\int_k^{t_{n-1}}(t_{n-1}-s)^{\alpha-1}s^{2-4H}\,ds
=t_{n-1}^{4-6H}\int_{1/(n-1)}^1(1-s)^{\alpha-1}s^{2-4H}ds
\le Ct_{n-1}^{4-6H}(\frac{1}{n-1})^{3-4H}
\le Ct_{n-1}^{1-2H}k^{3-4H}.
\]
Hence, we find that $J_3\le CR(n,k)$ still holds.

Overall, we have
\[
u_n(t_m)\le C\int_0^{t_m}(t_m-s)^{\alpha-1}u_n(s)\,ds
+CR(n,k).
\]

Applying \cite[Lemma 6.1]{feng2017}, we find
\[
\|R_{n,1}\|^2=u_n(t_n)\le C(T,H) R(n,k).
\]

{\bf Step 2}

We now estimate $R_{n,2}$. We similarly define
\[
I_{n,2}^j=\frac{B\beta_H}{\Gamma(\alpha)}
\int_{t_{j-1}}^{t_j}(t_n-s)^{\alpha-1}(B_{1-H}(t_{j-1})-B_{1-H}(s))\,ds.
\]
Then, it is clear that
\[
\|R_{n,2}\|^2=\left(\sum_{j}\|I_{n,2}^j\|^2
+2\sum_{i<j: j\le i+3}\mathbb{E} I_{n,2}^iI_{n,2}^j\right)
+2\sum_{i<j: j\ge i+4}\mathbb{E} I_{n,2}^iI_{n,2}^j=:K_1+K_2.
\]
Here we split the terms into two part so that the $i$ and $j$ in the second term are separated enough, which helps its estimate.
\[
\|I_{n,2}^j\|\le C\int_{t_{j-1}}^{t_j}(t_n-s)^{\alpha-1}\|B_{1-H}(t_{j-1})-B_{1-H}(s)\|\,ds
\le Ck^{1-H}\int_{t_{j-1}}^{t_j}(t_n-s)^{\alpha-1}\,ds,
\]
where the increments of the fBm are estimated as term $I_3$ of Lemma \ref{lmm:processestimate}.
Then a straightforward calculation shows,
\begin{align*}
\sum_{j=1}^n \|I_{n,2}^j\|^2
& \le C k^{2(1-H)}\left(k^{2\alpha}+\sum_{j=1}^{n-1}(\int_{t_{j-1}}^{t_j}(t_n-s)^{\alpha-1}\,ds)^2 \right)
 \le Ck^{2(1-H)}(k^{2\alpha}+k^2\sum_{j=1}^{n-1}(t_n-t_j)^{2\alpha-2})\\
 & =Ck^{2(1-H)}(k^{2\alpha}+k^{2\alpha}\sum_{m=1}^{n-1}m^{2\alpha-2}) = Ck^{6-6H}(1+\sum_{m=1}^{n-1}m^{2-4H}).
\end{align*}
Similar as the first part of Step 1, the estimates henceforth branch according to the behavior of $\displaystyle{\sum_{m=1}^{n-1}}m^{2-4H}$. Clearly, when $4H-2>1$, i.e. $H>3/4$, the finite sum is bounded by a constant, while $H = 3/4$, it is bounded by $\ln n$. When $H < 3/4$,
\[
\displaystyle{\sum_{m=1}^{n-1}}m^{2-4H} \le \int_0^n x^{2-4H} \,dx = \frac{n^{3-4H}}{3-4H}.
\]
Hence we then have
\begin{gather}\label{eq:rnk}
\sum_{j=1}^n \|I_{n,2}^j\|^2 \le
 C R(n,k).
\end{gather}

Clearly,
\[
K_1\le 7\sum_{j}\|I_{n,2}^j\|^2\le C R(n,k)
\]
by \eqref{eq:rnk}.

For $K_2$, which is present only if $n\ge 8$, we use \eqref{fb:cov} and have
\[
\mathbb{E} I_{n,2}^iI_{n,2}^j
=C\int_{t_{i-1}}^{t_{i}}\int_{t_{j-1}}^{t_j}(t_n-s)^{\alpha-1}(t_n-\tau)^{\alpha-1}(|s-t_{j-1}|^{2-2H}-|s-\tau|^{2-2H}
+|\tau-t_{i-1}|^{2-2H}-|t_{i-1}-t_{j-1}|^{2-2H})d\tau ds.
\]
where
\[
t_{i-1}\le s\le t_i<t_i+3k\le t_{j-1}\le \tau\le t_j.
\]
We first apply mean value theorem for $t\mapsto |t-t_{j-1}|^{2-2H}-|t-\tau|^{2-2H}$ so that there exists $\xi\in (t_{i-1}, s)$ such that
\[
|s-t_{j-1}|^{2-2H}-|s-\tau|^{2-2H}+|\tau-t_{i-1}|^{2-2H}-|t_{i-1}-t_{j-1}|^{2-2H}
=(2-2H)(s-t_{i-1})(|\xi-t_{j-1}|^{1-2H}-|\xi-\tau|^{1-2H})
\]
Then, we apply mean value theorem again so that there exists $\eta\in (t_{j-1}, \tau)$ and have the bound
\begin{multline*}
|s-t_{j-1}|^{2-2H}-|s-\tau|^{2-2H}+|\tau-t_{i-1}|^{2-2H}-|t_{i-1}-t_{j-1}|^{2-2H} \\
\le (2-2H)(2H-1)(s-t_{i-1})(\tau-t_{j-1})|\xi-\eta|^{-2H}
\le Ck^2|\tau-s-2k|^{-2H}.
\end{multline*}

We do change of variables $\tau-k\to \tau$ and $s+k\to s$ to find
\begin{multline*}
\mathbb{E} I_{n,2}^iI_{n,2}^j
\le Ck^2\int_{t_{i}}^{t_{i+1}}\int_{t_{j-2}}^{t_{j-1}}(t_n+k-s)^{\alpha-1}(t_n-k-\tau)^{\alpha-1}|s-\tau|^{-2H}d\tau ds \\
\le Ck^2 \int_{t_{i}}^{t_{i+1}}\int_{t_{j-2}}^{t_{j-1}}(t_{n-1}-s)^{\alpha-1}(t_{n-1}-\tau)^{\alpha-1}|s-\tau|^{-2H}d\tau ds.
\end{multline*}

It follows that
\begin{multline*}
K_2\le Ck^2\sum_{i\le j-2, j\le n-1}\int_{t_{i-1}}^{t_{i}}\int_{t_{j-1}}^{t_{j}}(t_{n-1}-s)^{\alpha-1}(t_{n-1}-\tau)^{\alpha-1}|s-\tau|^{-2H}d\tau ds \\
\le Ck^2\int_0^{t_{n-2}}\int_{s+k}^{t_{n-1}}(t_{n-1}-s)^{\alpha-1}(t_{n-1}-\tau)^{\alpha-1}|s-\tau|^{-2H}d\tau ds
=Ck^2\int_k^{t_{n-1}}\int_0^{s-k}s^{\alpha-1}\tau^{\alpha-1}|s-\tau|^{-2H}d\tau ds,
\end{multline*}
where in the last equality, we have made the change of variables $(t_{n-1}-s, t_{n-1}-\tau) \rightarrow (s,\tau)$.
Since
\[
\int_0^{s-k}\tau^{\alpha-1}|s-\tau|^{-2H}d\tau=
s^{2-4H}\int_0^{1-k/s}\tau^{\alpha-1}|1-\tau|^{-2H}d\tau
\le Cs^{2-4H}k^{1-2H}s^{2H-1}
= Ck^{1-2H}s^{1-2H},
\]
we thus have
\begin{multline*}
K_2\le Ck^2k^{1-2H}\int_k^{t_{n-1}}s^{2-4H}\,ds
=Ck^{3-2H}t_{n-1}^{3-4H}\int_{1/(n-1)}^1s^{2-4H}\,ds \\
\le Ck^{3-2H}t_{n-1}^{3-4H}(1+k^{3-4H}t_{n-1}^{4H-3})
\le Ck^{\min(3-2H, 6-6H)}.
\end{multline*}

Hence, we have
\[
\|R_n^n\|\le
\begin{cases}
Ck^{3-3H} & H>\frac{3}{4}, \\
C\sqrt{|\ln k|} k^{3-3H} & H=\frac{3}{4},\\
Ck^{3/2-H} & H\in (\frac{1}{2}, \frac{3}{4})
\end{cases}
=:\tilde{R}(T,k,H) .
\]

Using \eqref{eq:erraux}, we find
\[
\|x_n-x(t_n)\|\le \frac{L}{\Gamma(\alpha)}\sum_{j=1}^n\int_{t_{j-1}}^{t_j}(t_n-s)^{\alpha-1}ds \|x_{j-1}-x(t_{j-1})\|
+\tilde{R}(T,k,H).
\]
Applying \cite[Lemma 6.1]{feng2017} again, we obtain the desired error bound.
\end{proof}

By the proof, we find that proving the strong order for FSDE is much more difficult compared with the usual SDE (It\^o equations). The reason is that the increments of the fBm are not independent due to the memory. The key point we use is the fact that the correlation between the increments decay if the distance between them grows. In fact, we have an explicit
representation of the fractional Brownian motion, which is given by (see \cite{mv68})
\begin{gather} \label{bh_repre}
B_H(t)=C_1(H)\left(\int_0^t(t-s)^{H-1/2}dW(s)+\int_{-\infty}^0
((t-s)^{H-1/2}-(-s)^{H-1/2})dW(s)\right),
\end{gather}
where $C_1(H)$ is a constant. Define the filtration $\mathcal{A}(t)$ by
\begin{gather}
\mathcal{A}(t)=\cap_{s>t}\sigma(W(\tau), \tau \in (-\infty, s],x_0).
\end{gather} \label{lmm:conditionalbh}
It is then clear that under this representation
$\mathscr{G}_t\subset \mathcal{A}_t,~t\ge 0$
and $x(t)\in \mathcal{A}(t)$.

It is worth pointing out that we then have the following explicit formula regarding the decay of the correlation, though it is not directly used in this paper:
\begin{lemma}
Let $a\le b< c$ and $H\in (0, 1)$. We have
\begin{gather}
\sqrt{\mathbb{E}\Big(\mathbb{E}(B_H(c)-B_H(b) | \mathcal{A}(a))\Big)^2}
\le C(H)|H-\frac{1}{2}|((c-a)^H-(b-a)^H).
\end{gather}
\end{lemma}
\begin{proof}
Using the representation \eqref{bh_repre}, we find
\[
\Xi:=\mathbb{E}(B_H(c)-B_H(b) | \mathcal{A}(a))
=C_1(H)\int_{-\infty}^{a}((c-s)^{H-1/2}-(b-s)^{H-1/2})\,dW(s).
\]
The result then follows from the simple calculation below:
\begin{multline*}
\sqrt{\mathbb{E}(\Xi)^2}
=C_1(H)\sqrt{\int_{-\infty}^{a}((c-s)^{H-1/2}-(b-s)^{H-1/2})^2\,ds}
= C_1(H) \lvert H-\frac{1}{2}\rvert \sqrt{\int_{-\infty}^{a}  \left( \int_b^c (r - s)^{H-3/2} \,dr\right)^2 \,ds } \\
\le C_1(H)|H-\frac{1}{2}|\int_{b}^{c}\|(r-\cdot)^{H-3/2}\|_{L^2(-\infty, a)}\,dr
=\frac{C_1(H)|H-\frac{1}{2}|}{\sqrt{2-2H}}\int_{b}^{c}(r-a)^{H-1}\,dr.
\end{multline*}
\end{proof}

For general $b(x)$, we expect that the strong order can also be improved compared with Proposition \ref{pro:directmethoderr} by making use of the decay of correlation. However, this seems difficult and we leave this for future.

\section{A fast scheme} \label{sec:fast}

The cost of the direct discretization will be large if we compute many sample paths.  In this section, we propose a fast scheme. The idea is to use sum-of-exponentials (SOE) approximation in \cite{soe_jiwei1}.

\subsection{The fast scheme}
We first introduce the SOE approximation:
\begin{lemma}\label{lmm:soe}
For $\alpha\in (0, 1)$, tolerance $\epsilon>0$,  truncation $\delta>0$,and fixed $T>0$, there exist positive numbers $s_i, \omega_i$, with $1\le i\le M$ such that
\[
|t^{\alpha-1}-\sum_{i=1}^{M}\omega_i e^{-s_i t}|\le \epsilon,~\forall t\in [\delta, T],
\]
where
\[
M=O\left(\log\frac{1}{\epsilon}(\log\log\frac{1}{\epsilon}
+\log\frac{T}{\delta})
+\log\frac{1}{\delta}(\log\log\frac{1}{\epsilon}
+\log\frac{1}{\delta})\right).
\]
\end{lemma}

If $\epsilon=k^{\sigma}$ and  $\delta=k$, we have
\[
M=O((\log N)^2).
\]

We then break the convolution kernel
\[
t^{\alpha-1}=t^{\alpha-1}\chi(t\le k)+t^{\alpha-1}\chi(t>k).
\]
Then, for $t>k$, we apply the SOE approximation \eqref{lmm:soe}. Hence, we are then led to the kernel
\begin{gather}\label{eq:soekernel}
\gamma(t)=
\begin{cases}
\frac{1}{\Gamma(\alpha)}t^{\alpha-1} & t\le k,\\
\frac{1}{\Gamma(\alpha)}\sum_{i=1}^M \omega_i
e^{-s_i t} & t>k.
\end{cases}
\end{gather}
This kernel is discontinuous, integrable and nonnegative.

Then, approximation \eqref{eq:apprb} gives the following scheme:
\begin{gather}\label{eq:newscheme}
x_n=x_0+\sum_{j=1}^n b(x_{j-1})\int_{t_{j-1}}^{t_j}\gamma(t_n-s)\,ds+G(t_n).
\end{gather}
We denote
\begin{gather}
\eta_i^n=
\begin{cases}
0 ,& n=1,\\
\frac{1}{\Gamma(\alpha)}\sum_{j=1}^{n-1}b(x_{j-1})\int_{t_{j-1}}^{t_{j}}e^{-s_i(t_n-s)}\,ds ,& n\ge 2.
\end{cases}
\end{gather}
Then, we obtain our fast scheme:
\begin{gather}\label{eq:fastscheme}
x_n=x_0+\frac{k^{\alpha}}{\Gamma(1+\alpha)}b(x_{n-1})
+\sum_{i=1}^M\omega_i \eta_i^n+G(t_n).
\end{gather}

To make the efficiency of the solver more transparent, we note:
\begin{lemma}\label{lmm:etaseq}
For any $1\le i\le M$, the sequence $\{\eta_i^n\}$ satisfies
\begin{gather}
\eta_i^{n+1}
=e^{-s_i k}\eta_i^n+\frac{1}{s_i\Gamma(\alpha)}(e^{-s_ik}-e^{-2s_ik})b(x_{n-1}).
\end{gather}
\end{lemma}
\begin{proof}
This is just a consequence of direct computation.
\end{proof}
 In fact, Lemma \ref{lmm:etaseq} is a consequence of the well-known fact that dynamics with exponentially decaying memory kernels can be made Markovian. With the fast algorithm, we only need
 \[
 O(NM)=O(N (\log N)^2)
 \]
  time to compute $\{\eta_i^n\}$ for a sample path. To make the computational efficiency more transparent, we list some intuitive comparisons of the size of $N$ and $\log(N)$: for example, when $N = 1000$, $(\log N)^2 \approx 47$; when $N = 10000$, $(\log N)^2 \approx 85$. In fact, $M$ used here is a number even smaller than $(\log N)^2$, which is illustrated more specifically in later numerical sections \ref{sec_strongfast}.

\subsection{Stability and convergence of the fast algorithm}

We will use the following fractional ODE as a reference:
\begin{gather}\label{eq:v}
D_c^{\alpha}v=(1+\frac{\epsilon}{\Gamma(\alpha)} T^{1-\alpha})f(v(t)),~v(0)=y_0.
\end{gather}
The solution exists on $[0, T_b)$ where either $T_b=\infty$ or $\lim_{t\to T_b^-}v(t)=+\infty$ by the result in \cite{feng2017}.

We first of all investigate a Volterra type integral equation that is useful for the stability of our scheme.
\begin{lemma}\label{lmm:wellposedmodifiedeq}
Let $y_0\ge 0$ and $f(\cdot)$ is a non-negative locally Lipschitz function. Consider
\begin{gather}\label{eq:modifiedInt}
y(t)=y_0+\int_0^t \gamma(t-s)f(y(s))\,ds.
\end{gather}
Then, there is a unique continuous solution $y(t)$ on $[0, T]\cap [0, T_b)$ with $T_b$ being the blowup time for \eqref{eq:v}. On $[0, T]\cap [0, T_b)$, $y(t)$ is a non-decreasing continuous function, and satisfies
\[
y(t)\le v(t)
\]
where $v(\cdot)$ solves \eqref{eq:v}.
If $f(\cdot)$ is globally Lipschtiz, then $T_b=\infty$ and $y(t)$ exists on $[0, T]$.
\end{lemma}

The proof of this lemma is standard. For the readers' convenience, we give one proof in Appendix \ref{app:modifiedeq}.

The following lemma gives the stability of the fast algorithm:
\begin{lemma}\label{lmm:discretecp}
Let $y_0\ge 0$. Assume $f(v)$ is a non-negative, non-decreasing, Lipschitz function on $[0,\infty)$. Let $y(t)$ be the solution to \eqref{eq:modifiedInt} on $[0, T]$.
For a given sequence $z=\{z_m\}$, define
\begin{gather}
T_n(z,y_0)=y_0+\sum_{j=1}^nf(z_{j-1})\int_{t_{j-1}}^{t_j}\gamma(t_n-s)\,ds.
\end{gather}

\begin{enumerate}
\item Assume that $a=\{a_n\}$ solves the induction relation $a_n=T_n(a, y_0)$.
 Then, $\{a_n\}$ is non-negative, non-decreasing and $a_n\le y(t_n)$.
In particular, if $f(v)=Lv$, we have
\[
a_n\le y_0 E_{\alpha}\left(L(1+\frac{\epsilon}{\Gamma(\alpha)})t_n^{\alpha}\right)\le C(T,H,\alpha)y_0,~\forall \epsilon<1,
\]
where $E_{\alpha}(\cdot)$ is the Mittag-Leffler function defined in \eqref{eq:mlfunc}.

\item If a non-negative sequence $c=\{c_n\}$ satisfies $c_n\le T_n(c, y_0)$, then
\[
c_n\le a_n\le y(t_n).
\]
\end{enumerate}

\end{lemma}
\begin{proof}

For (1),  the claims follow by induction.

Indeed, $a_0=y_0=y(0)\ge 0$. Then,
\[
0\le a_1=y_0+\frac{f(y_0)}{\Gamma(\alpha)}\int_0^k(k-s)^{\alpha-1}\,ds
\le y_0+\frac{1}{\Gamma(\alpha)}\int_0^k(k-s)^{\alpha-1}f(y(s))\,ds=y(t_1)
\]
by the monotonicity of $f$ and $y$. Assume that for $m\le n-1, n\ge 2$, we have $0\le a_m\le y(t_m)$.  By the monotonicity of $f$ and $y$, $a_n=T_n(a, y_0)\ge 0$ is trivial. Moreover,
\begin{gather*}
a_n =T_n(a, y_0)\le y_0+\sum_{j=1}^nf(a_{j-1})\int_{t_{j-1}}^{t_j}\gamma(t_n-s)\,ds\le y_0+\sum_{j=1}^n\int_{t_{j-1}}^{t_j}\gamma(t_n-s)f(y(s))\,ds=y(t_n).
\end{gather*}
Now, we show the monotonicity of $\{a_n\}$. $a_1\ge a_0=y_0$ is clear. For $n\ge 2$, we have:
\begin{gather}\label{eq:tmp1}
\frac{f(a_{n-1})}{\Gamma(\alpha)} \int_{t_{n-1}}^{t_n}(t_n-s)^{\alpha-1}\,ds
 \ge \frac{f(a_{n-2})}{\Gamma(\alpha)}\int_{t_{n-1}}^{t_n}(t_n-s)^{\alpha-1}\,ds=\frac{f(a_{n-2})}{\Gamma(\alpha)}\int_{t_{n-2}}^{t_{n-1}}(t_{n-1}-s)^{\alpha-1}\,ds.
\end{gather}
Similarly,
\begin{align}\label{eq:tmp2}
f(a_{j-1})\int_{t_{j-1}}^{t_{j}}\gamma(t_n-s)\,ds
 \ge
 \begin{cases}
 f(a_{j-2})\int_{t_{j-2}}^{t_{j-1}}\gamma(t_n-s)\,ds,& j\ge 2,\\
 0,& j=1.
 \end{cases}
\end{align}
Relations \eqref{eq:tmp1}-\eqref{eq:tmp2} then give $a_n\ge a_{n-1}$.

If $f(v)=Lv$, we have by Lemma \ref{lmm:wellposedmodifiedeq} that
\[
a_n\le y(t_n)\le v(t_n)=y_0 E_{\alpha}(L(1+\frac{\epsilon}{\Gamma(\alpha)})t_n^{\alpha}).
\]

Claim (2) is straightforward by induction. To see this, we note that if we have $c_m\le a_m$ for all $m\le n-1$, then
\[
T_n(c, y_0)\le T_n(a, y_0),
\]
which is just $c_n\le a_n$.
\end{proof}

Now, we have the convergence of the fast algorithm:
\begin{theorem} \label{thm:fastmethoderr}
 Consider the fast scheme \eqref{eq:fastscheme}.
\begin{enumerate}
\item Assume the SOE approximation as in Lemma \ref{lmm:soe} is applied with tolerance $\eps>0$, then the strong error satisfies
$\|x_n-x(t_n)\|\le C(T)( k^{H+\alpha-1}+\epsilon)$.
In particular, if we set the tolerance $\epsilon=k^{H+\alpha-1}$, then sampling a path needs time complexity
$\mathcal{C}=O(N(\log N)^2+N_G)$,
 and we have the following strong error bound
\begin{gather}
\sup_{n\le T/k}\|x_n-x(t_n)\|\le C(T) k^{H+\alpha-1}.
\end{gather}

\item In the case $\alpha=2-2H$ and $b(x)=Bx$, if we choose, $\epsilon=k^{\min(3/2-H, 3-3H)}$, then the time complexity for sampling a path is $\mathcal{C}=O(N(\log N)^2+N_G)$ and the strong error is controlled as
\begin{gather}
\sup_{n\le T/k}\|x_n-x(t_n)\|\le
\begin{cases}
C(T,H)k^{3-3H}, & H\in (3/4, 1) \\
C(T,H)\sqrt{|\ln k|} k^{3/4}, & H=3/4,\\
C(T,H) k^{3/2-H}, & H\in (1/2, 3/4).
\end{cases}
\end{gather}
\end{enumerate}
\end{theorem}

\begin{proof}
The complexity part is easy and we omit. We now focus on the error estimates.

Let
\begin{gather}
r(t):=x(t)-\left(x_0+\int_0^t\gamma(t-s)b(x(s))\,ds + G(t)\right).
\end{gather}
Then, we have $r(t)=0$ for $t\in [0, k]$ and $|r(t)|\le \epsilon\int_0^t |b(x(s))|\,ds,~t\in [k, T]$ by SOE approximation.
It follows that
\[
\sup_{t\in [0, T]}\|r(t)\|\le C_1(T)\epsilon
\]
It then follows from the definition of $r(t)$ that
\begin{multline*}
x(t_n)-x_n
=\sum_{j=1}^n \int_{t_{j-1}}^{t_j}\gamma(t_n-s)(b(x(s))-b(x_{j-1}))\,ds+r(t_n)\\
=\sum_{j=1}^n \int_{t_{j-1}}^{t_j}\gamma(t_n-s)(b(x(t_{j-1}))-b(x_{j-1}))\,ds
- R_n +r_1(t_n),
\end{multline*}
where $R_n$ is defined as in \eqref{eq:Rn} and
\begin{gather}
r_1(t_n)=r(t_n)+\sum_{j=1}^n \int_{t_{j-1}}^{t_j}\left(\gamma(t_n-s)-\frac{1}{\Gamma(\alpha)}(t_n-s)^{\alpha-1}\right)\left(b(x(s))-b(x(t_{j-1}))\right)\,ds.
\end{gather}
By the SOE approximation, we again have
\[
\sup_{n\le T/k}\|r_1(t_n)\|\le C(T)\epsilon.
\]

We define $E_n=\|x(t_n)-x_n\|$. We then have
\begin{gather}
E_n\le \sum_{j=1}^n LE_{j-1}\int_{t_{j-1}}^{t_j}\gamma(t_n-s)\,ds
+\sup_{n\le T/k}\|R_n\|+\sup_{t\in [0, T]}\|r_1(t)\|.
\end{gather}

Applying Lemma \ref{lmm:discretecp} for
\[
y_0=\sup_{n\le T/k}\|R_n\|+\sup_{t\in [0, T]}\|r_1(t)\|
\]
 and $f(v)=Lv$ we have for $\epsilon<1$,
 \[
 \sup_{n\le T/k}E_n\le C(T,H,\alpha)(\sup_{n\le T/k}\|R_n\|+\epsilon).
 \]
 The result follows by the estimates of $R_n$ in the proof of Proposition \ref{pro:directmethoderr} and Theorem \ref{thm:directmethoderr}.
\end{proof}

\section{sampling fractional Brownian motion and process $G$}

To give a complete description of the numerical scheme, we must understand how to sample fractional Brownian motion and the process $G(t)$.

Since fractional Brownian motion is a Gaussian process. Using the covariant matrices, one can transform the standard Brownian motion into the desired Gaussian process. {}

For fractional Brownian motions, making use the time homogenuity, one has a fast algorithm to sample fractional Brownian motion. We use the circulant method (or Wood-Chan algorithm) (see \cite{shevchenko2015} for example).
The idea is to sample
\[
\xi_n=B_H(t_n)-B_H(t_{n-1}).
\]
By the self-similarity,
\[
\xi_n\myeq k^H(B_H(n)-B_H(n-1)).
\]
Letting $\zeta_n=B_H(n)-B_H(n-1)$, $(\zeta_n)$ forms a Gaussian sequence whose covariance matrix satisfies
\begin{gather}
\Sigma_{ij}=Cov(\zeta_i,\zeta_j)=\rho_H(|i-j|)
\end{gather}
for some function $\rho_H$. This structure allows us to
construct a circulant matrix $M$ of size $2(N-1)\times 2(N-1)$ such that
$\Sigma=M(1:N, 1:N)$.
Then, one is able to take the square root of $M$ using fast Fourier transform (FFT). With the square root of $M$, it is straightforward to transform the standard multivariable normal variables to a Gaussian sequence with covariant matrix $M$. The complexity is $O(N\log N)$. The first $N$ elements will be a sample for the sequence $(\zeta_n)$.
For the details, one can refer to \cite[section 6]{shevchenko2015}.

As stated in \cite{liliulu2017}, the case $\alpha=2-2H$ is the physical case. As mentioned above, if $\alpha=2-2H$,
\[
G(t)\myeq \beta_H B_{1-H}
\]
 is a fractional Brownian motion up to a factor, so $G(t)$ can be sampled directly using the method here. This is a key observation for the simulation of overdamped GLE with fractional noise. Hence we conclude
 \begin{theorem}
For the overdamped GLE with fractional noise (i.e. $\alpha=2-2H$, $\sigma=\sqrt{2}/\Gamma(2H+1)$), we can sample $G(t)$ with complexity
 \begin{gather}
 N_G=O(N\log N).
 \end{gather}
 Consequently, the total complexity of the direct scheme \eqref{eq:scheme} is
 \begin{gather}
 \mathcal{C}=O(N^2),
 \end{gather}
 and the total complexity of our fast algorithm \eqref{eq:fastscheme} is
 \begin{gather}
 \mathcal{C}=O(N(\log N)^2)
 \end{gather}
 \end{theorem}

For general $\alpha$, the FSDE model is not physical, but may be used in other situations. The covariance matrix of $G(t)$ has been given in \cite{liliulu2017}, for which the above trick fails, so that we do not have fast algorithms for sampling $G(t)$.
Another option is to consider the following equivalent form of the FSDE:
\begin{gather}\label{eq:integral}
J_{1-\alpha}(x-x_0)(t)=\frac{1}{\Gamma(1-\alpha)}\int_0^t(t-s)^{-\alpha}(x(s)-x_0)\,ds=\int_0^tb(x(s))\,ds+\sigma B_H(t)
\end{gather}
This is like integrate the differential form formally. It is known that a discretization of $J_{1-\alpha}(x-x_0)(t_n)-J_{1-\alpha}(x-x_0)(t_{n-1})$ is the $L^1$ scheme \cite{linxu2007,nochetto2016}. Hence, a possible numerical scheme is
\begin{gather}
(\mathcal{D}^{\alpha}x)_n=b(x_{n-1})+\frac{\sigma}{k} \xi_n,
\end{gather}
where $\mathcal{D}^{\alpha}$ refers to the $L^1$ scheme in \cite{linxu2007}. This numerical scheme is like a Euler scheme for the differential form \eqref{fle}. For this scheme, though we can sample $\xi_n$ fast, we do not have a fast algorithm for the $L^1$ scheme. Moreover, proving the convergence of this scheme is challenging.
Hence, developing a fast algorithm for the general $\alpha$ case is left for future.

\section{Numerical study of the overdamped GLE}

As we have mentioned,  the overdamped GLE with fractional noise is equation \eqref{eq:fsde1} with
\[
\alpha=2-2H,~~\sigma=\frac{\sqrt{2}}{\sqrt{\Gamma(2H+1)}}.
\]
We aim to study the ergodicity of the overdamped GLE
\[
D^{2-2H}x=-\nabla V(x)+\sqrt{\frac{2}{\Gamma(2H+1)}}\dot{B}_H.
\]

\subsection{Example 1 (harmonic potential)}

In this subsection, we aim to validate the order of strong convergence and study numerically the weak convergence of the schemes.
The example we use is in the 1D case and
\[
\nabla V(x)=x.
\]
In \cite{liliulu2017}, the formula of the exact solution in terms of $G(t)$ is given:
\begin{gather}
x(t)=x_0e_{\alpha,1}(t)+G(t)+\int_0^tG(t-s)\dot{e}_{\alpha,1}(s)\,ds.
\end{gather}
Given a sample of $G(t)$, the integral here can be evaluated numerically with small time steps:
\[
\int_0^tG(t-s)\dot{e}_{\alpha,1}(s)\,ds \approx \sum_i G(t-t_{i})(e_{\alpha,1}(t_{i})-e_{\alpha,1}(t_{i-1}))
\]
as the reference solution. For strong convergence, we must use the same sample of $G$, hence, we obtain $\xi_n^{m}$ for the small step $k_m$. Then, for $k_{m-1}=2k_m$,
\[
\xi_n^{m-1}=\xi_{2n-1}^m+\xi_{2n}^m.
\]
($k_m$ should be small so that the error from numerical integral is much smaller than the error from the scheme.)
\subsubsection{The strong convergence of the direct solver} \label{sec_ex1_strong} We first test the strong convergence of the direct solver.
The time step for the reference solution is chosen as $k_m = 2^{-14}$, and $x_0 =1$. The strong order of convergence is tested for various $H = 0.5, 0.55, 0.6, 0.65, 0.7, 0.73, 0.75, 0.8, 0.84$ over 20000 sample paths. Note that $H= 0.5$ is the memoryless case with the normal Brownian motion. In Fig. \ref{ex1_fig1}, we first plot the strong error of the solution
\[
\mathbb{E} \lvert x_n - x(t_n) \rvert^2
\]
in terms of different time steps for some values of $H$ to get a sense of convergence. Here cases of $H = 0.8, 0.6$ are plotted, respectively: (1) in the case of $H = 0.8$, the strong order reads $\min \{\frac{3}{2}-H, 3-3H\} = 0.6$. As can be seen in Fig. \ref{ex1_fig1}, the slope is approximately 1.1718, and hence the convergence order is 0.586. (2) similarly in the case $H = 0.6$, the convergence rate numerically reads 0.945, which approximately matches the analytical results $\min \{\frac{3}{2}-H, 3-3H\} = 0.9$. Then Fig. \ref{ex1_fig2} shows the plot of convergence rate from our numerical results in terms of $H$, which matches the order proved in Theorem \ref{thm:directmethoderr}.

\begin{figure}[htbp]
\centering
\subfloat[$H = 0.8$]{\includegraphics[width=0.5\textwidth]{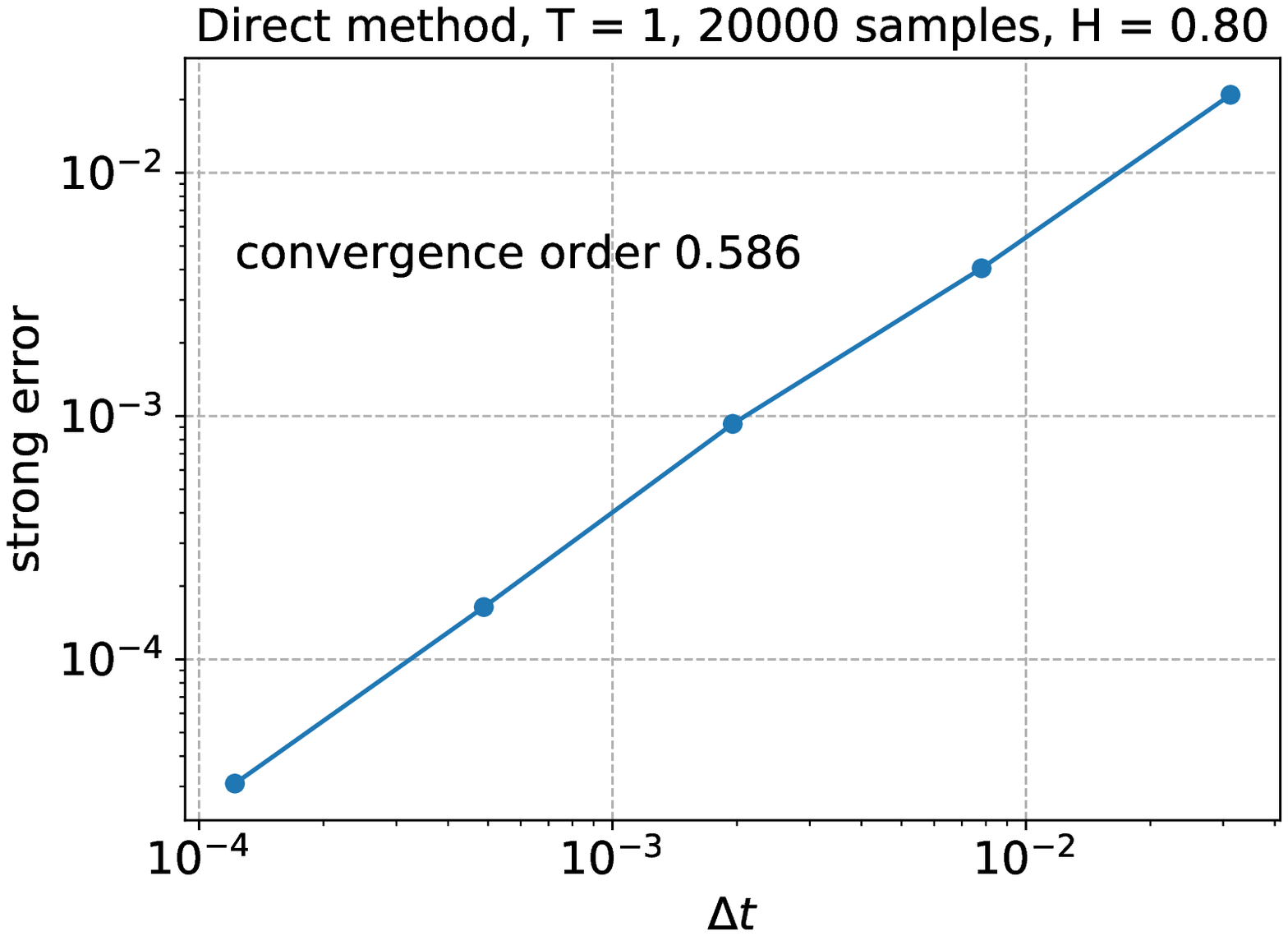} }
\subfloat[$H = 0.6 $]{\includegraphics[width=0.5\textwidth]{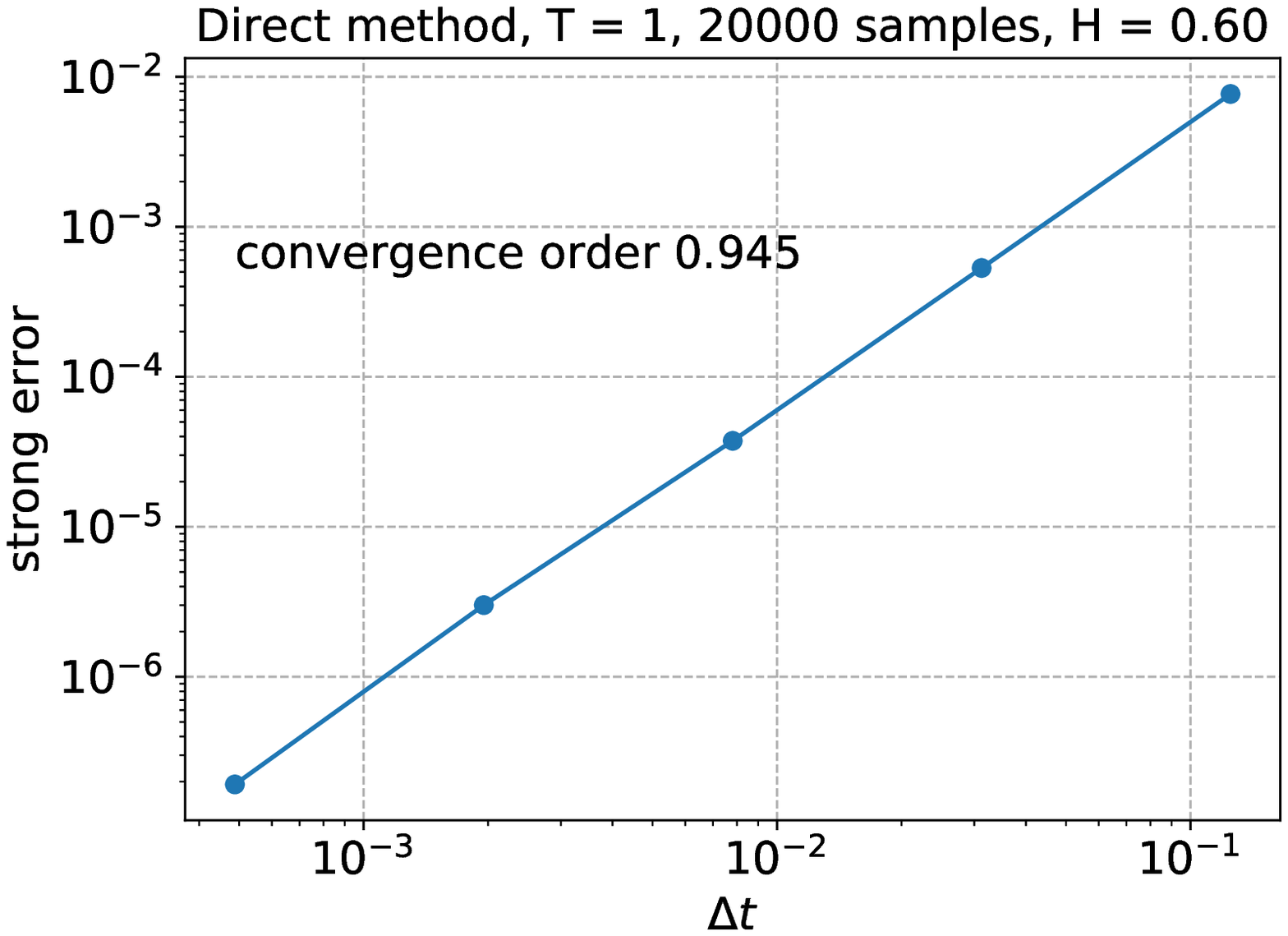} }
\caption{Example 1 (Section \ref{sec_ex1_strong}): The strong convergence of the direct method for $H = 0.8$ and $H = 0.6$, respectively in terms of various $\Delta t$ in log-log scale. The convergence orders match the theoretical result $\min \{\frac{3}{2}-H, 3-3H\}$, as is proved in Theorem \ref{thm:directmethoderr}.}
\label{ex1_fig1}
\end{figure}

\begin{figure}[htbp]
\centering
\includegraphics[width=0.7\textwidth]{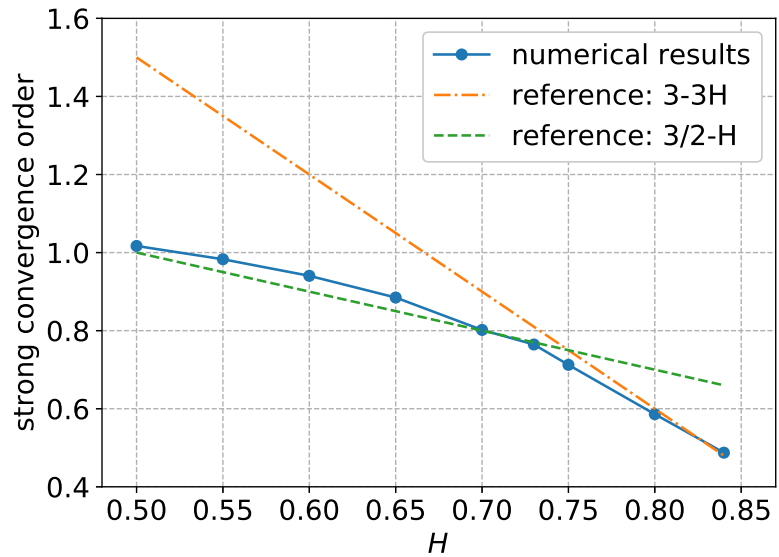}

\caption{Example 1 (Section \ref{sec_ex1_strong}): The strong convergence of the direct method for $H = 0.5, 0.55, 0.6, 0.65, 0.7, 0.73, 0.75, 0.8, 0.84$ computed over $20000$ sample paths in terms of $\Delta t = 2^{-11}, 2^{-9}, 2^{-7}, 2^{-5}, 2^{-3}$. As can be seen in the figure, the convergence orders match the reference as is proved in Theorem \ref{thm:directmethoderr}.}
\label{ex1_fig2}
\end{figure}

\subsubsection{The strong convergence of the fast solver} \label{sec_strongfast}
For the strong convergence of the fast solver. Similarly as previous section, the time step for the reference solution is chosen as $k_m = 2^{-14}$, and $x_0 =1$. We first plot the strong order of convergence for $H = 0.8$ and $H = 0.6$ in Fig \ref{ex1_fig3}, where both cases match the order $\min \{\frac{3}{2}-H, 3-3H\}$ as proved in Theorem \ref{thm:fastmethoderr}. Next, the strong convergence order for various $H = 0.55, 0.6, 0.65, 0.7, 0.73, 0.75, 0.8, 0.84$ is plotted in Fig \ref{ex1_fig4}. The numerical orders agree with theoretical results.

Note that for the computational cost of the fast solver, as is shown in previous discussion is $O(NM)$, instead of $O(N^2)$, where $M$ is the number of terms used in the SOE approximation. For this example with $H = 0.8$, $\Delta t = 2^{-11}, T = 1$ and the tolerance $\epsilon = 10^{-9}$, we only need $M = 36$ while $N = 2048$. To compute longer time behavior as is interested here, say, $T = 128$, for $H = 0.8$ using $\Delta t = 2^{-9} $, one only needs $M = 44$ while $N = 65536$. This drastically reduced the computational cost especially with many sample paths to be simulated.
\begin{figure}[htbp]
\centering
\subfloat[$H = 0.8$]{\includegraphics[width=0.5\textwidth]{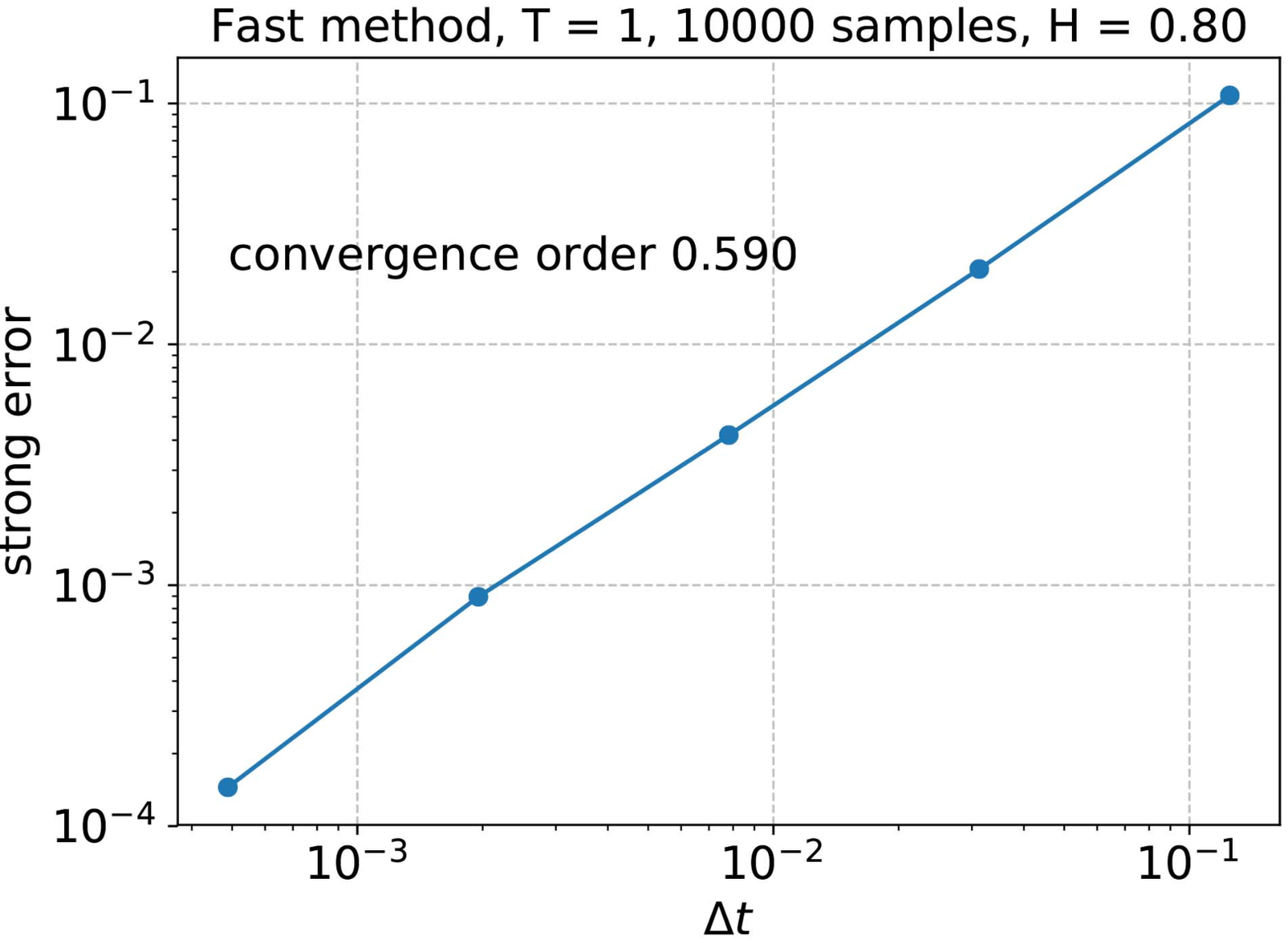} }
\subfloat[$H = 0.6 $]{\includegraphics[width=0.5\textwidth]{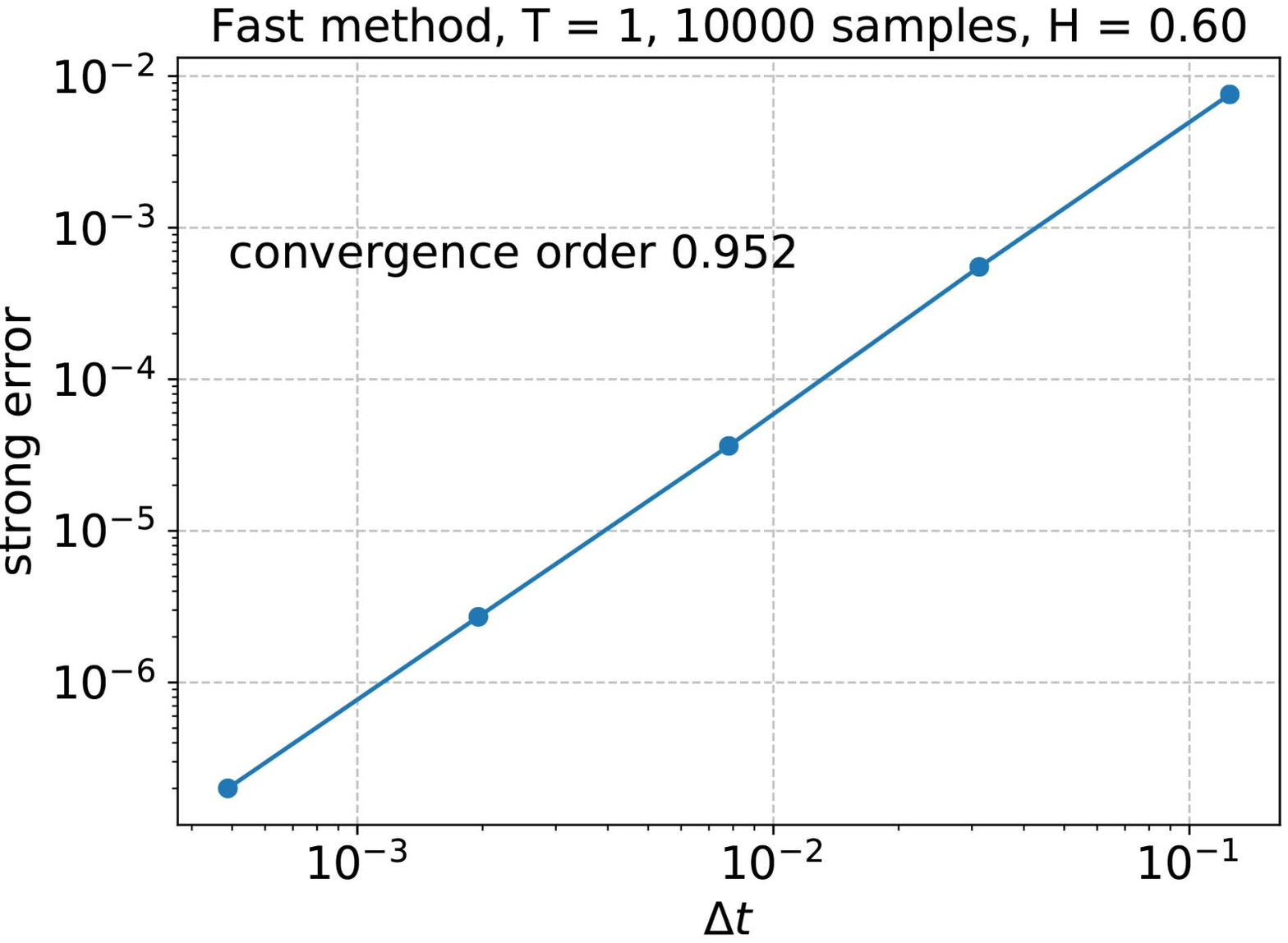} }
\caption{Example 1 (Section \ref{sec_ex1_strong}): The strong convergence of the direct method for $H = 0.8$ and $H = 0.6$, respectively in terms of various $\Delta t$ in log-log scale. The convergence orders match the theoretical result $\min \{\frac{3}{2}-H, 3-3H\}$, as is proved in Theorem \ref{thm:directmethoderr}.}
\label{ex1_fig3}
\end{figure}

\begin{figure}[htbp]
\centering
\includegraphics[width=0.7\textwidth]{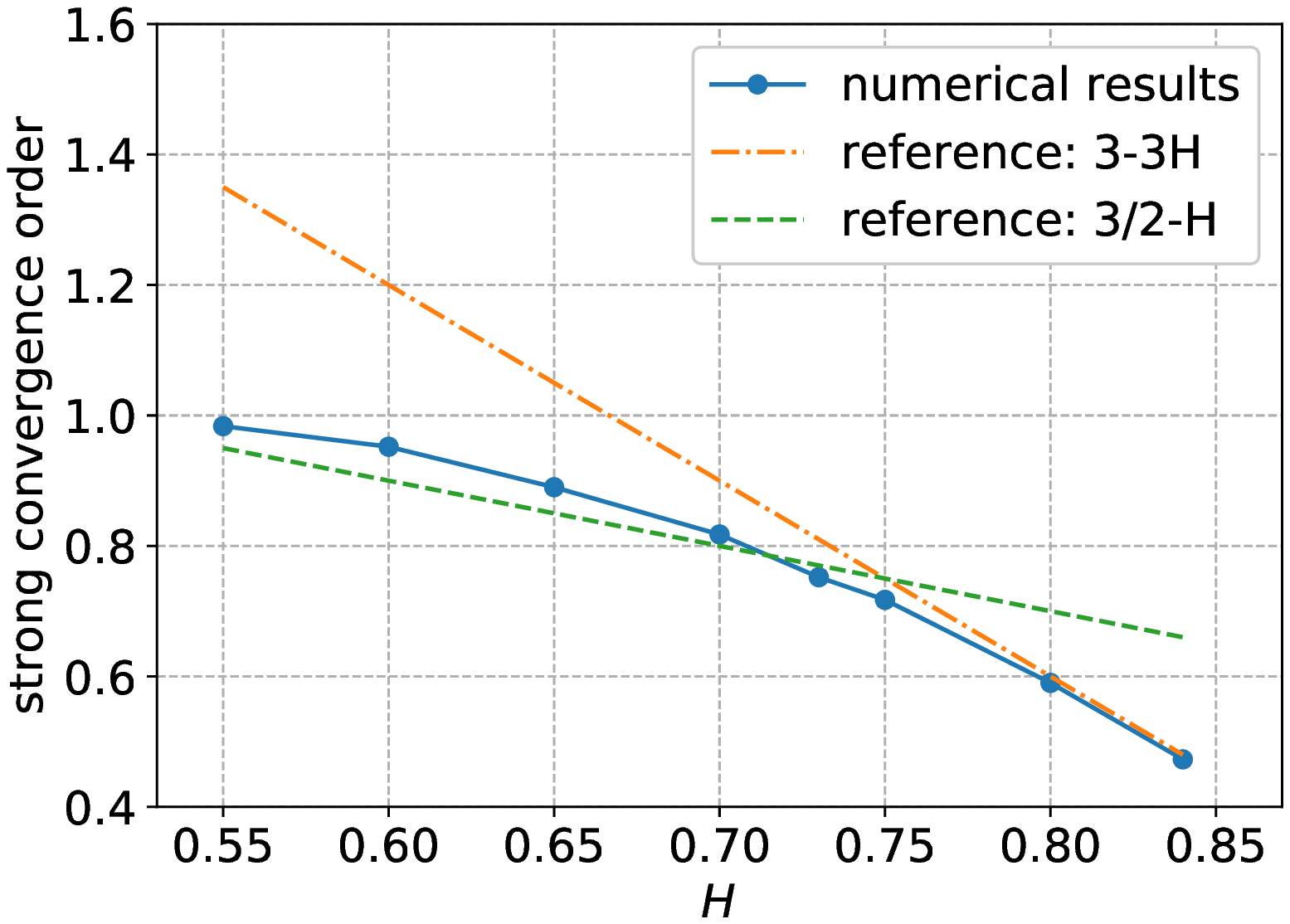}

\caption{Example 1 (Section \ref{sec_ex1_strong}): The strong convergence of the fast method for $H = 0.55, 0.6, 0.65, 0.7, 0.73, 0.75, 0.8, 0.84$ computed over $10000$ sample paths in terms of $\Delta t = 2^{-11}, 2^{-9}, 2^{-7}, 2^{-5}, 2^{-3}$. The convergence orders match those proved in Theorem \ref{thm:fastmethoderr}.}
\label{ex1_fig4}
\end{figure}

\subsubsection{The test of ergodicity} \label{sec_ex1_erg}
As has been proved rigorously in \cite[Theorem 2]{liliulu2017} that for the linear force case, the process has ergodicity and converges algebraically to the Gibbs measure
\begin{equation}
\mu(dx) \sim \exp \left( -\frac{1}{2} x^2\right)\,dx. \label{ex1_gibbs}
\end{equation}
Consider the initial data $x_0 = 0$, $H = 0.75$ and computed over $50000$ sample paths with $\Delta t = 2^{-9}$. Fig. \ref{ex1_fig5} plots the empirical distribution at different times $t = 0, 0.25, 1, 2, 8, 32$. Next, we plot the variance of $x(t)$ (also called the mean square displacement in physical literature) and its difference between its equilibrium. As can be seen in Fig. \ref{ex1_fig6}, the variance of $x(t)$ convergences to its equilibrium $E(x^2(\infty)) = 1$ algebraically.

\begin{figure}[htbp]
\centering
\subfloat[$t = 0$]{\includegraphics[width=0.5\textwidth]{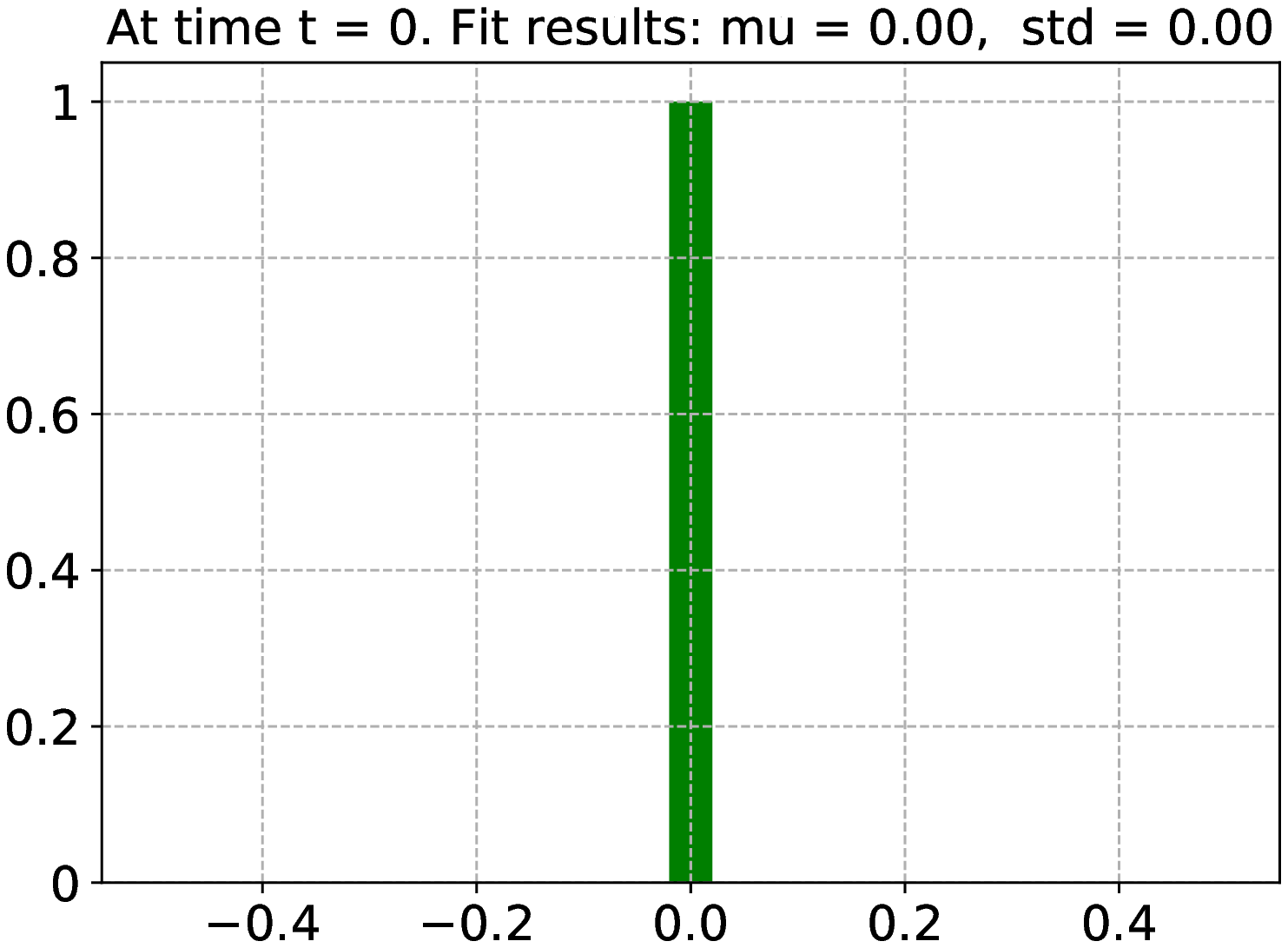} }
\subfloat[$t = 0.25 $]{\includegraphics[width=0.5\textwidth]{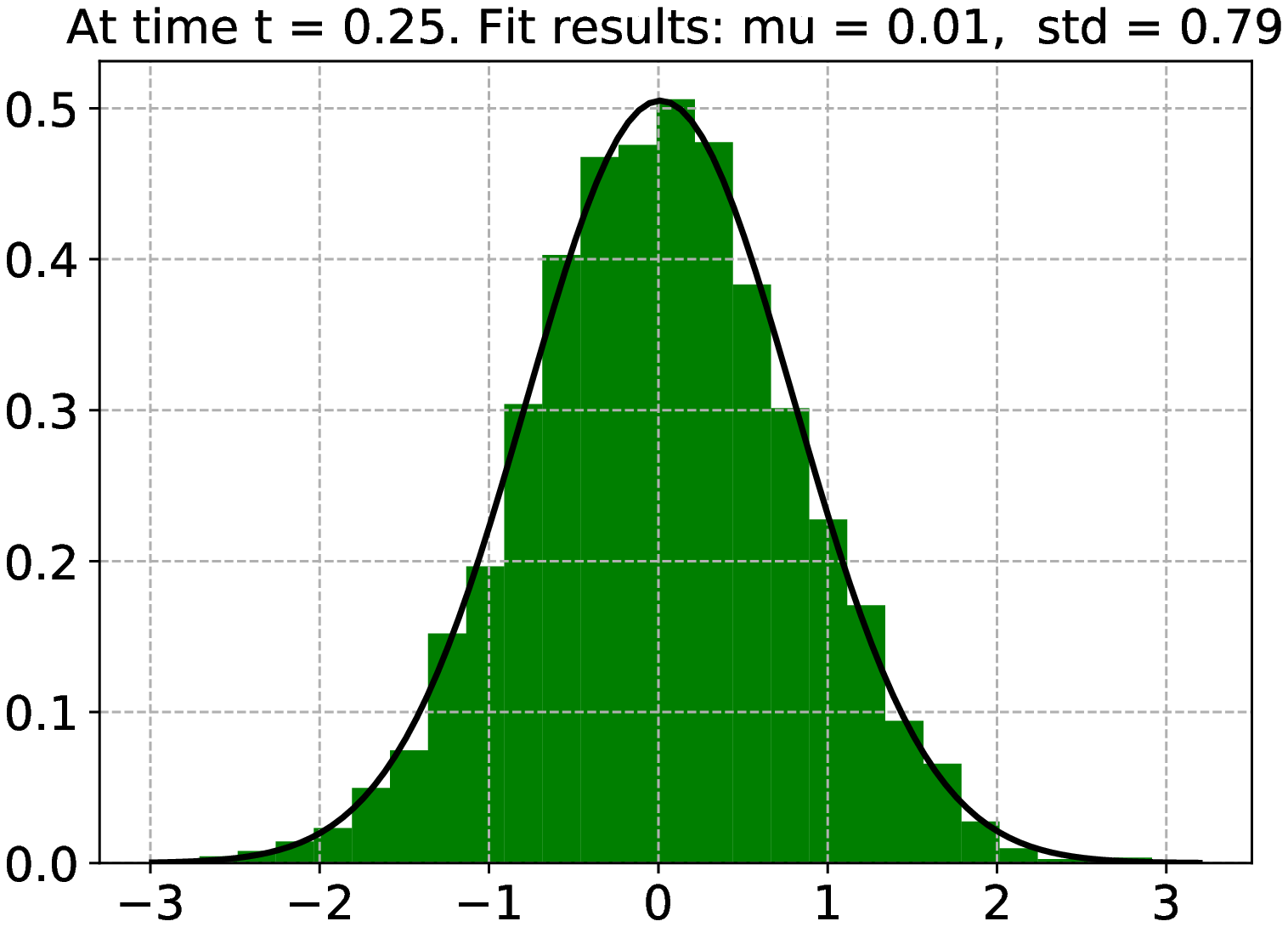} } \\
\subfloat[$t = 1$]{\includegraphics[width=0.5\textwidth]{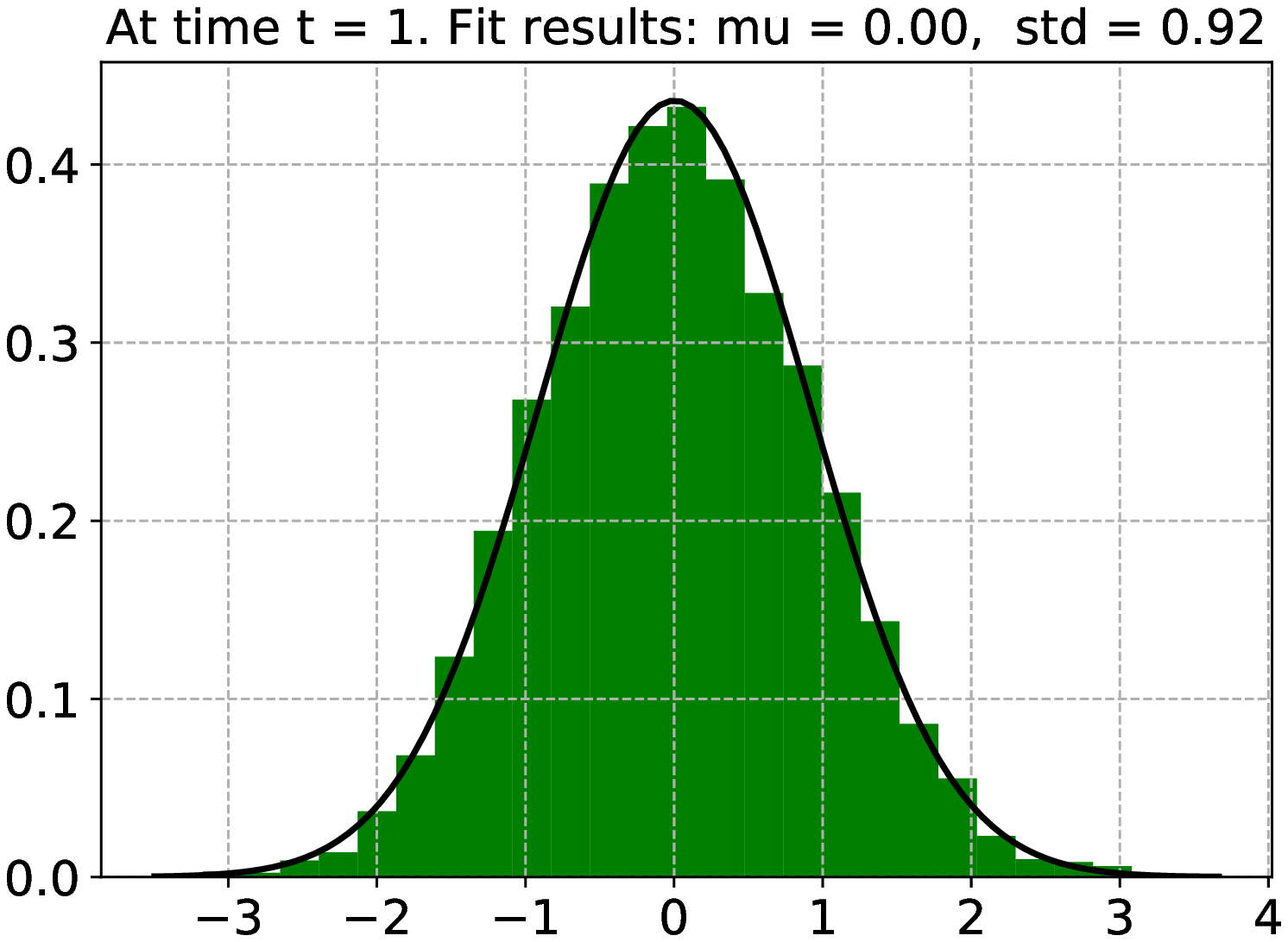} }
\subfloat[$t = 2$]{\includegraphics[width=0.5\textwidth]{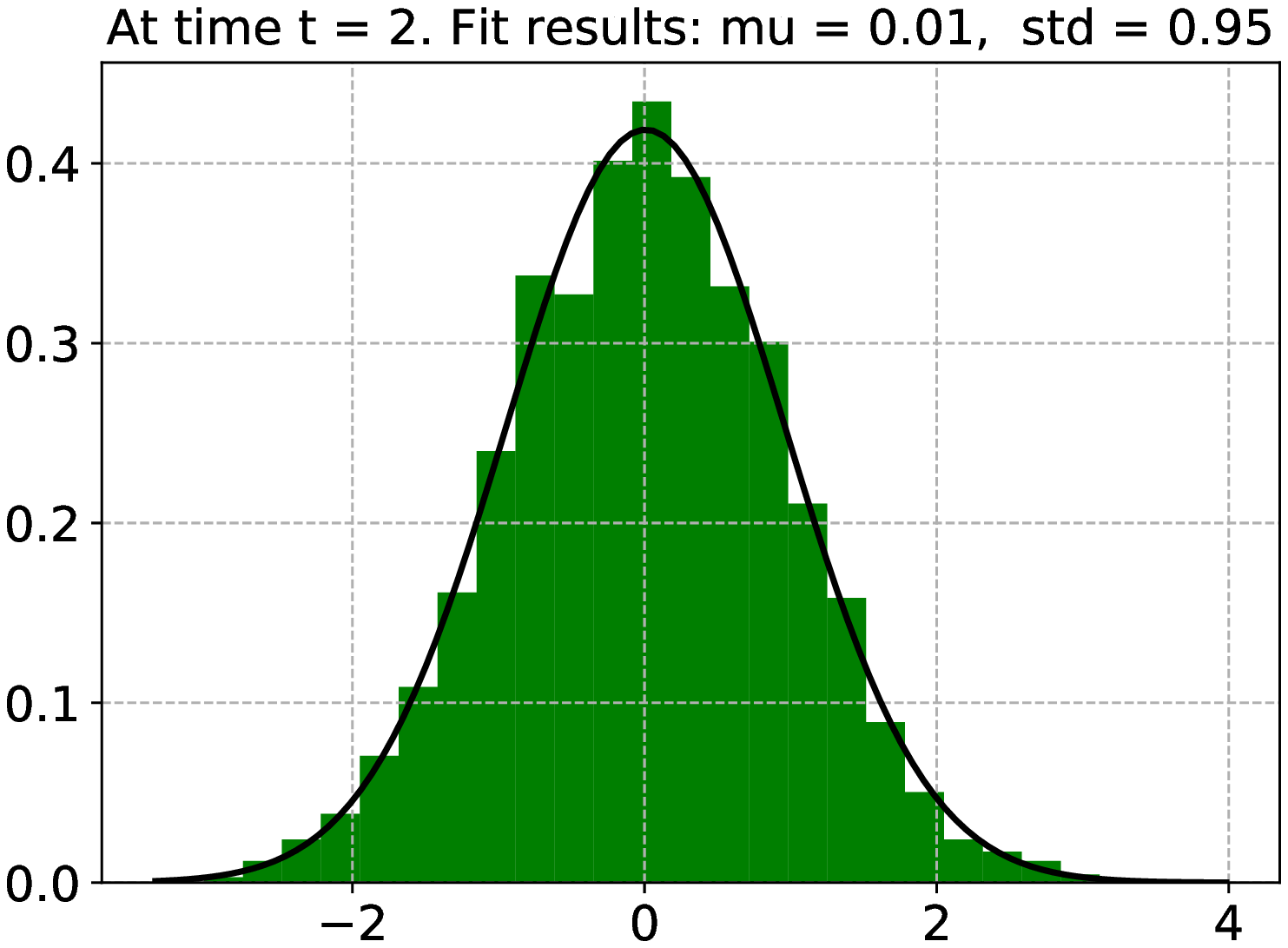} } \\
\subfloat[$t = 8$]{\includegraphics[width=0.5\textwidth]{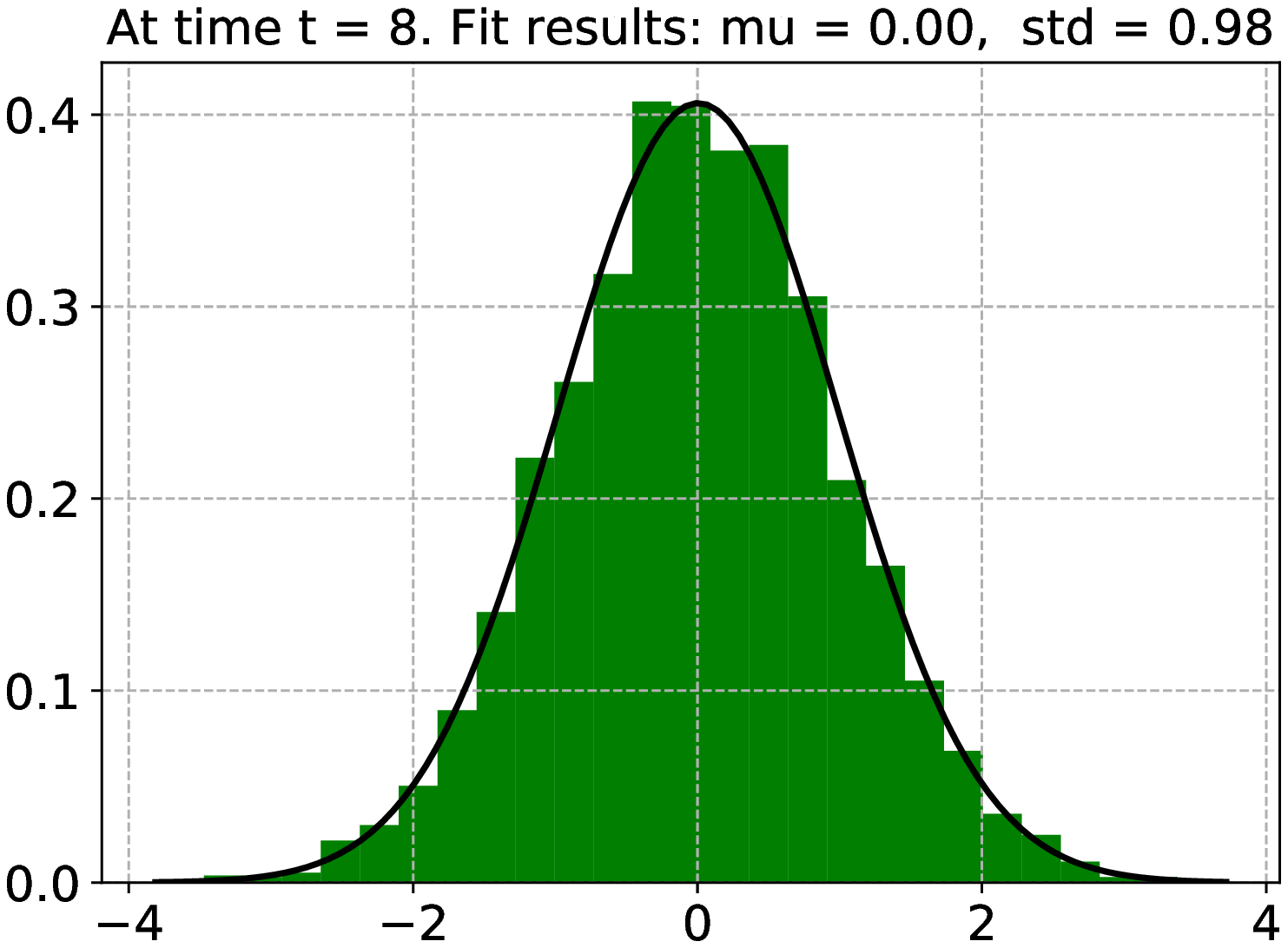} }
\subfloat[$t = 32 $]{\includegraphics[width=0.5\textwidth]{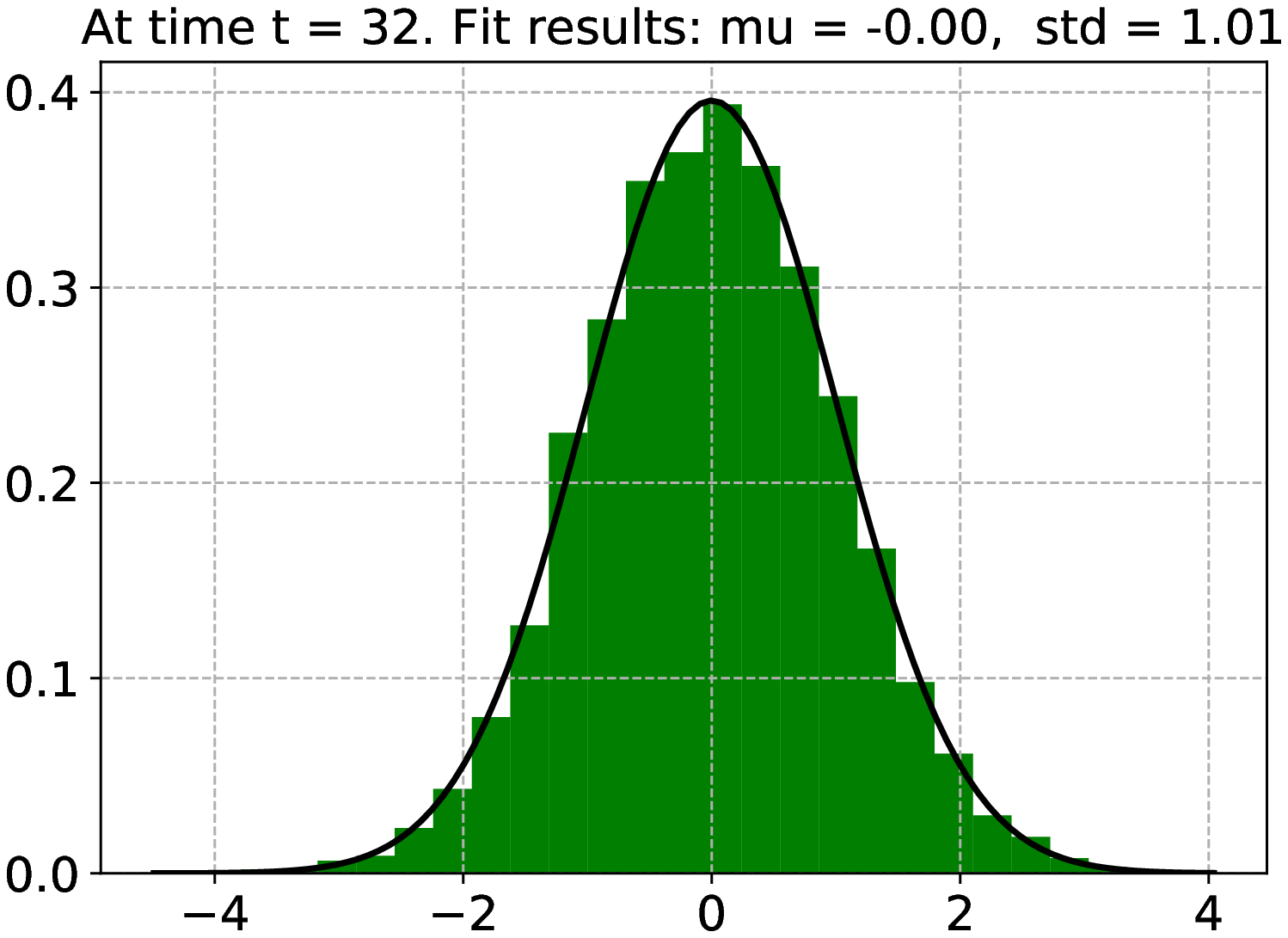} }
\caption{Example 1 (Section \ref{sec_ex1_erg}): the empirical distribution of $x$ at different times $t = 0, 0.25, 1, 2, 8, 32$. It can be seen that the distribution stays as a gaussian shape and converges to the Gibbs measure \eqref{ex1_gibbs} with mean 0 and variance 1.}
\label{ex1_fig5}
\end{figure}

\begin{figure}[htbp]
\centering
\subfloat[variance of $x(t)$]{\includegraphics[width=0.5\textwidth]{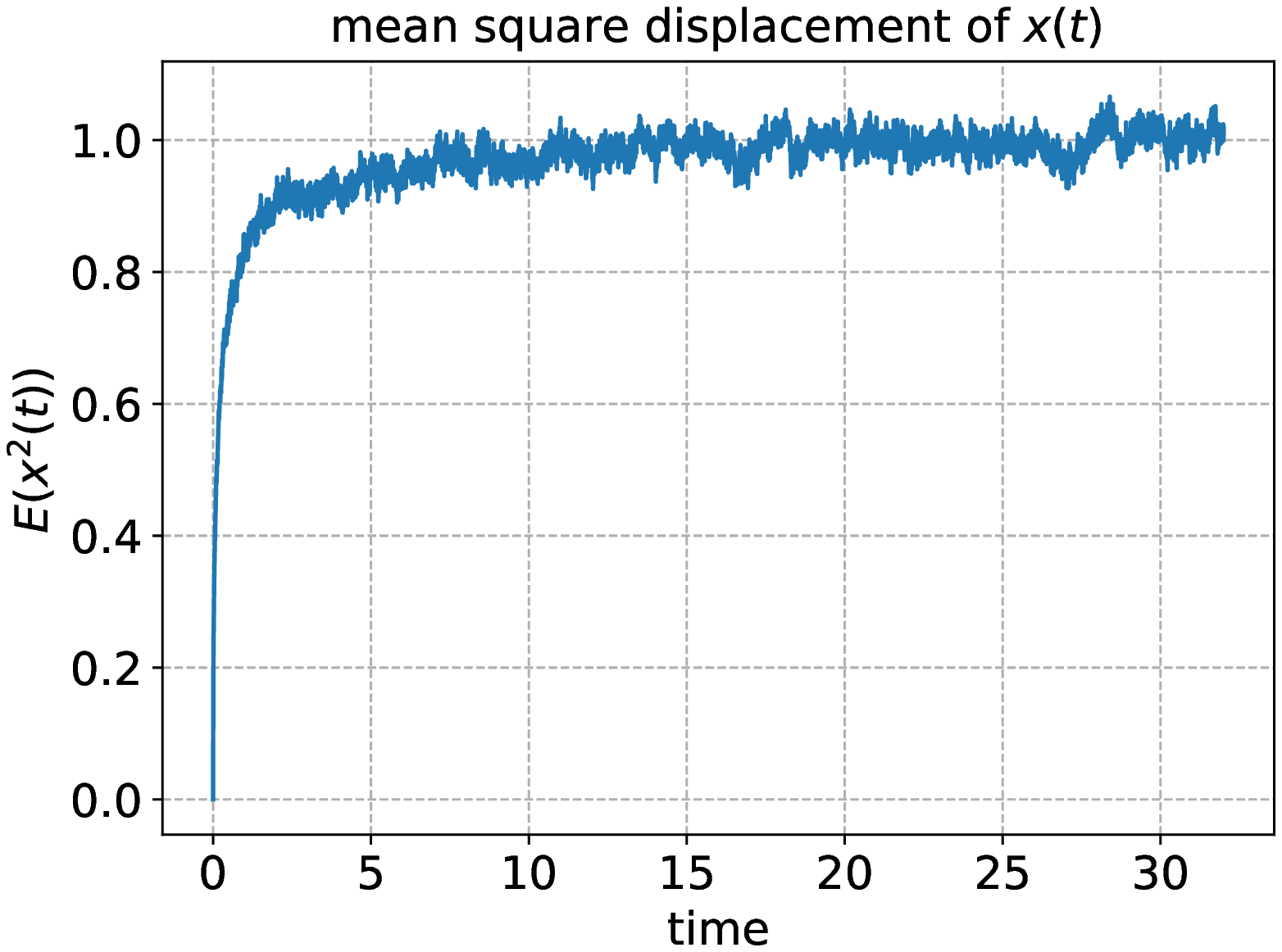} }
\subfloat[$E(x^2(\infty)) - E(x^2(t))$]{\includegraphics[width=0.5\textwidth]{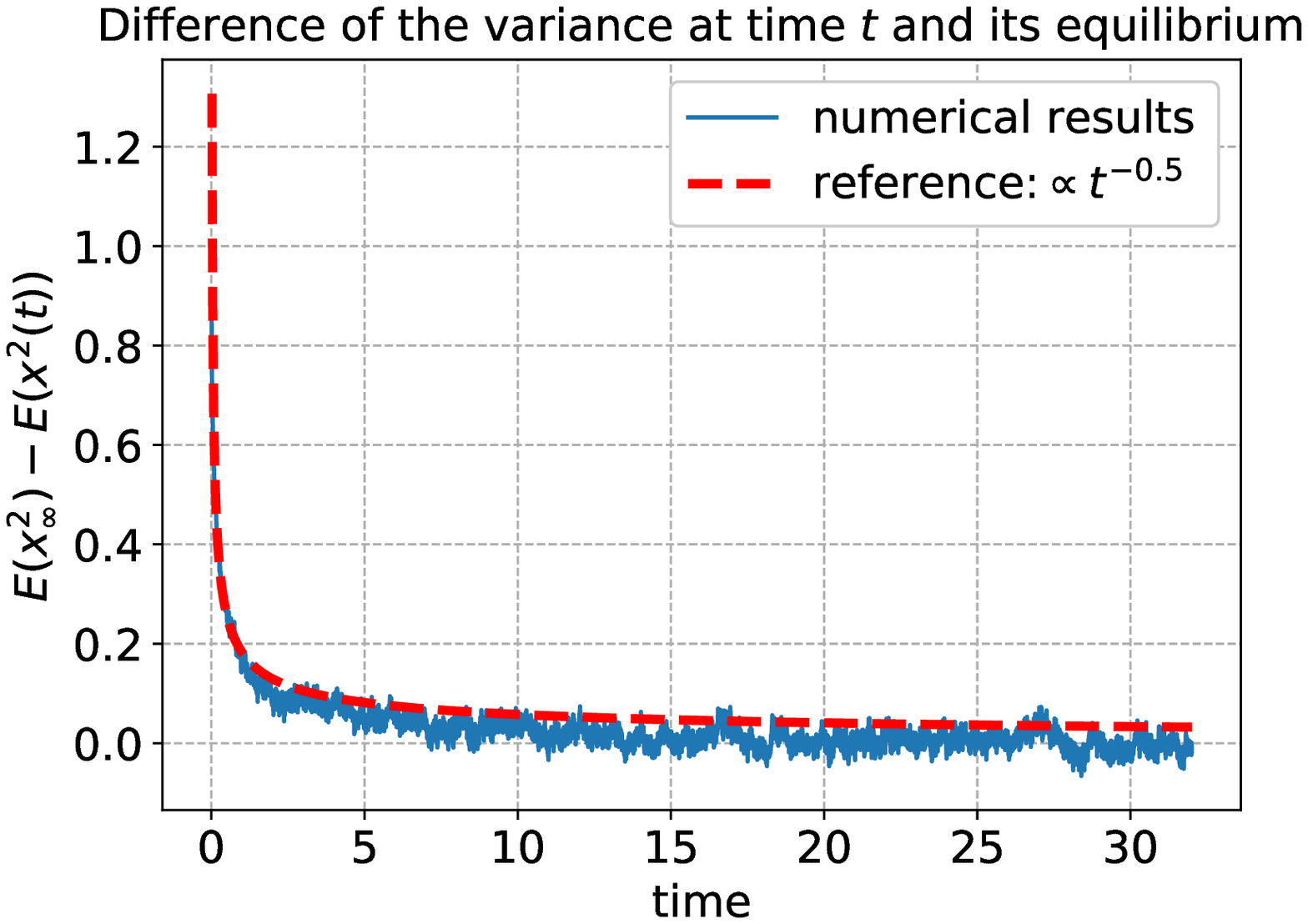} }
\caption{Example 1 (Section \ref{sec_ex1_erg}): The mean-square displacement of $x(t)$, i.e., the variance of $x(t)$ in the case $H = 0.75$. The left plot shows its tendency of approaching its equilibrium. The right figure plots its different between its equilibrium. As can be seen the decay rate is algebraic.}
\label{ex1_fig6}
\end{figure}

\subsection{Example 2 (confining potential)}

Here, we consider a 1D example, but with general potential $V(x)$ with confining structure. To be specific,
\[V(x) = ax^4 + bx^3 + cx^2 .\]
is considered for its ergodicity, where $b =0$ gives rise to the symmetric case whereas $b \neq 0$ corresponds to the asymmetric case.
\subsubsection{Ergodicity of symmetric confining potential} \label{sec_ex2_sym}
Consider the symmetric double well potential
$$V(x) = \frac{1}{4}x^4 - \frac{1}{2}x^2, \quad V'(x) = x^3 - x.$$
Consider the initial data $x_0 = 1$, $H = 0.6$ and computed over $50000$ sample paths with $\Delta t = 2^{-5}$ till the final time $T = 512$ using the fast algorithm. Fig. \ref{ex2_symfig} shows the empirical distribution at different times. It can be seen that the empiral distribution of $x$ concentrates at $x=1$ initially, then gradually expands and shifts to the left, and finally presents a symmetric double-well shape that matches the reference Gibbs measure
$$\mu(dx) \sim \exp(-V(x)) \,dx.$$
Note that to consider the long time behavior, the initial values of $x$ does not matter much.
\begin{figure}[htbp]
\centering
\subfloat[$t = 0$]{\includegraphics[width=0.5\textwidth]{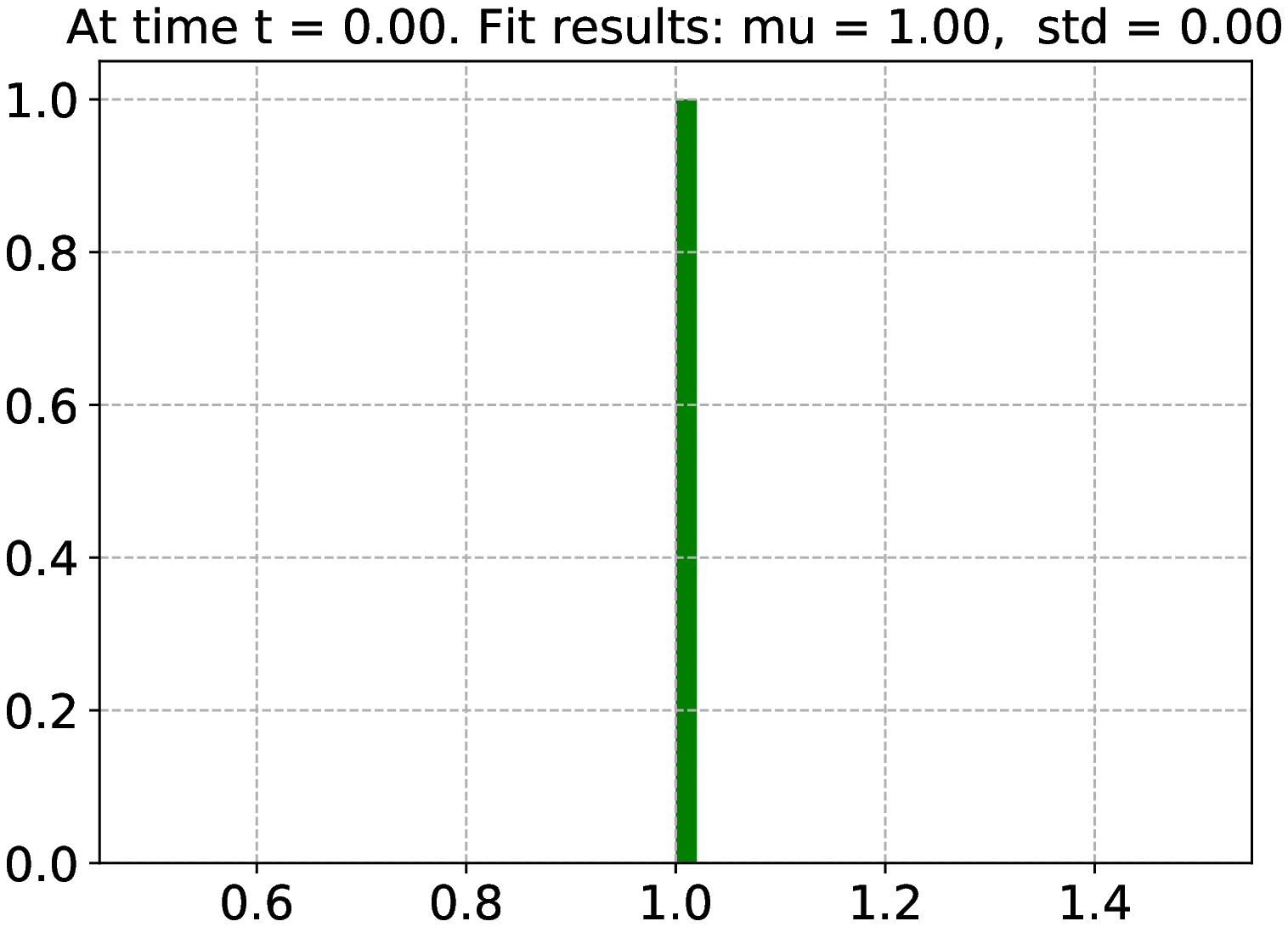} }
\subfloat[$t = 0.0625$]{\includegraphics[width=0.5\textwidth]{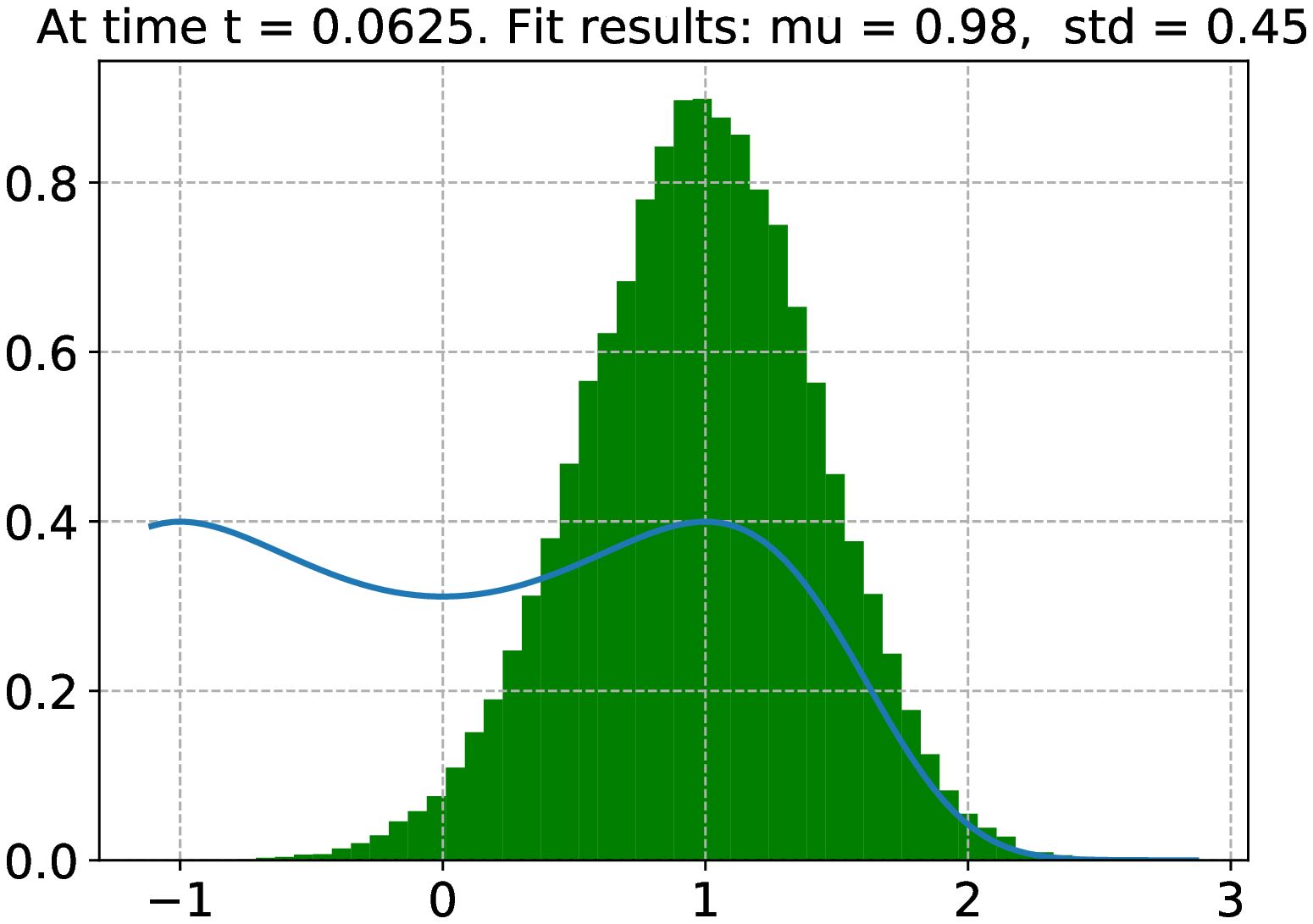} }\\
\subfloat[$t = 0.25 $]{\includegraphics[width=0.5\textwidth]{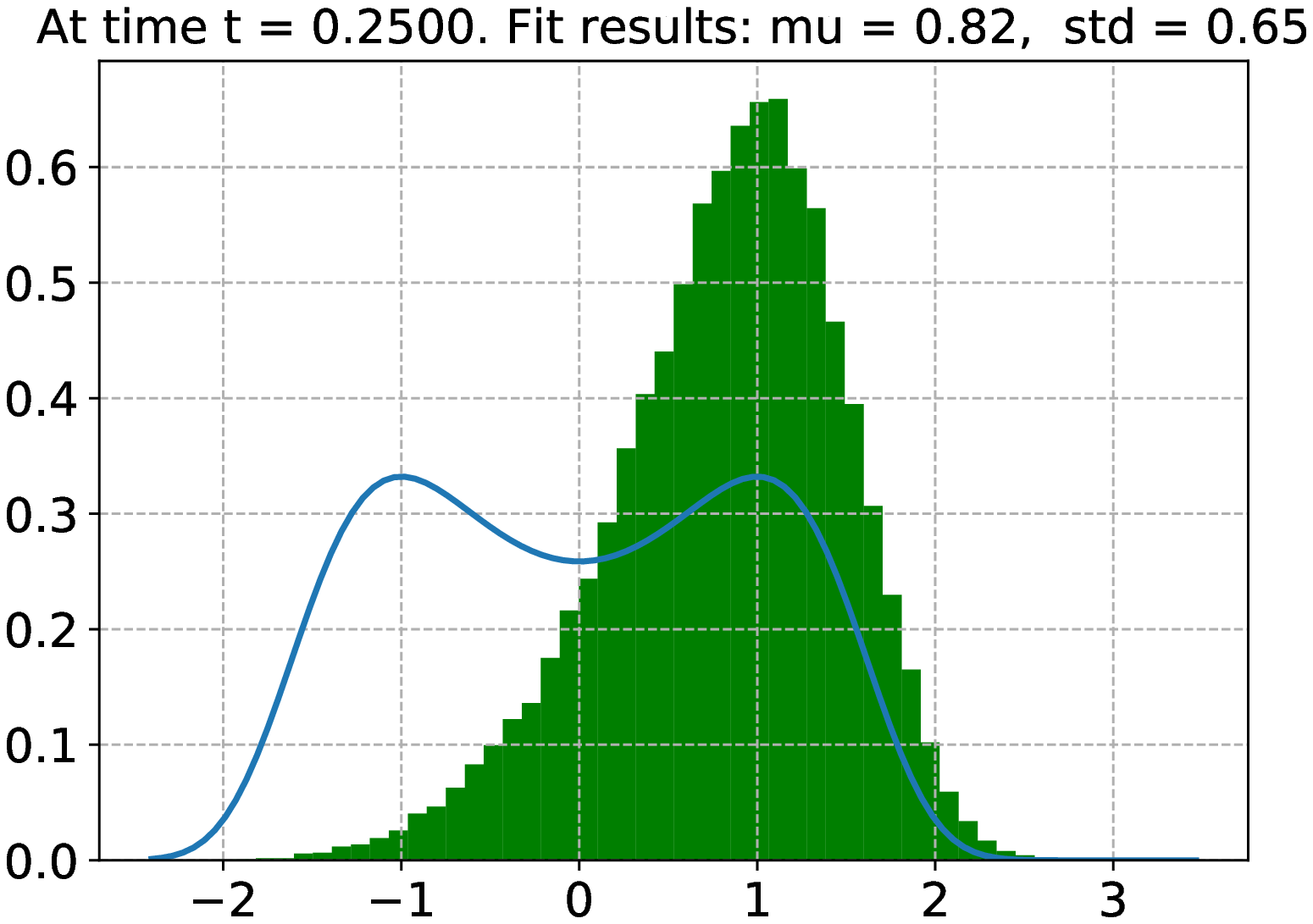} }
\subfloat[$t = 1 $]{\includegraphics[width=0.5\textwidth]{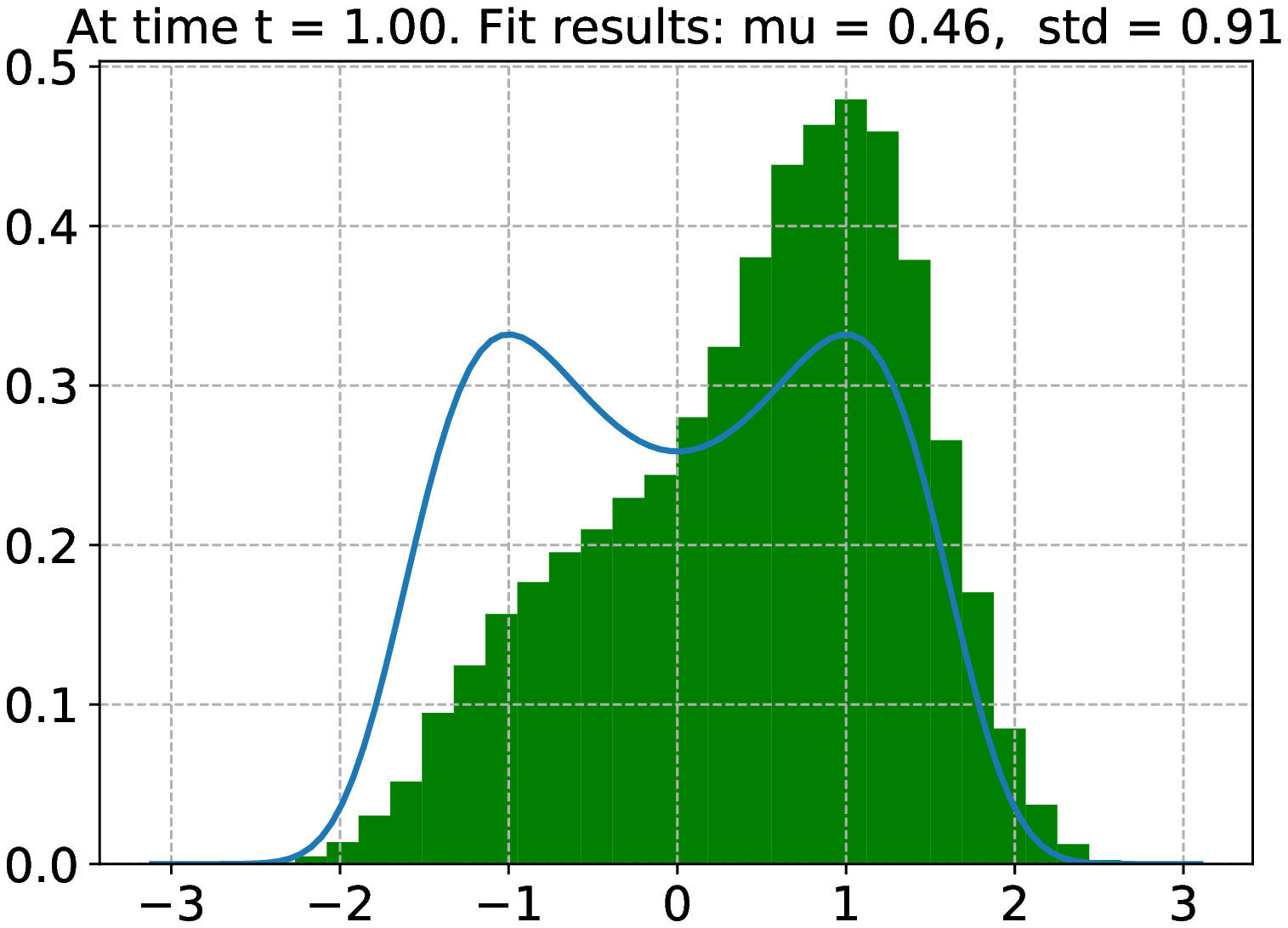} } \\
\subfloat[$t = 2$]{\includegraphics[width=0.5\textwidth]{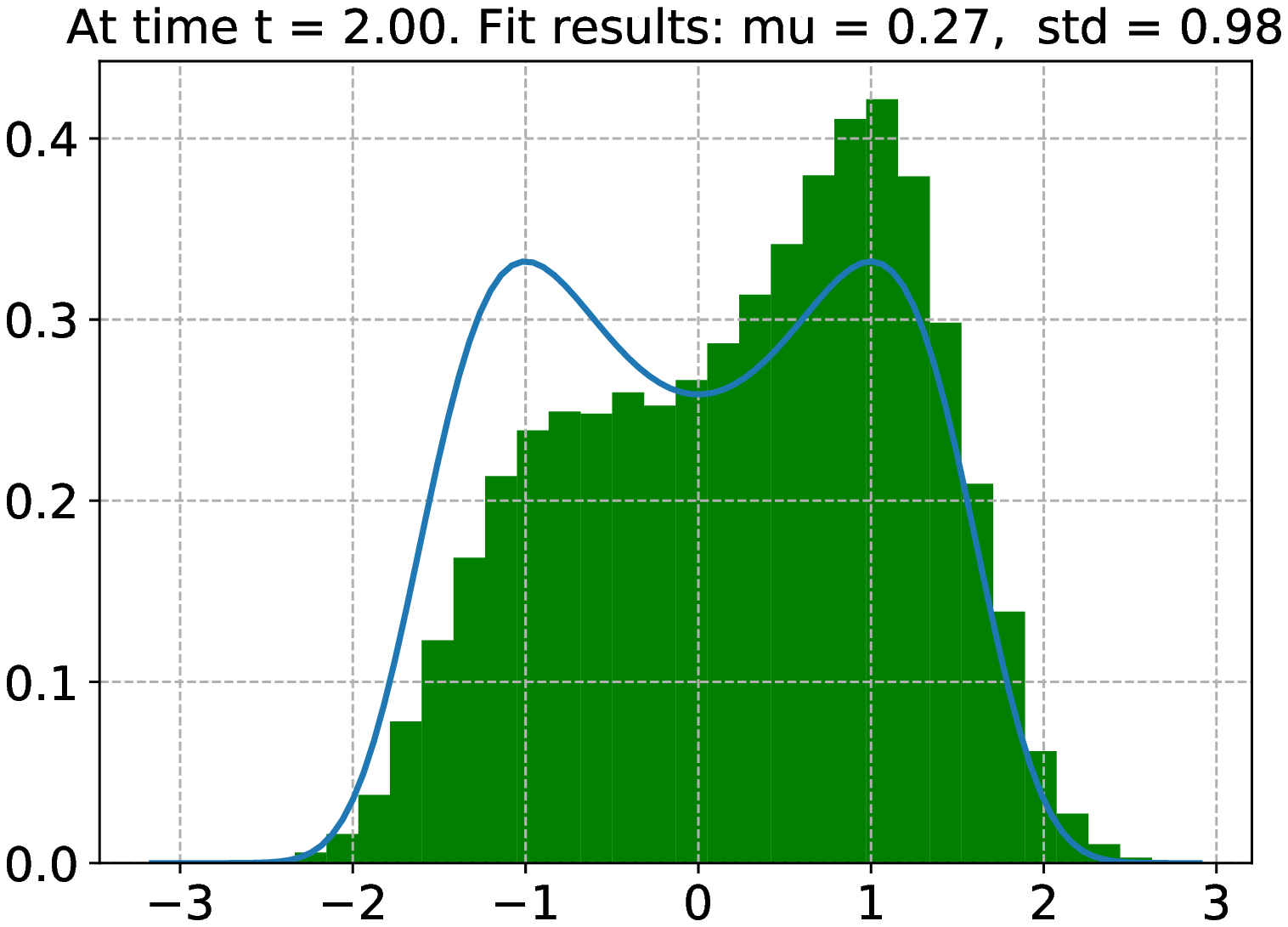} }
\subfloat[$t = 4$]{\includegraphics[width=0.5\textwidth]{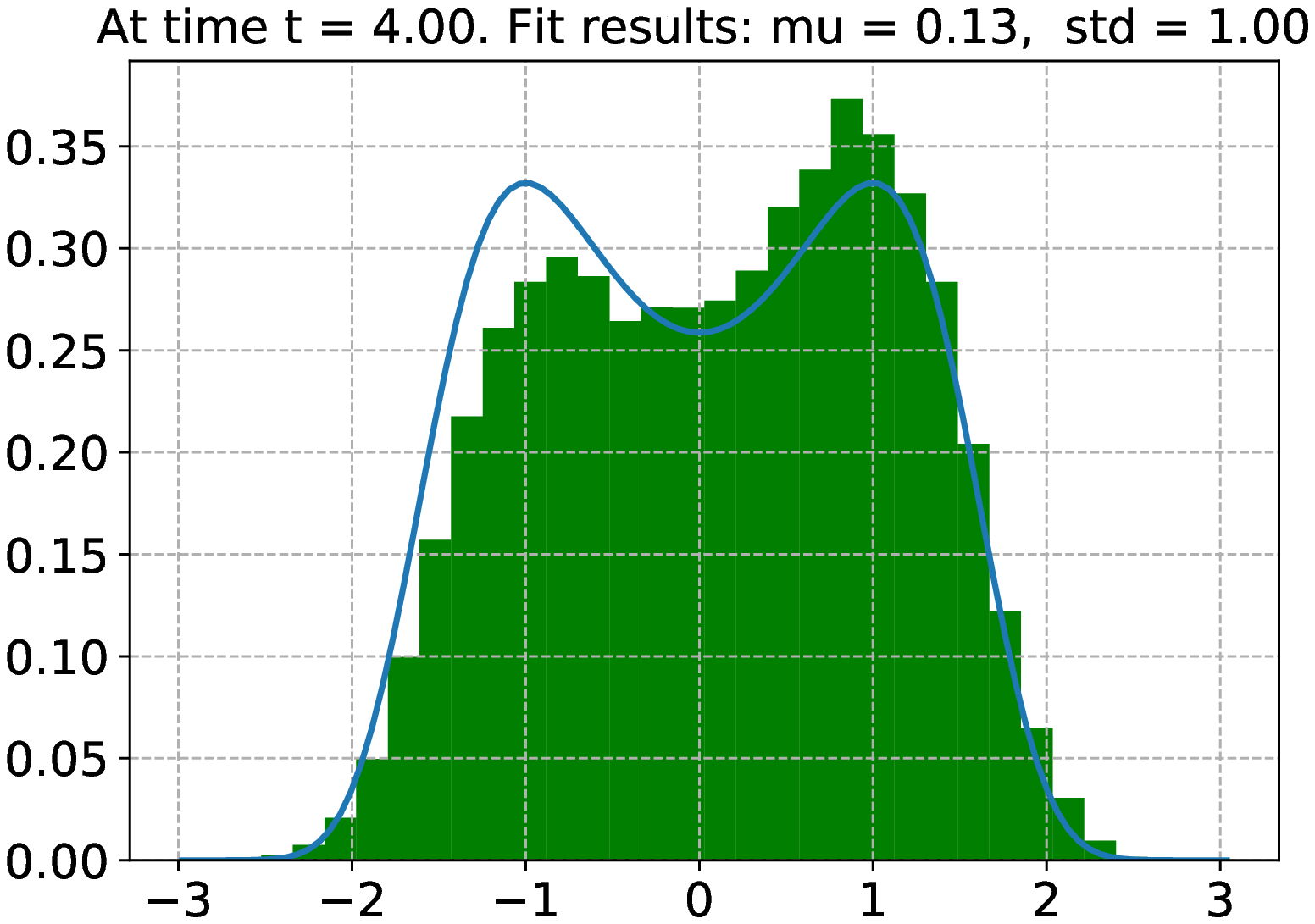} }
\caption{Example 2 (Section \ref{sec_ex2_sym}): the empirical distribution of $x$ at different times $t = 0, 0.0625, 0.25, 1, 2, 4, 8, 16, 32, 512$. The solid line is the reference gibbs measure $\sim \exp(-V(x))$. It can be seen that given the intial data concentrating at $x=0$, the distribution of $x$ expands, and moves gradually to create a symmetric double-well shape.}
\label{ex2_symfig}
\end{figure}

\begin{figure}[htbp]
\ContinuedFloat
\centering
\subfloat[$t = 8$]{\includegraphics[width=0.5\textwidth]{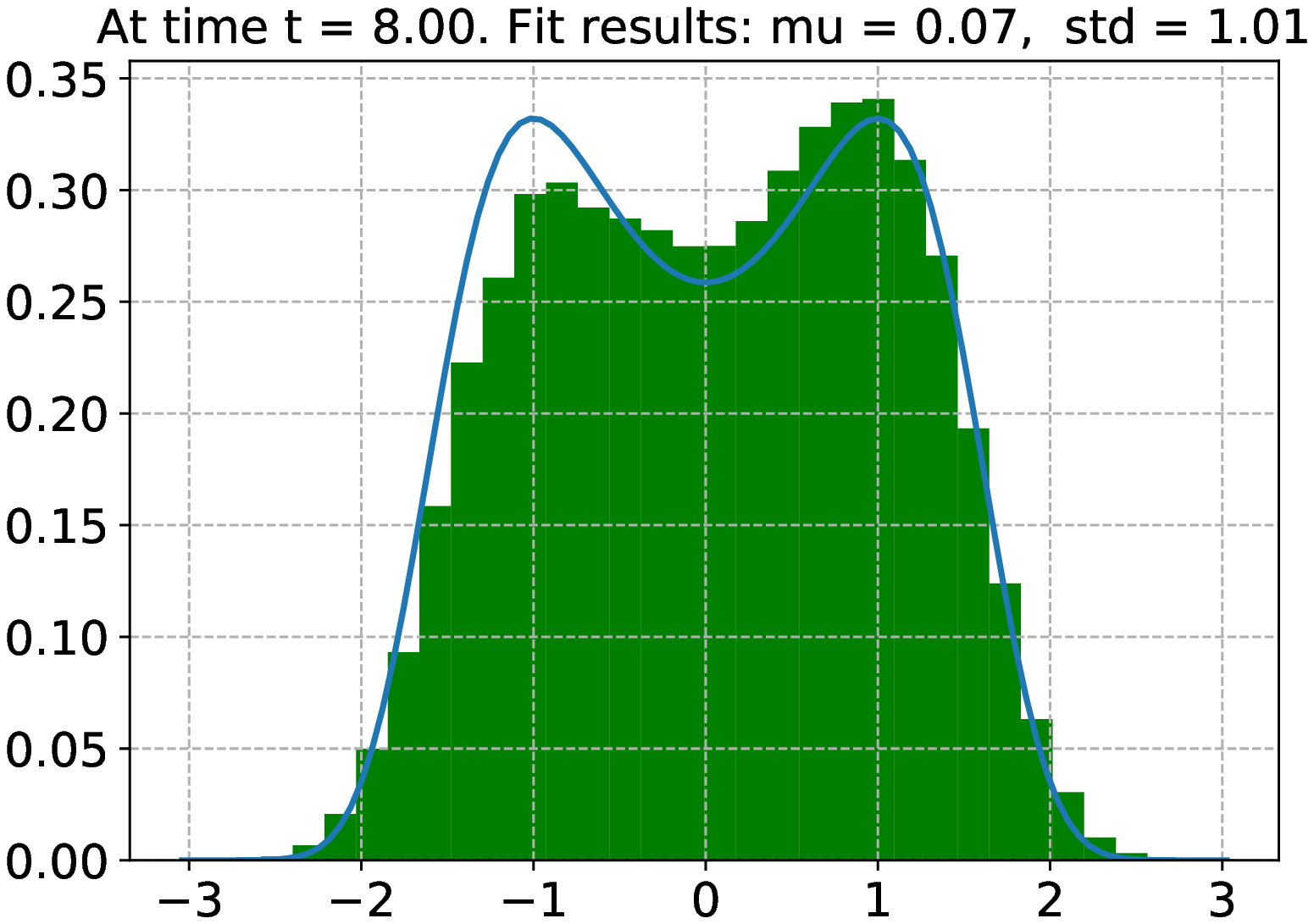} }
\subfloat[$t = 16$]{\includegraphics[width=0.5\textwidth]{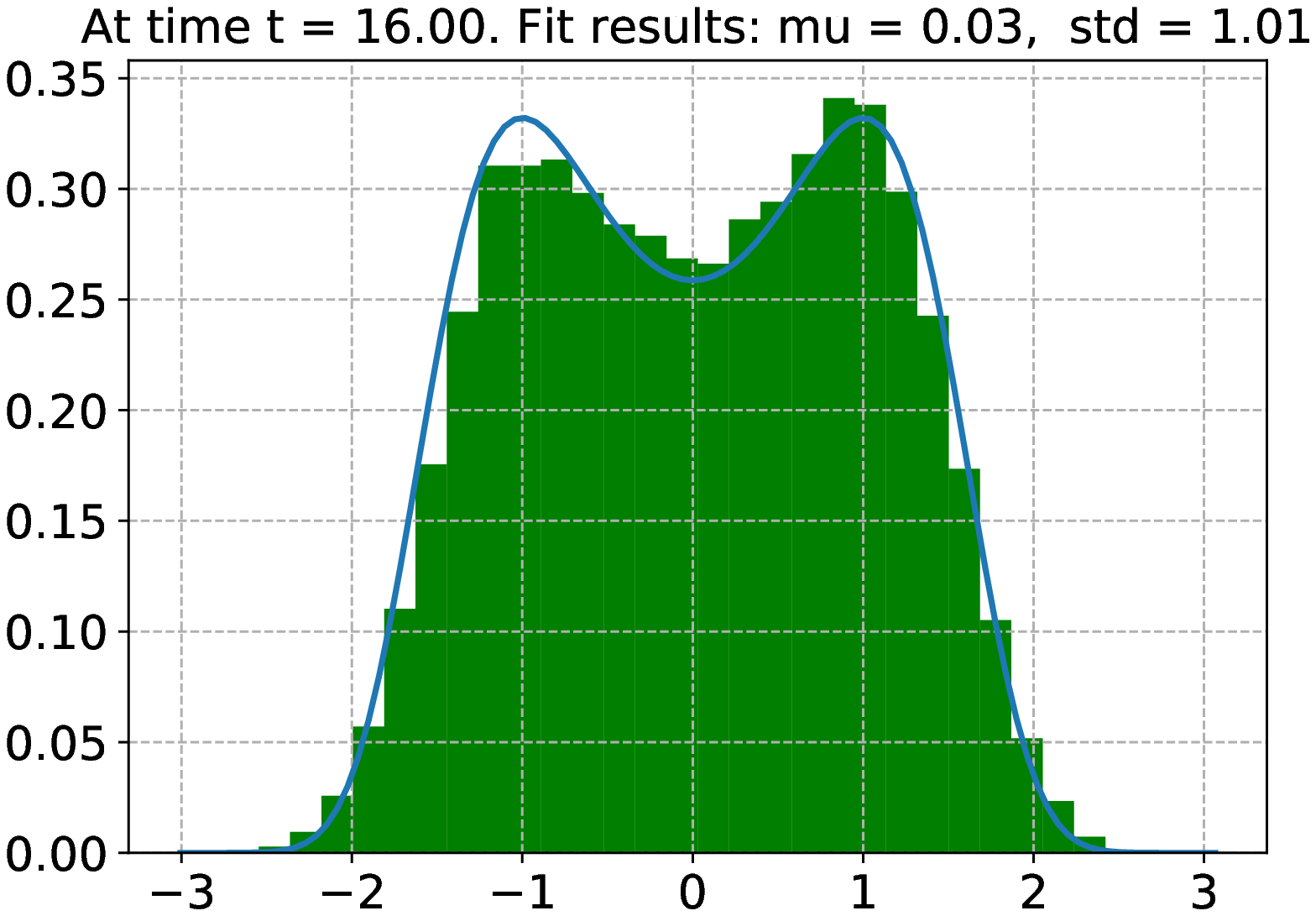} } \\
\subfloat[$t = 32 $]{\includegraphics[width=0.5\textwidth]{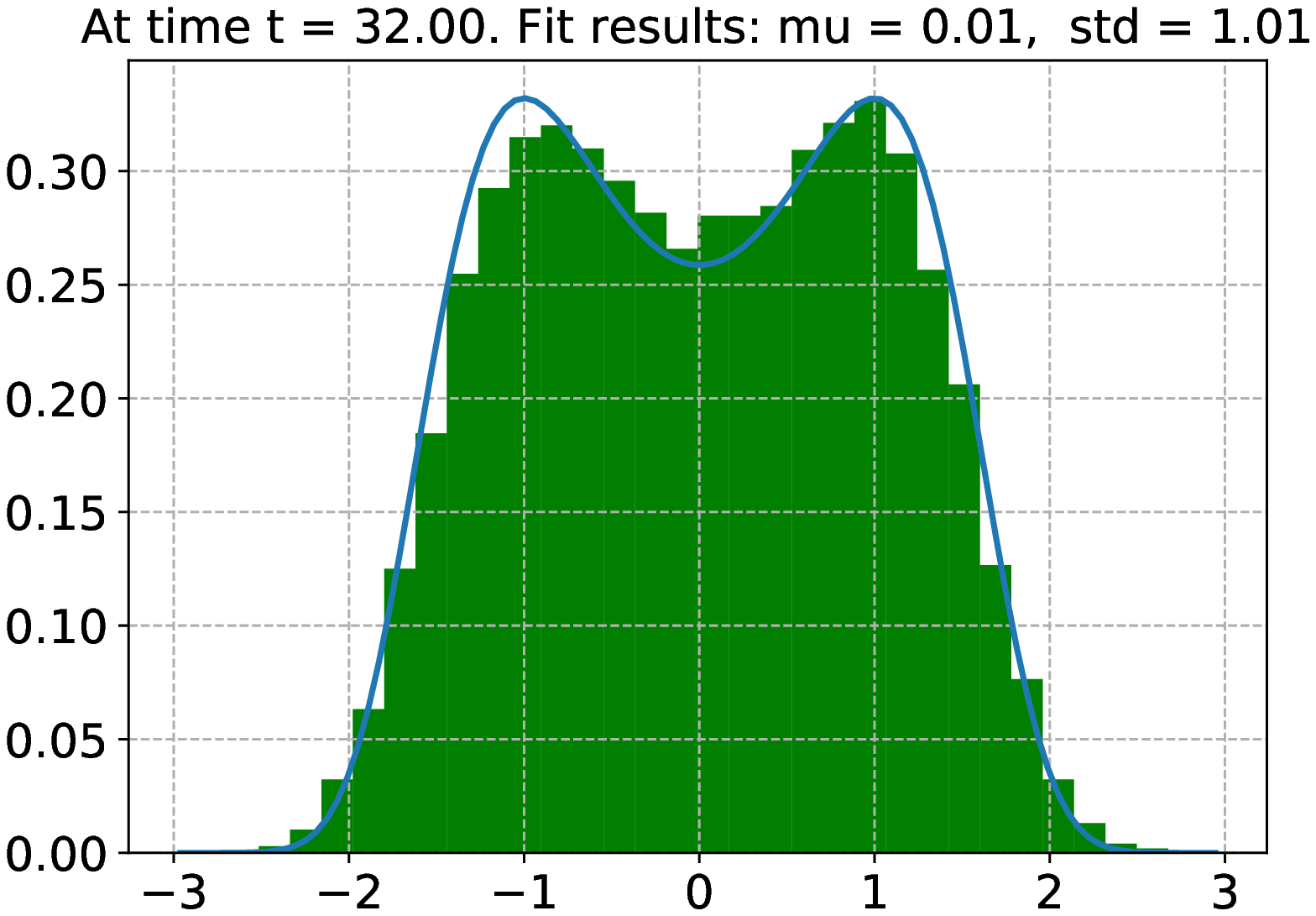} }
\subfloat[$t = 512 $]{\includegraphics[width=0.5\textwidth]{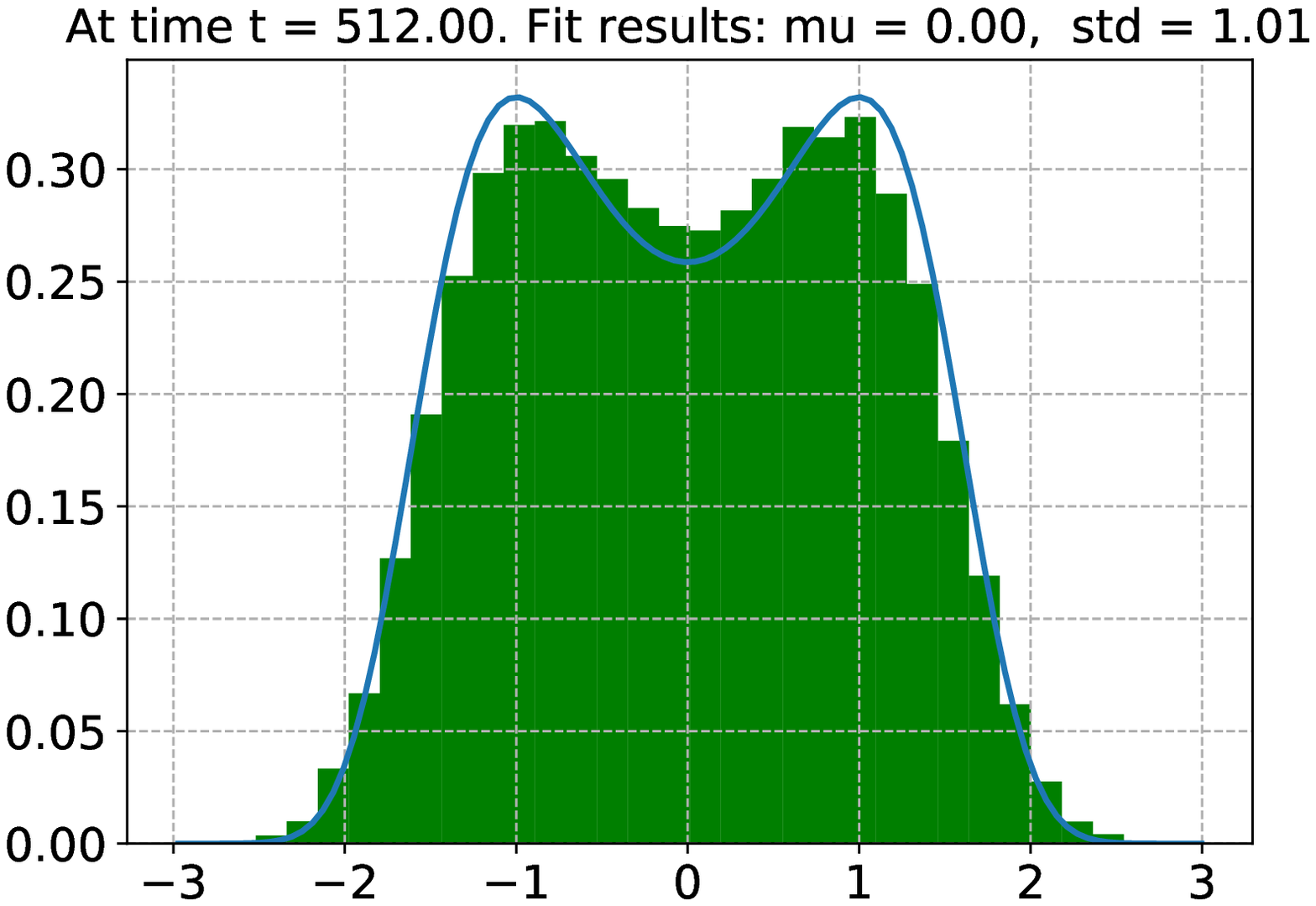} }
\captionsetup{list=off,format=cont}
\caption{Example 2 (Section \ref{sec_ex2_sym}): the empirical distribution of $x$ at various moments $t = 0, 0.0625, 0.25, 1, 2, 4, 8, 16, 32, 512$. }
\label{ex2_symfig2}
\end{figure}

\subsubsection{Ergodicity of asymmetric confining potential} \label{sec_ex2_asym}
Consider the asymmetric double well potential
$$V(x) = \frac{1}{4}x^4 + \frac{1}{3}x^3 - x^2 , \quad V'(x) = x^3 + x^2 - 2x.$$
Consider the initial data $x_0 = 1$, $H = 0.6$ and computed over $50000$ sample paths with $\Delta t = 2^{-5}$ till the final time $T = 512$. Fig. \ref{ex2_asymfig1} and Fig. \ref{ex2_asymfig2} (cont.) plot the empirical distribution at different times. Since the initial data are all assigned as $x_0 = 1$, at first the empiral distribution of $x$ concentrates at $x=1$, then expands according to the time evolution, gradually shifts to the left and presents an asymmetric double-well shape in the end, which resembles the reference Gibbs measure
$$\mu(dx) \sim \exp(-V(x)) \,dx.$$

\begin{figure}[htbp]
\centering
\subfloat[$t = 0$]{\includegraphics[width=0.5\textwidth]{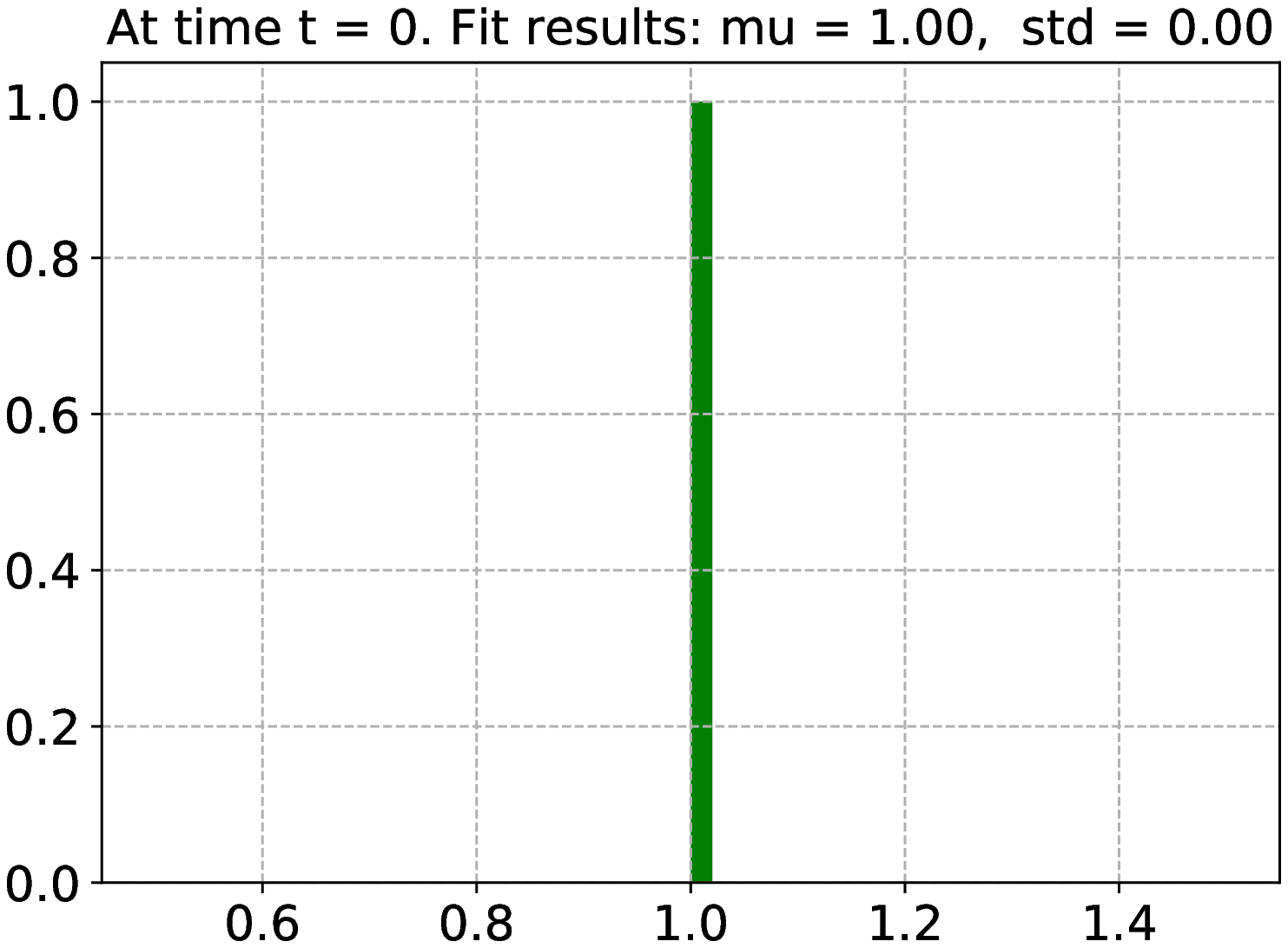} }
\subfloat[$t = 0.0625$]{\includegraphics[width=0.5\textwidth]{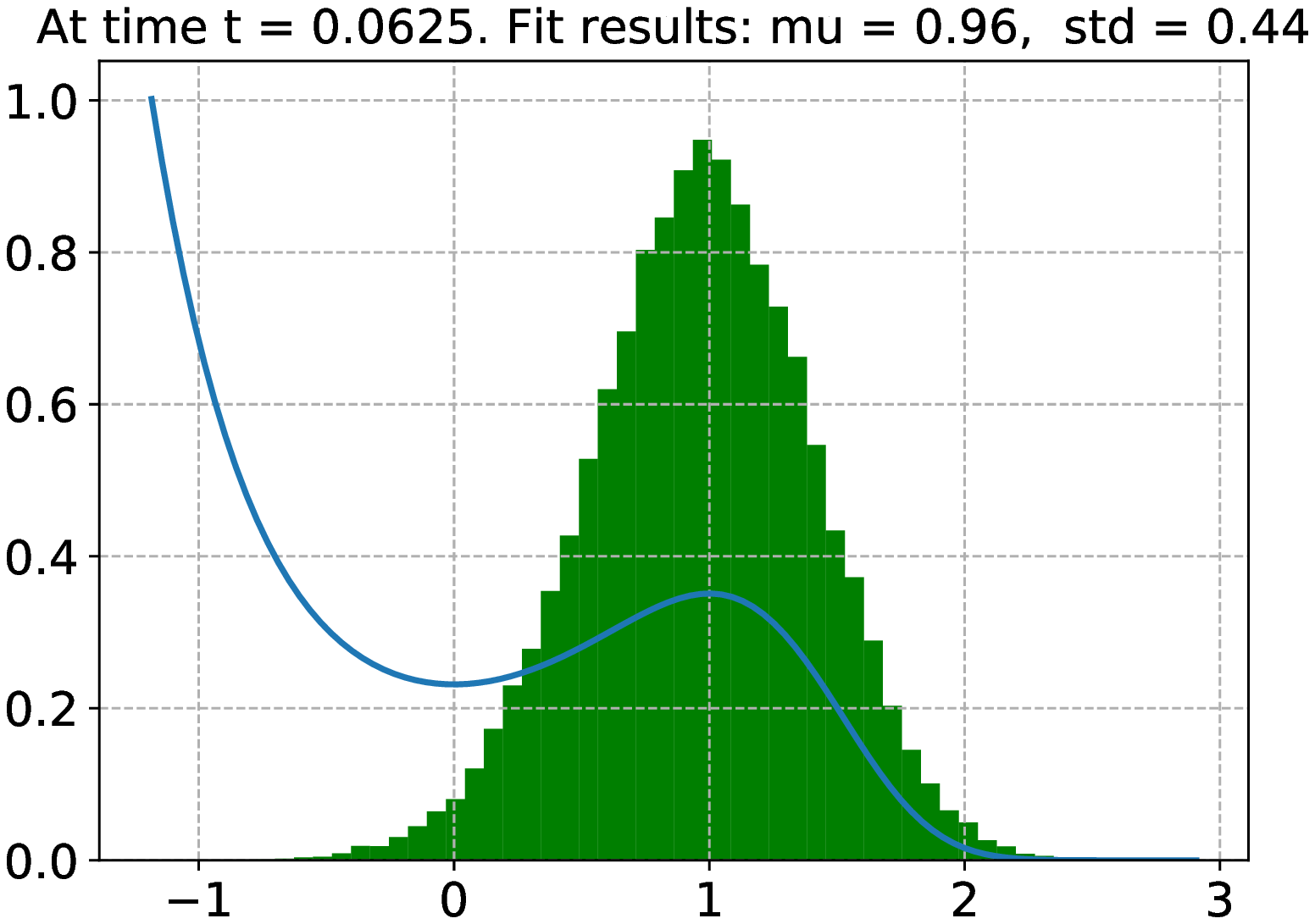} }\\
\subfloat[$t = 0.25 $]{\includegraphics[width=0.5\textwidth]{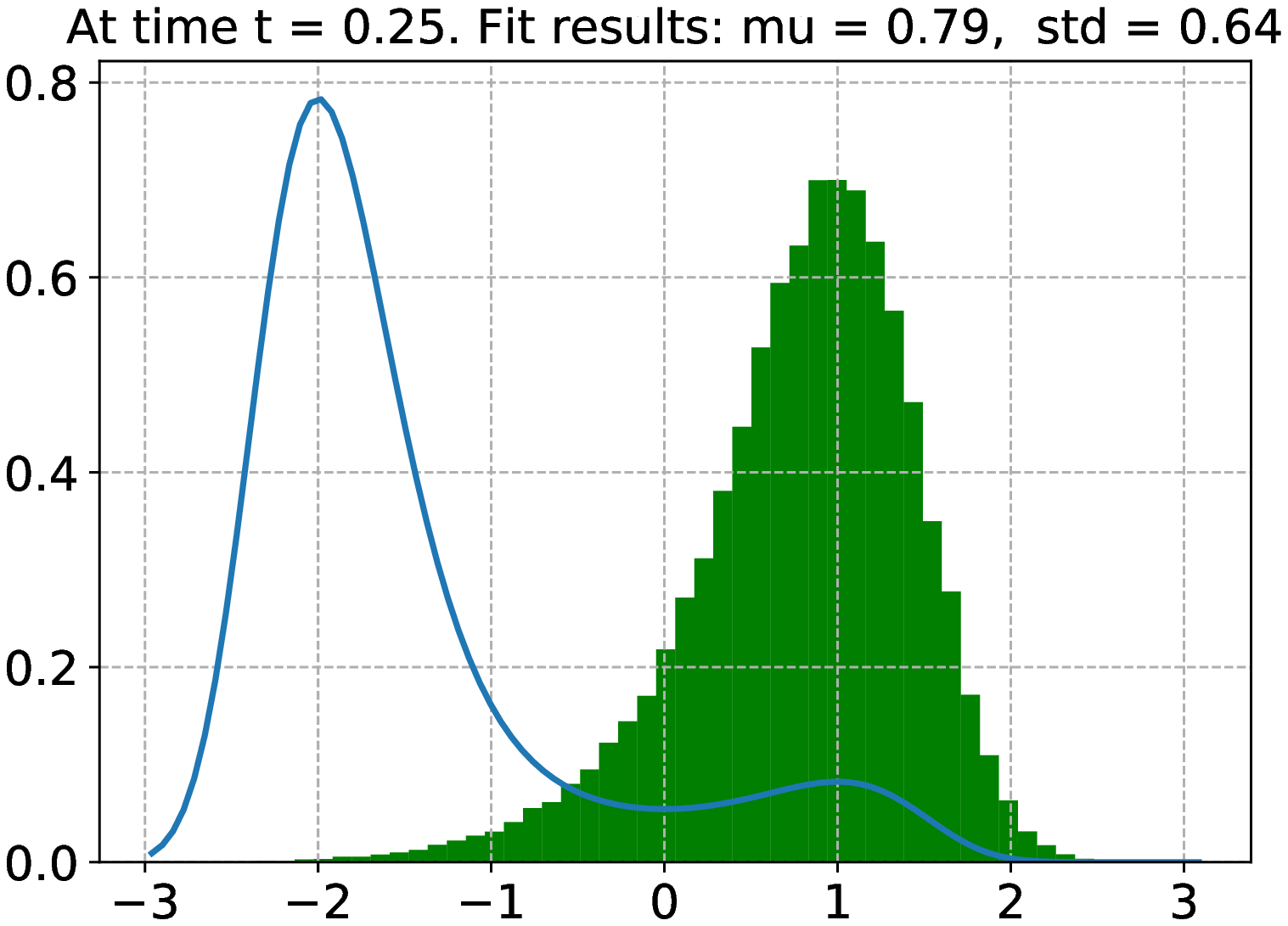} }
\subfloat[$t = 1 $]{\includegraphics[width=0.5\textwidth]{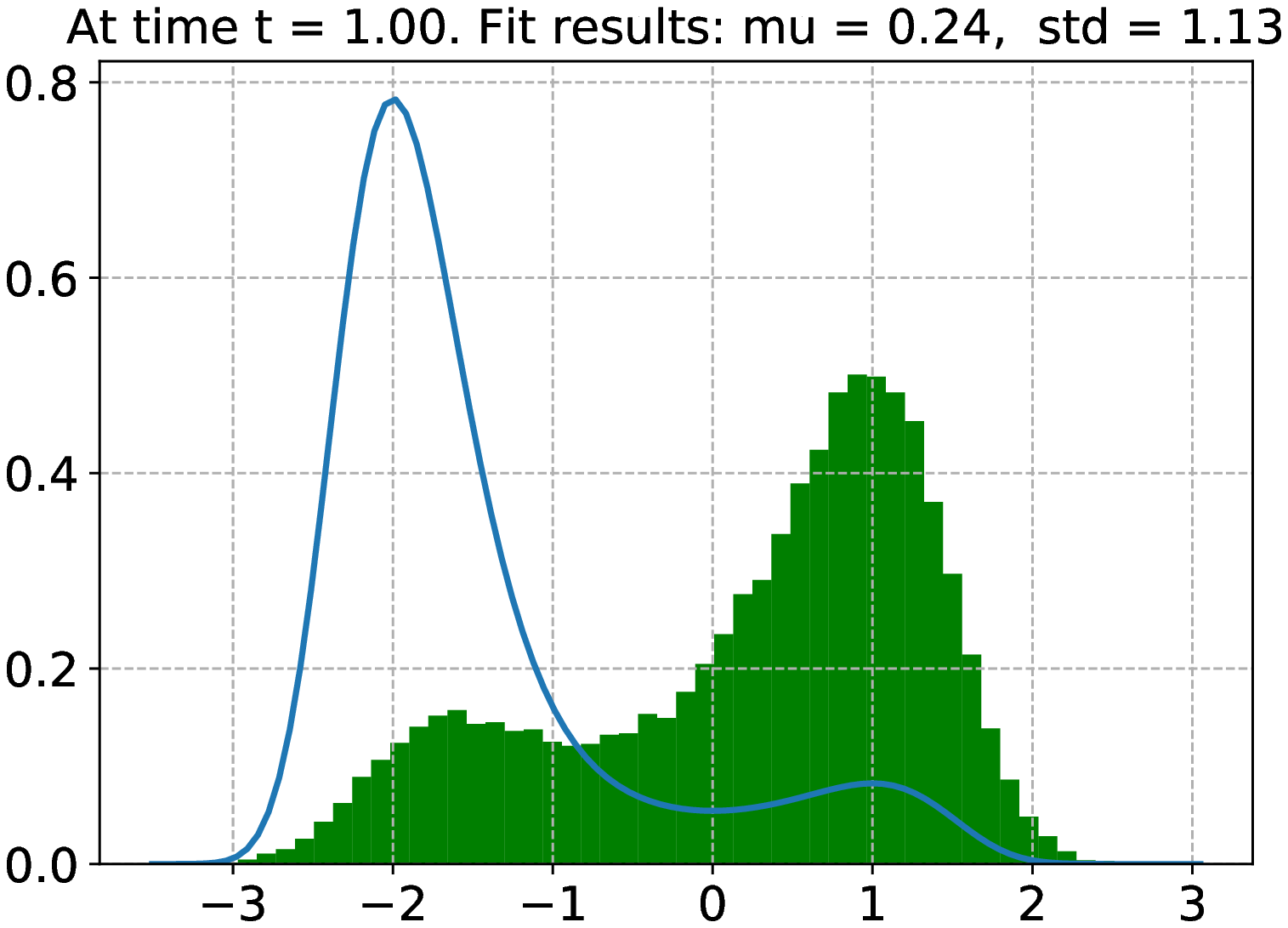} } \\
\subfloat[$t = 2$]{\includegraphics[width=0.5\textwidth]{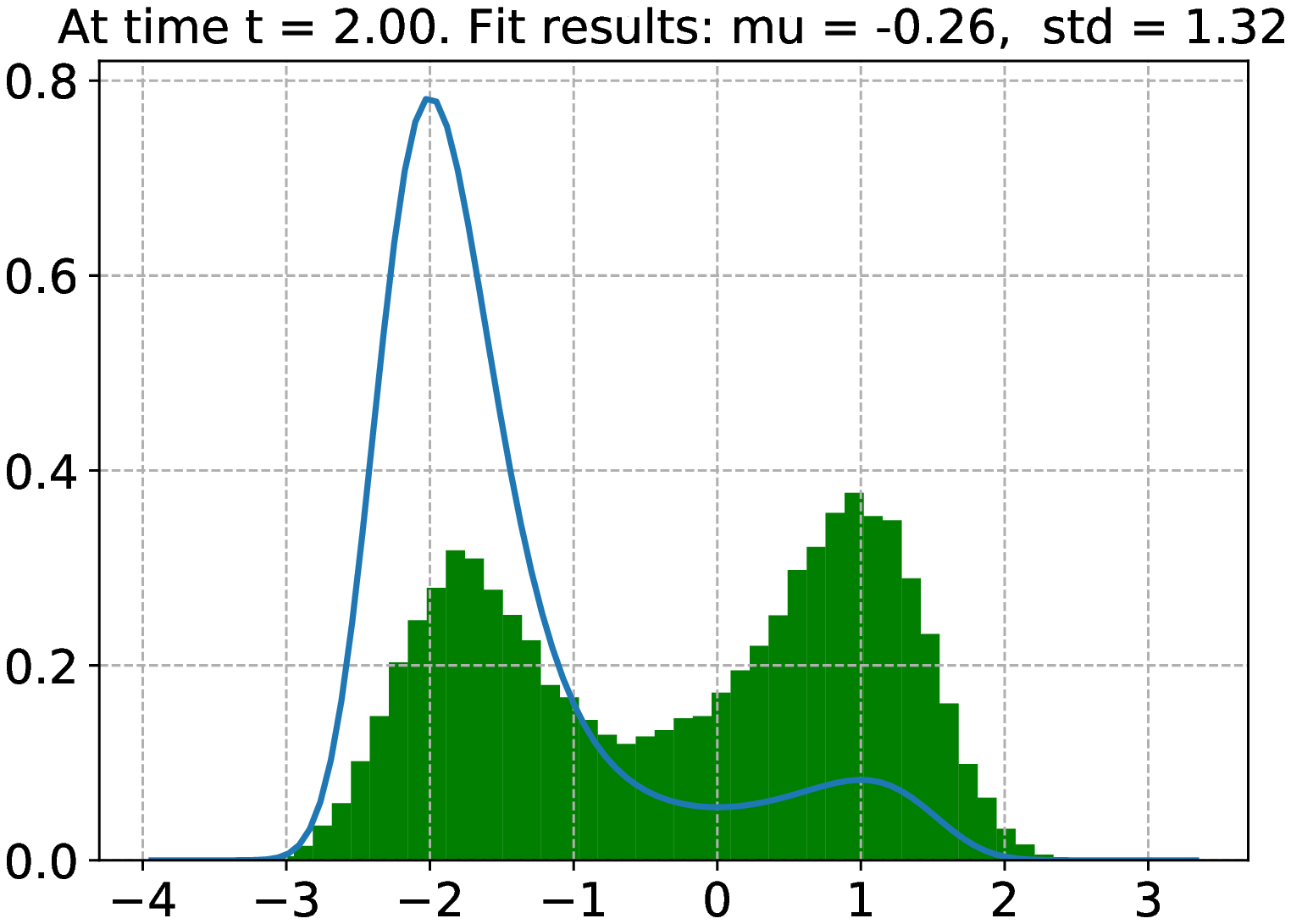} }
\subfloat[$t = 4$]{\includegraphics[width=0.5\textwidth]{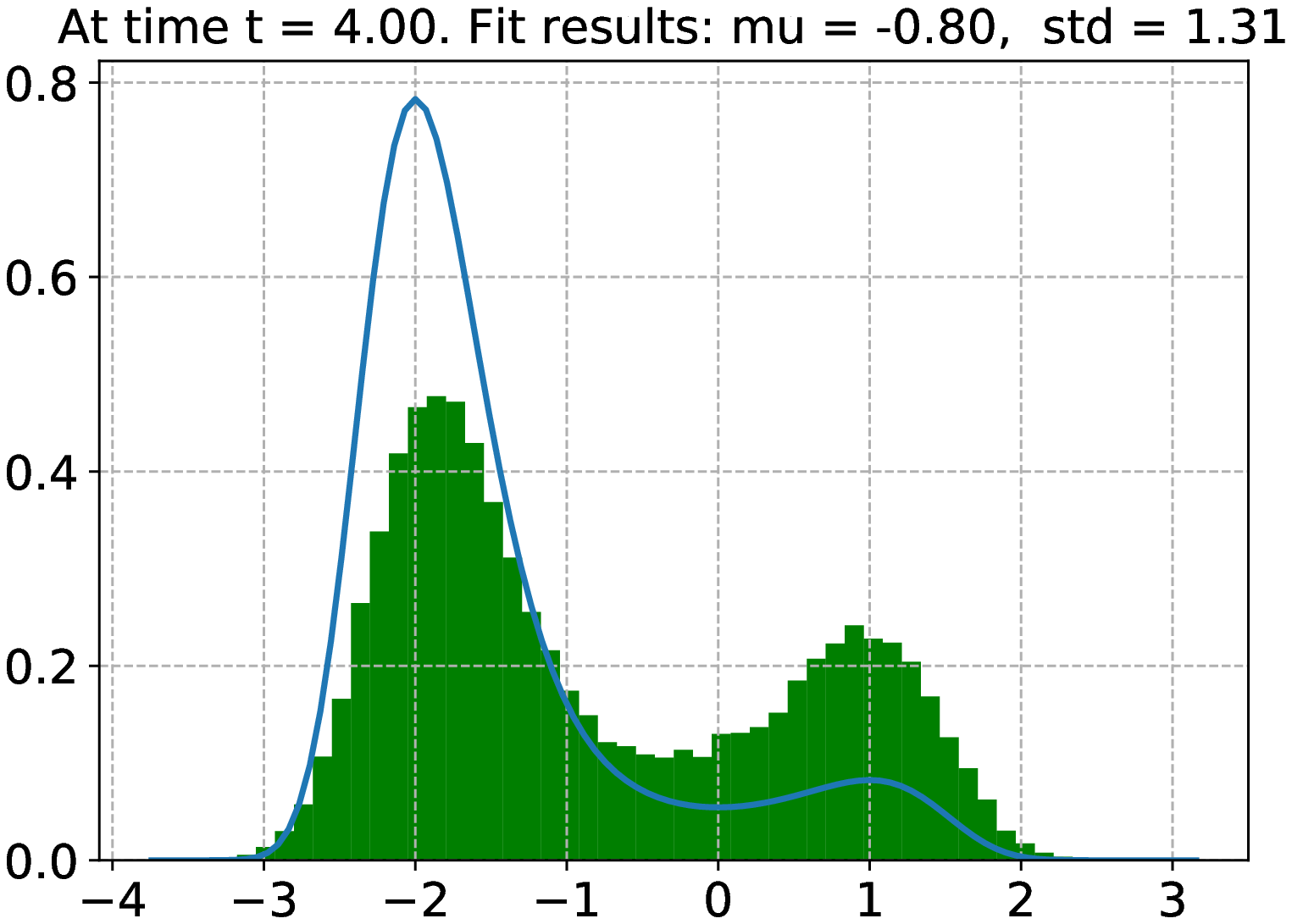} }
\caption{Example 2 (Section \ref{sec_ex2_asym}): the empirical distribution of $x$ at different times $t = 0, 0.0625, 0.25, 1, 2, 4, 8, 16, 32, 512$. The solid line is the reference Gibbs measure $\sim \exp(-V(x))$. It can be seen that provided the intial data all gathering at $x=1$, the distribution of $x$ expands, and moves towards the left, and finally creates a double-well shape.}
\label{ex2_asymfig1}
\end{figure}
\begin{figure}[htbp]
\ContinuedFloat
\centering
\subfloat[$t = 8$]{\includegraphics[width=0.5\textwidth]{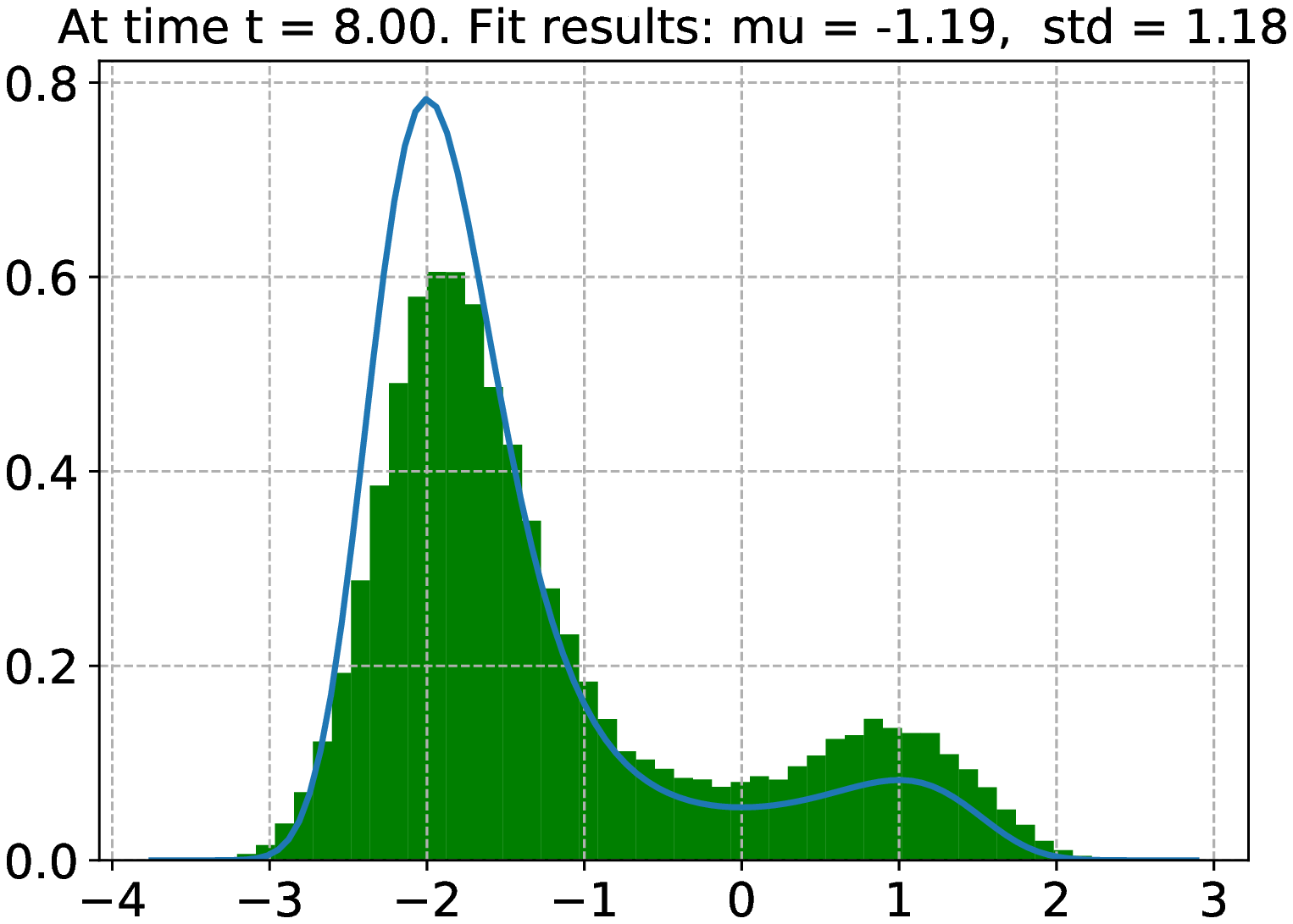} }
\subfloat[$t = 16$]{\includegraphics[width=0.5\textwidth]{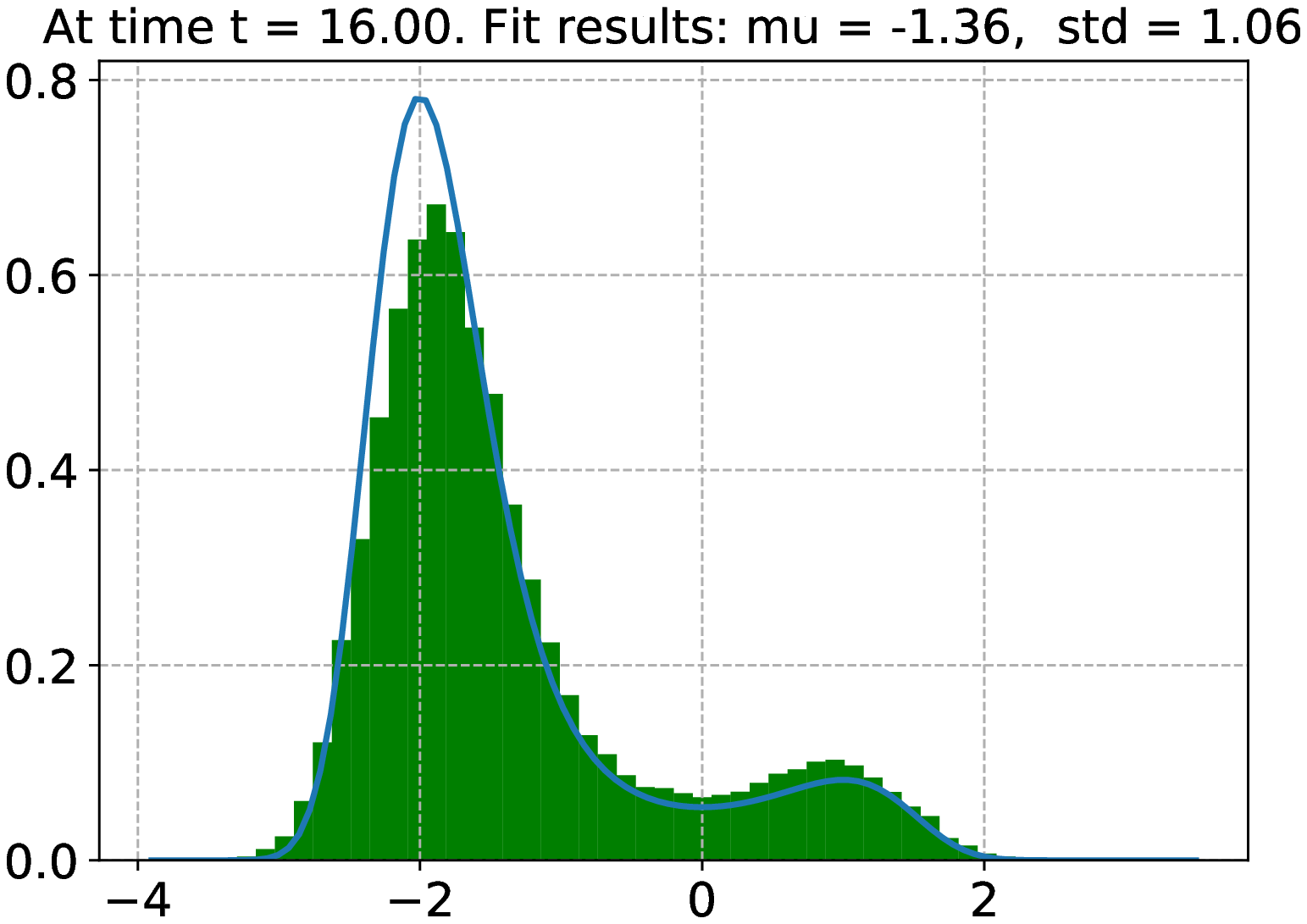} } \\
\subfloat[$t = 32 $]{\includegraphics[width=0.5\textwidth]{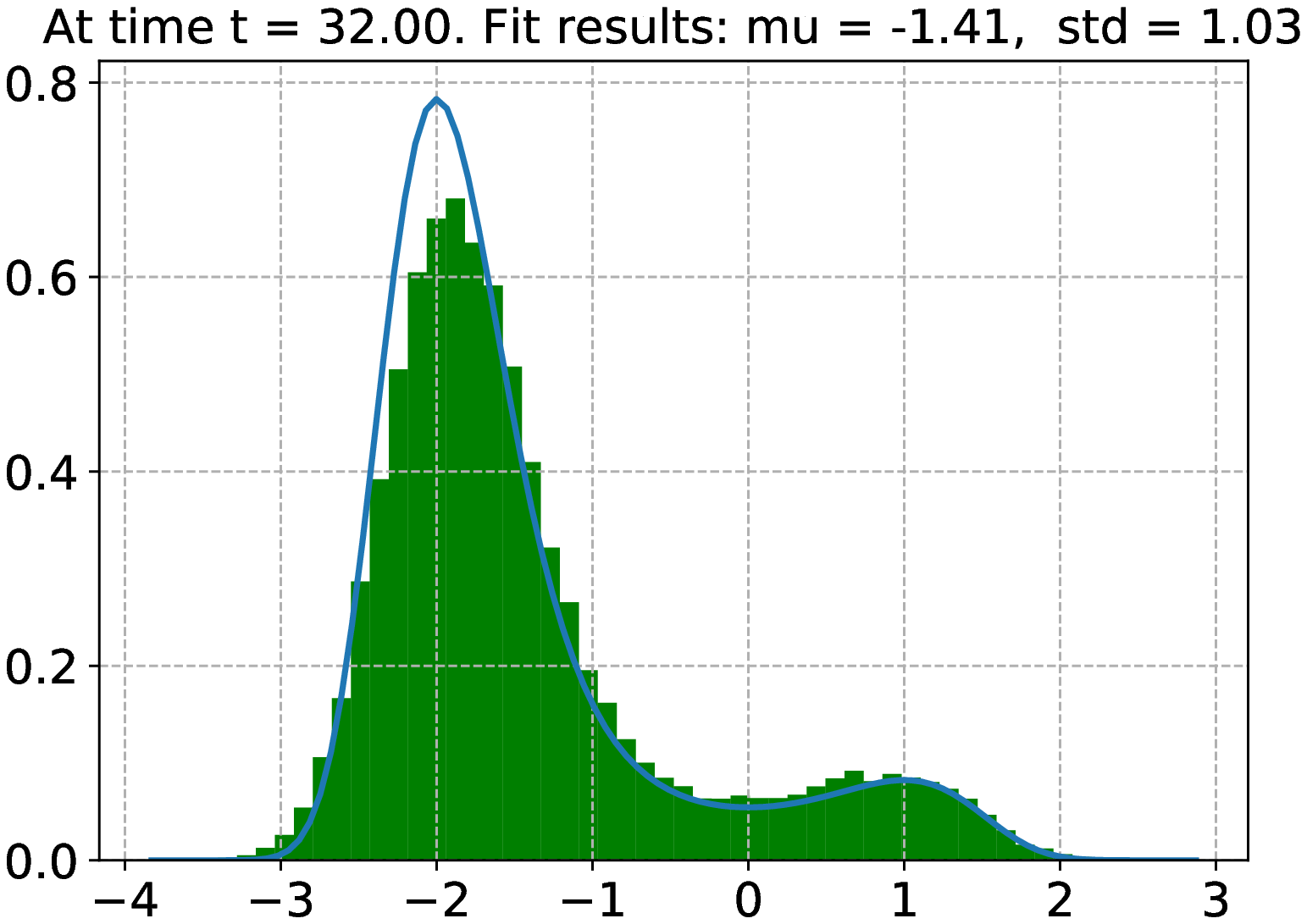} }
\subfloat[$t = 512 $]{\includegraphics[width=0.5\textwidth]{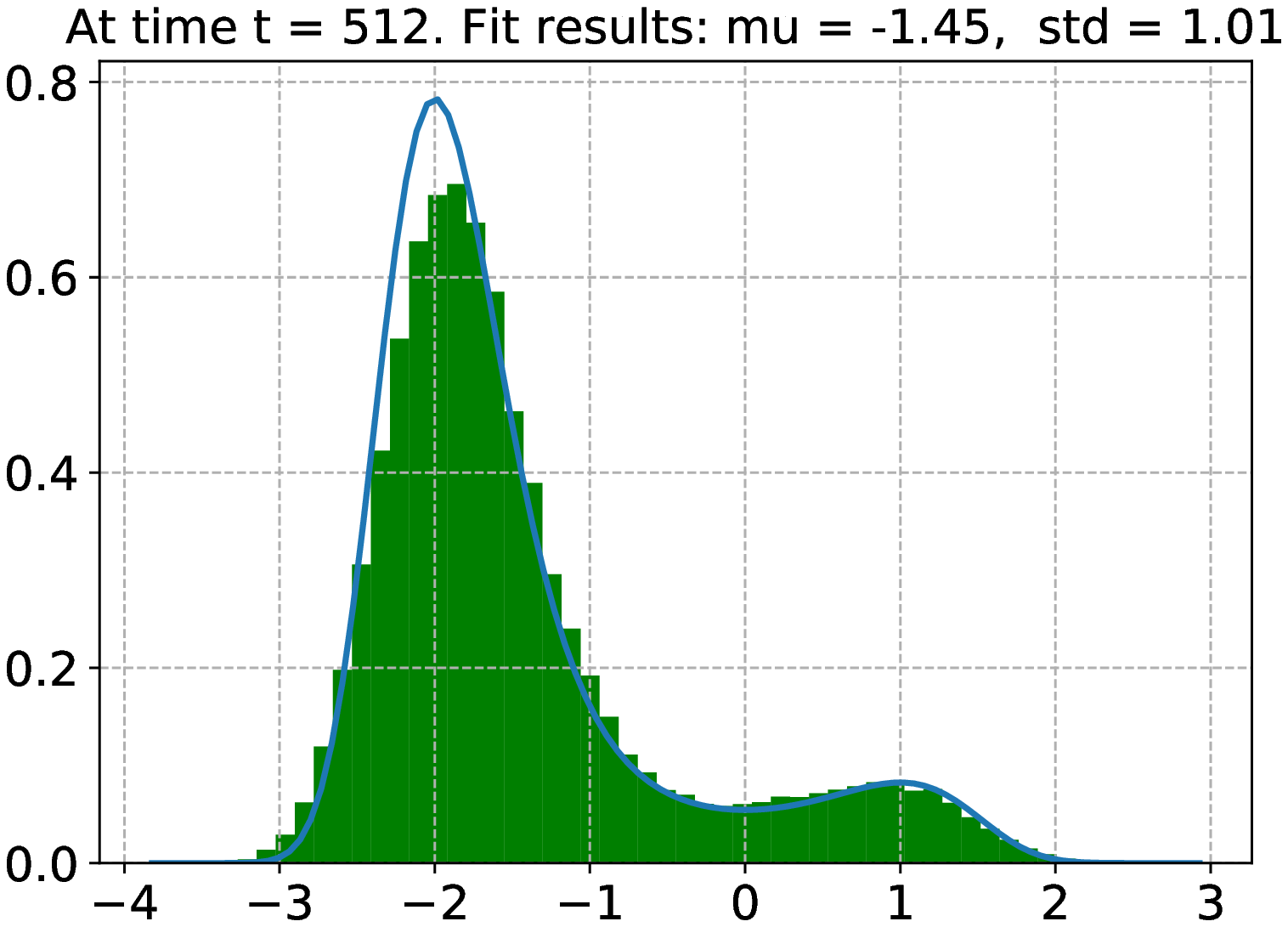} }
\captionsetup{list=off,format=cont}
\caption{Example 2 (Section \ref{sec_ex2_asym}): the empirical distribution of $x$ at various moments $t = 0, 0.0625, 0.25, 1, 2, 4, 8, 16, 32, 512$. }
\label{ex2_asymfig2}
\end{figure}

\subsection{Example 3 (2D double well potential)} \label{sec_ex3_2d}

For the 2D case, we should use the driven process as $B_H:=(B_H^1, B_H^2)$ where $B_H^i$ are two independent fractional Brownian motions. To see this, using the FDT, we find that the kernel is simply $\frac{1}{\Gamma(1-\alpha)}\tau^{-\alpha} I_d$ where $I_d$ is the $2\times 2$ identity matrix.
Hence, the equation \eqref{eq:fsde1} still holds.
The fractional Brownian motion $B_H$ here is different from the 2D fractional Brownian random field (see for example \cite{2dfbm}).

It is of interests to study double well potential in literature due to its application to describe the chemical phenomena, such as vibrionic spectra \cite{doublewell1}, proton transfer \cite{doublewell2} and etc. Here in the numerical test, we consider the double-well potential
$$V(x,y) = \frac{1}{4}(x^2+y^2)^2- x^2 - x^2y,$$
the quadratic coupling cases as is considered in, say, \cite{db_linearquadratic1,db_linearquadratic2}, which can be visualized in Fig. \ref{ex3_potentialfig} . In this example, consider intial datum $(x,y) = (0, 0.2)$ and compute till $T = 512$ with $\Delta t = 2^{-5}$ via the fast solver. We plot the empirical distribution at different time  $ t = 0, 0.25, 1, 2, 32, 512$ in Fig. \ref{ex3_fig1} and \ref{ex3_fig2} (cont.), where both 3D histogram and its 2D contour are plotted for the convenience of visualization. In Fig. \ref{ex3_msdfig}, the mean square displacement of $x(t)$, namely,
\begin{equation}
MSD(t) : = \mathbb{E}|x(t)-x(0)|^2,
\label{def_msd}
\end{equation}
is plotted in terms of time. The numerics indicate that the mean square displacement approaches to an equilibrium in an algebraic rate, instead of exponential. 

\begin{figure}[htbp]
\centering
\includegraphics[width=0.5\textwidth]{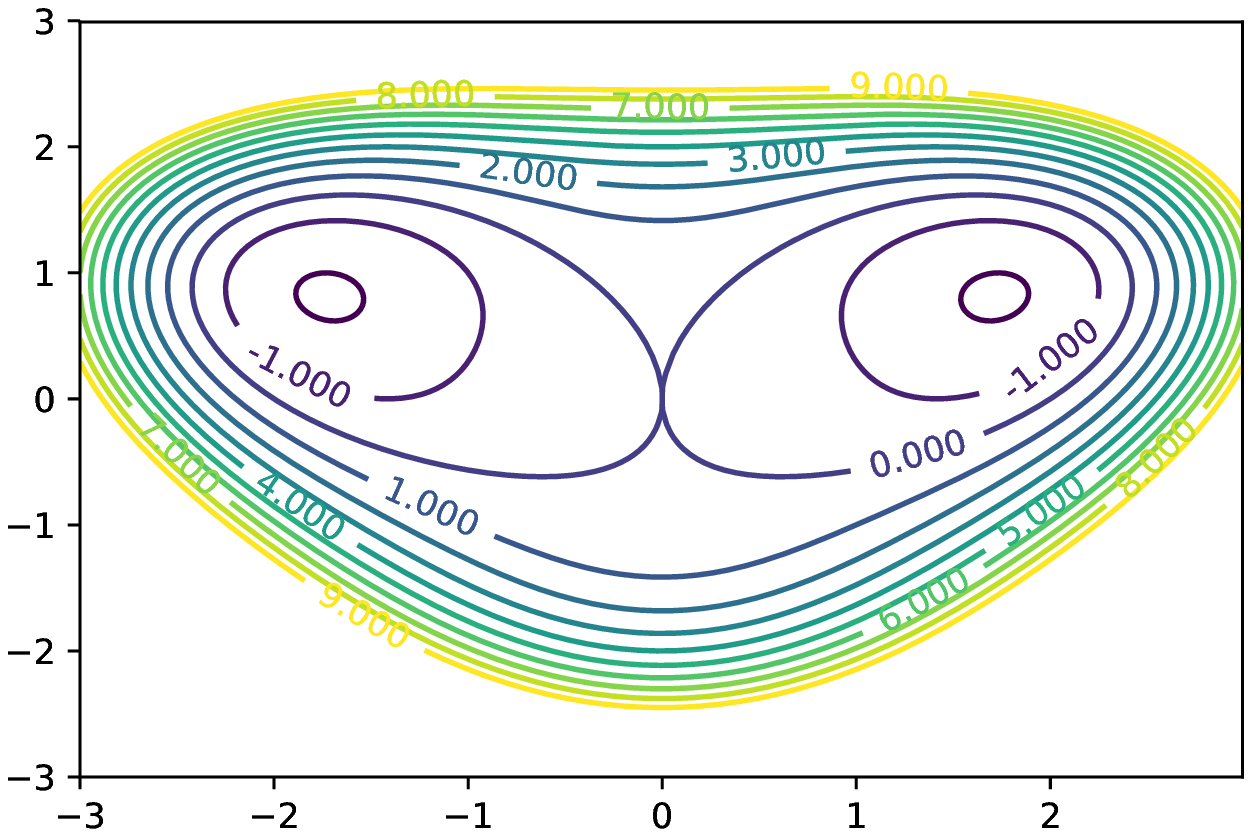}
\caption{Example 3 (Section \ref{sec_ex3_2d}): The mean square displacement of $x(t)$ versus time. }
\label{ex3_potentialfig}
\end{figure}

\begin{figure}[htbp]
\centering
\subfloat[$t = 0$]{\includegraphics[width=0.5\textwidth]{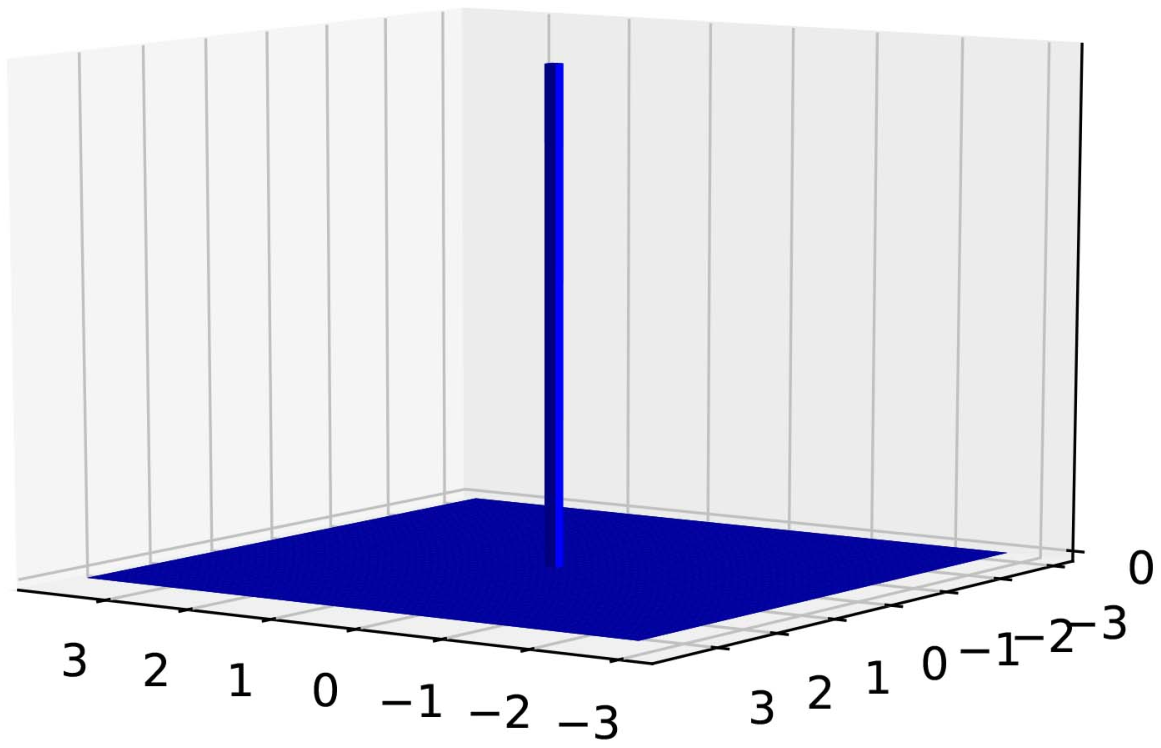} }
\subfloat[$t = 0$]{\includegraphics[width=0.5\textwidth]{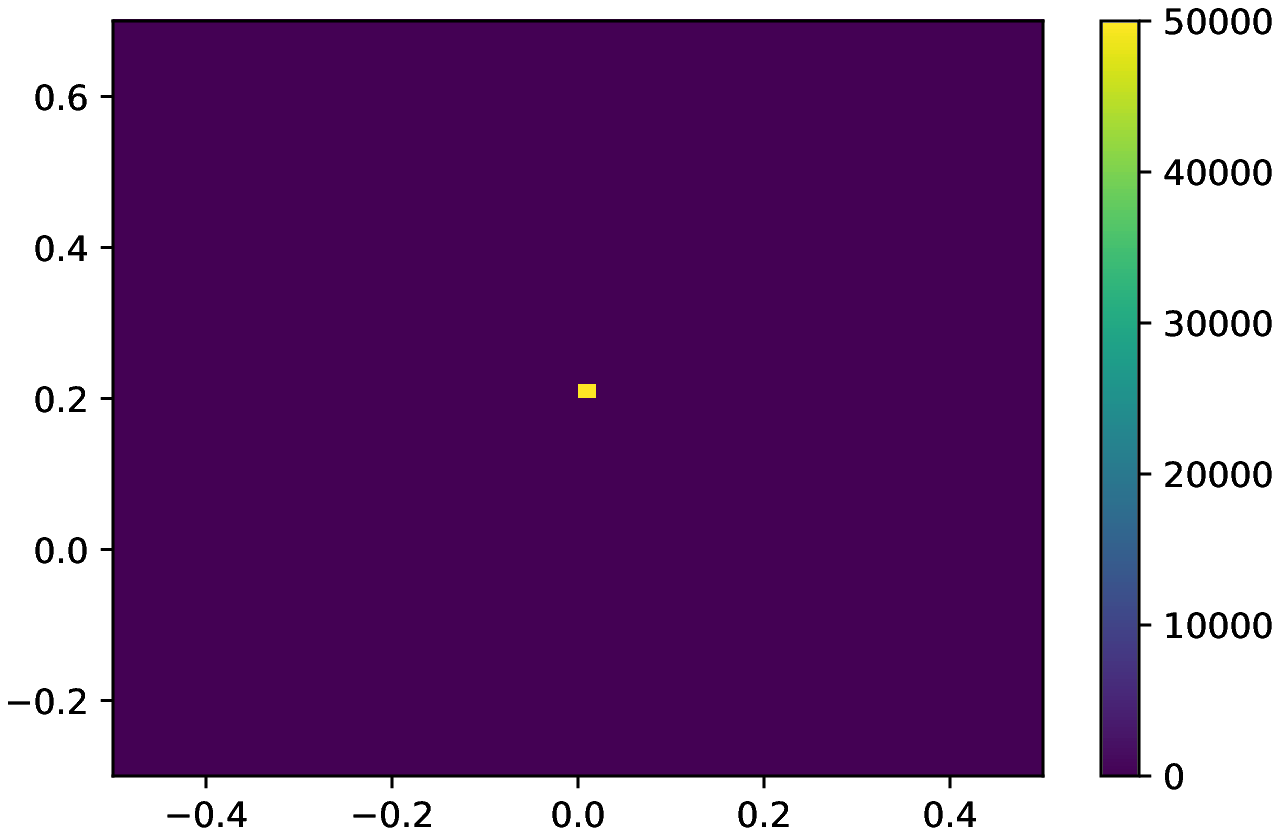} }\\
\subfloat[$t = 0.25$]{\includegraphics[width=0.5\textwidth]{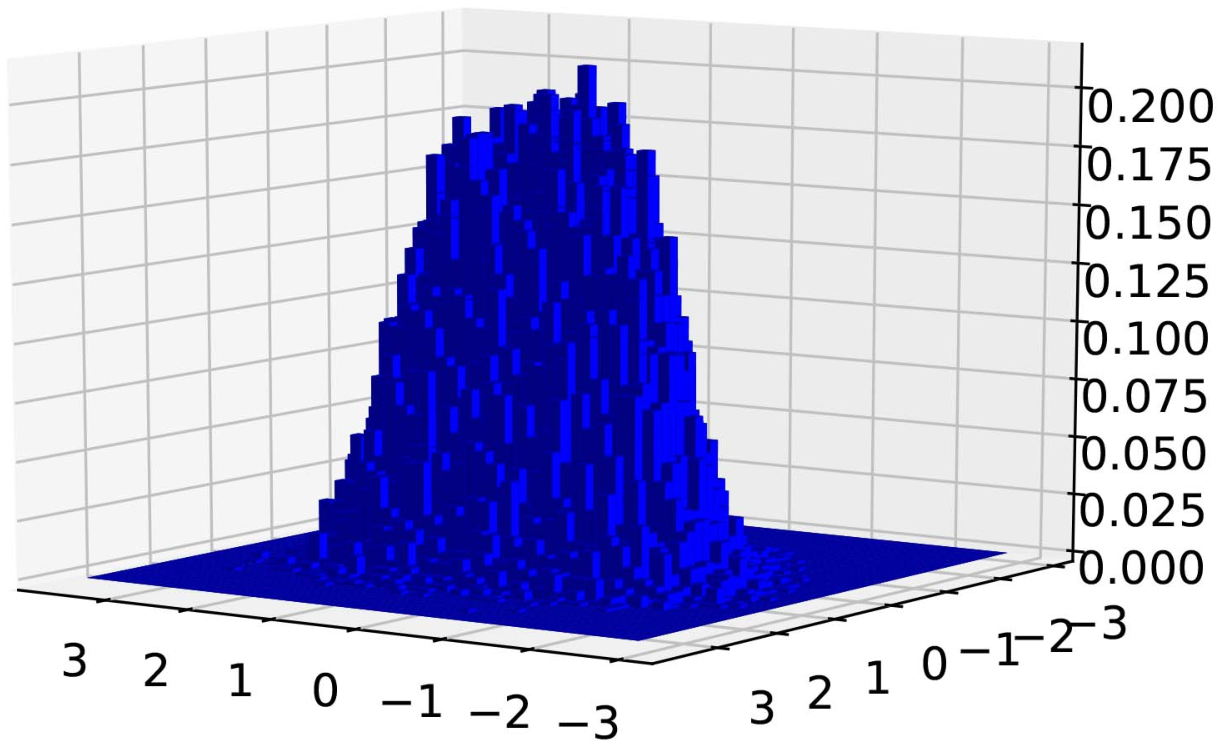} }
\subfloat[$t = 0.25$]{\includegraphics[width=0.5\textwidth]{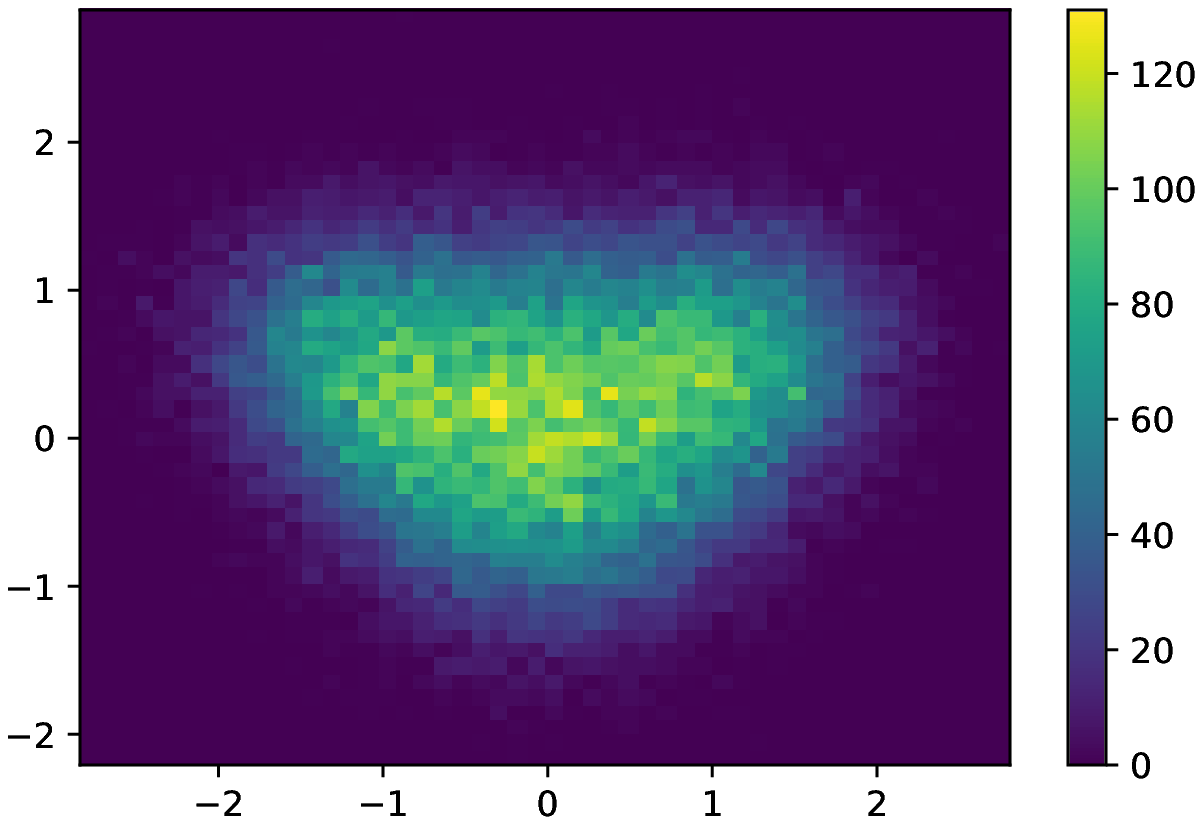} }\\
\subfloat[$t = 1$]{\includegraphics[width=0.5\textwidth]{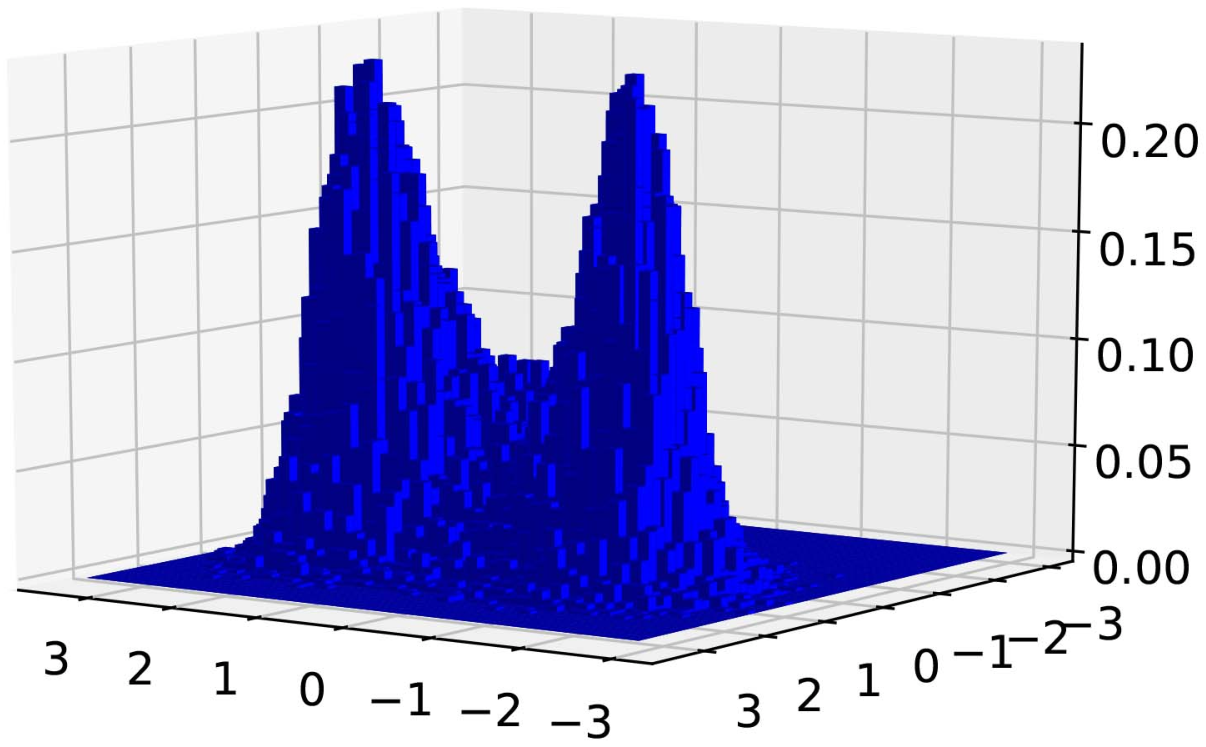} }
\subfloat[$t = 1$]{\includegraphics[width=0.5\textwidth]{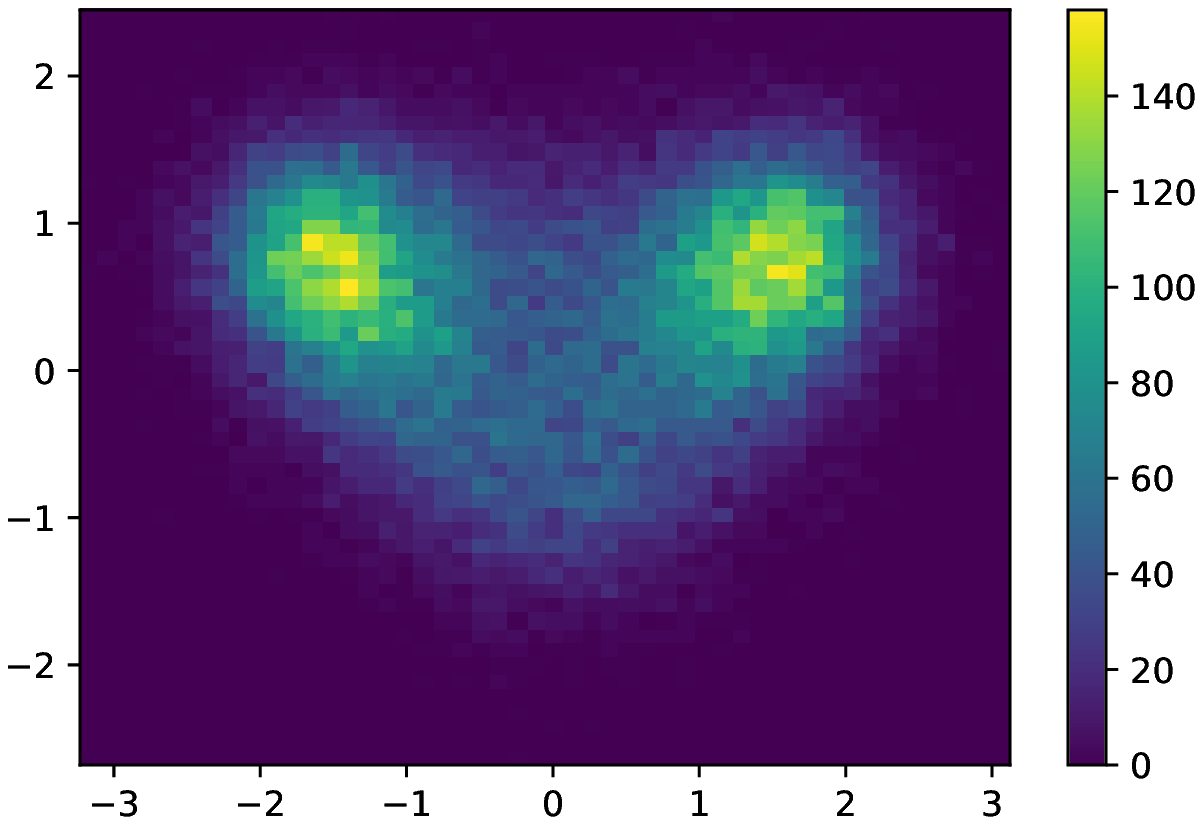} }\\
\caption{Example 3 (Section \ref{sec_ex3_2d}): the empirical distribution of $x$ in 3D histogram and 2D contour, respectively, in at various moments $t = 0, 0.25, 1, 2, 32, 512$. The solid line is the reference Gibbs measure $\sim \exp(-V(x))$. It can be seen that provided the intial data all gathering at $x=(0, 0.2)$, the distribution of $x$ expands, and moves towards a double-well shape.}
\label{ex3_fig1}
\end{figure}

\begin{figure}[htbp]
\ContinuedFloat

\centering
\subfloat[$t = 2$]{\includegraphics[width=0.5\textwidth]{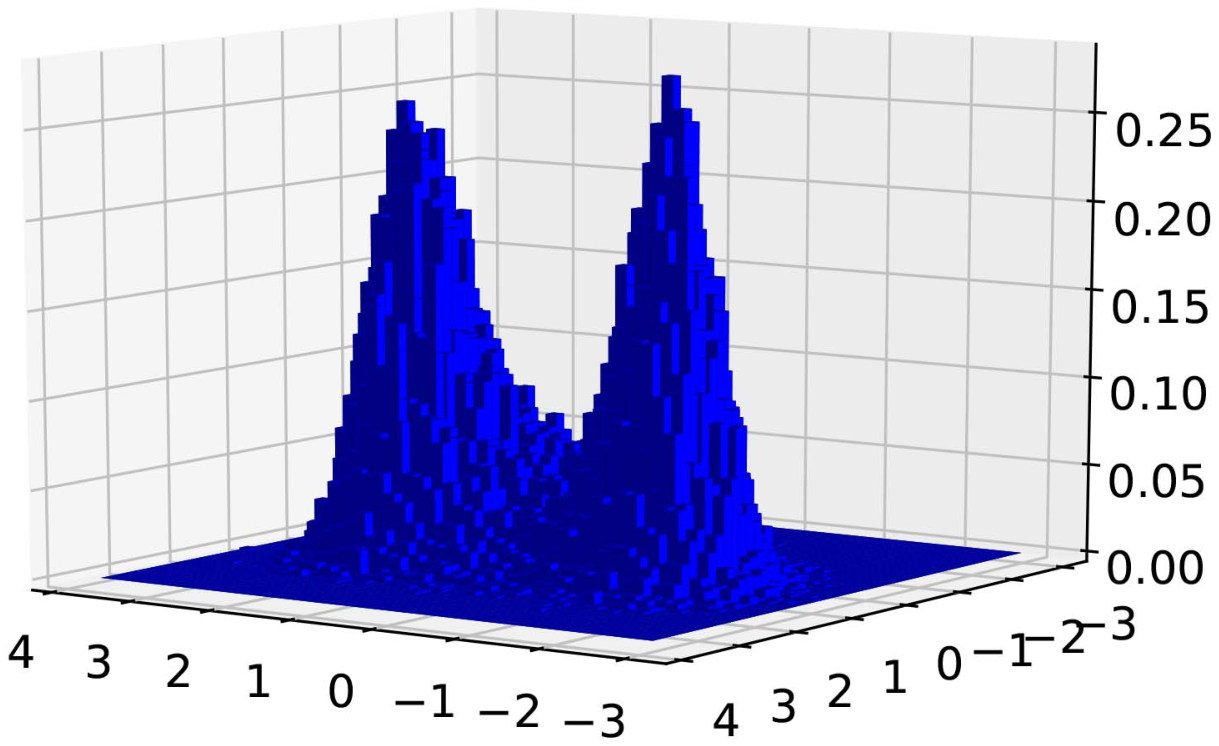} }
\subfloat[$t = 2$]{\includegraphics[width=0.5\textwidth]{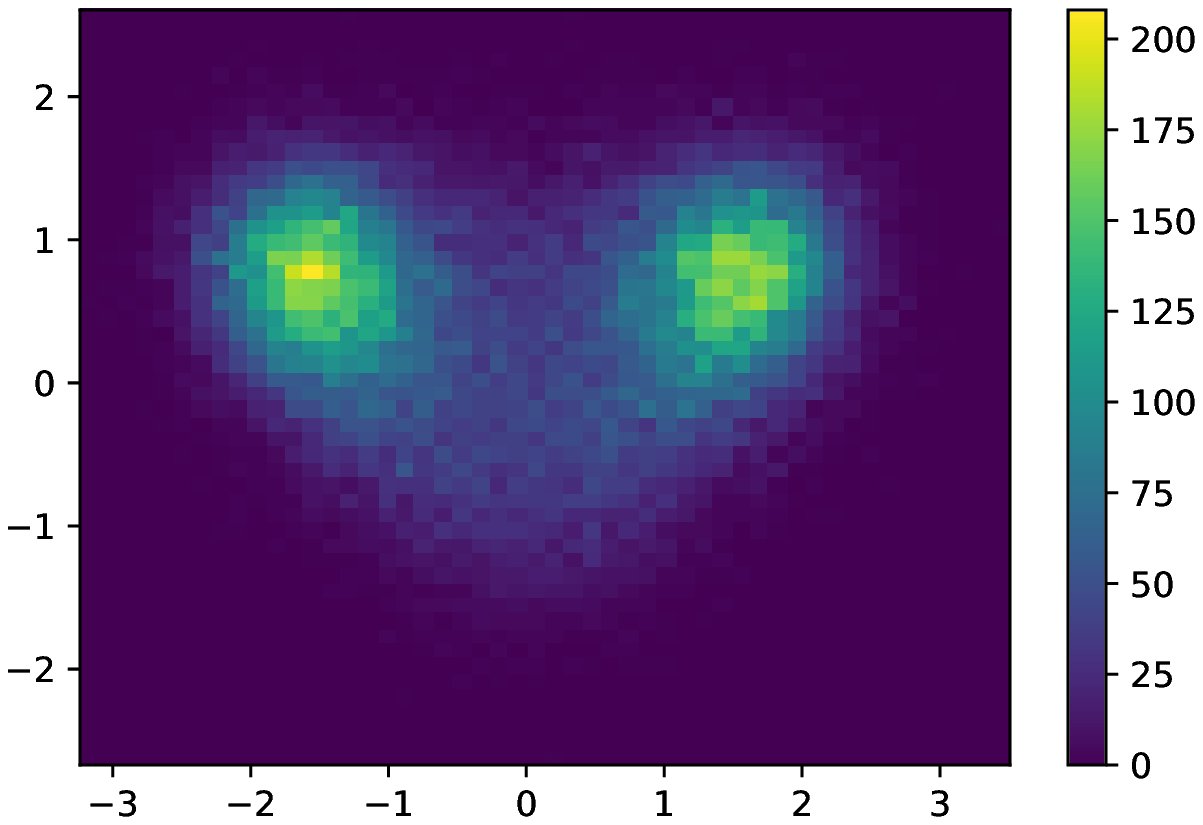} }\\
\subfloat[$t = 32$]{\includegraphics[width=0.5\textwidth]{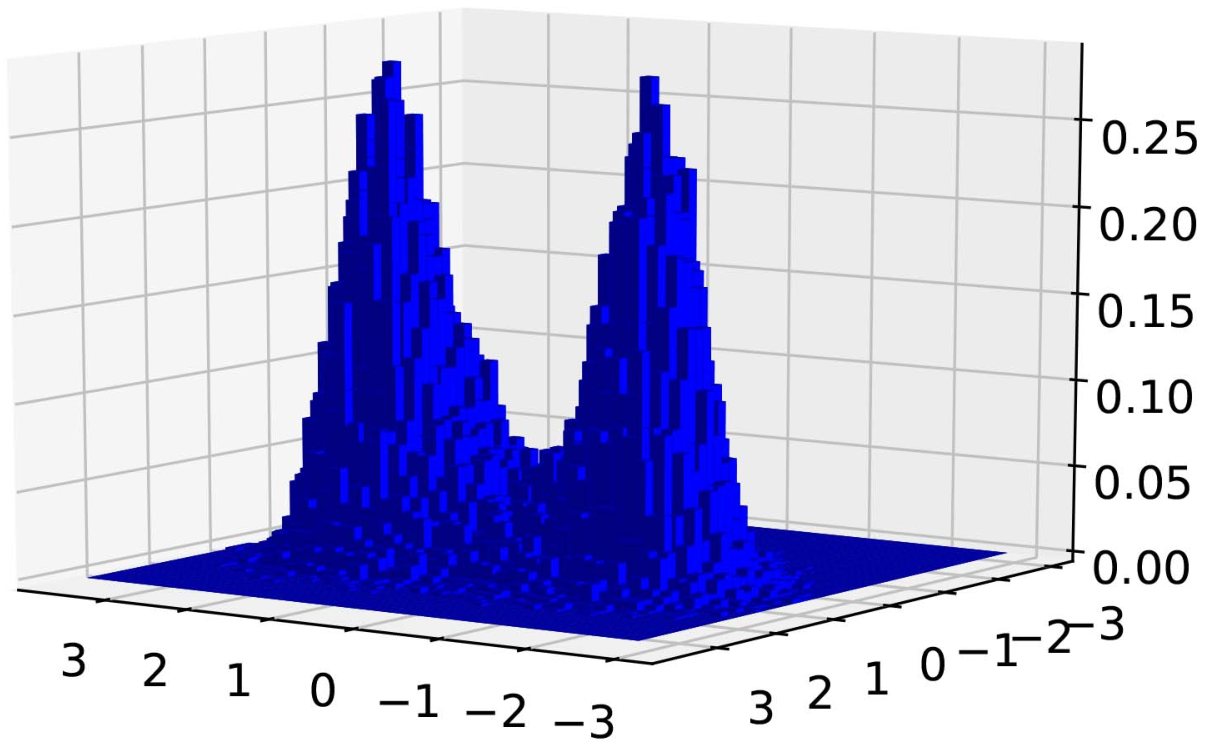} }
\subfloat[$t = 32$]{\includegraphics[width=0.5\textwidth]{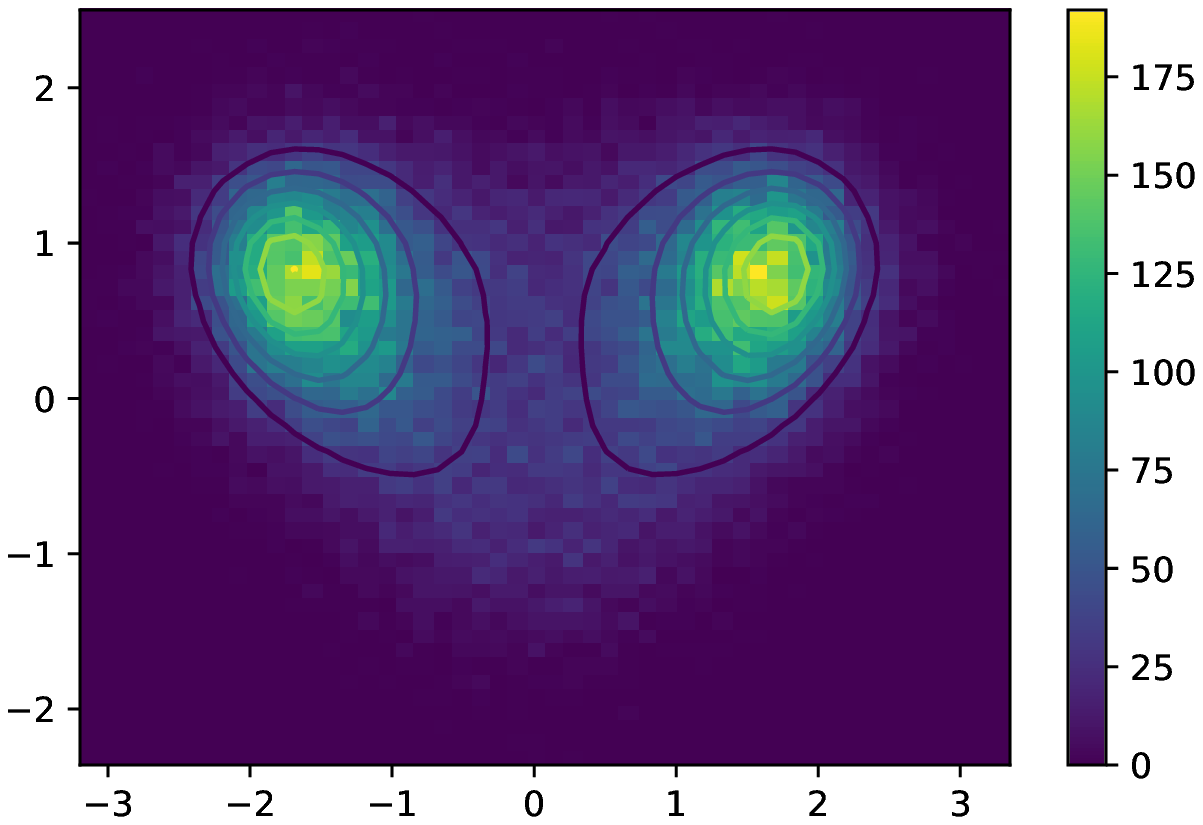} }\\
\subfloat[$t = 512$]{\includegraphics[width=0.5\textwidth]{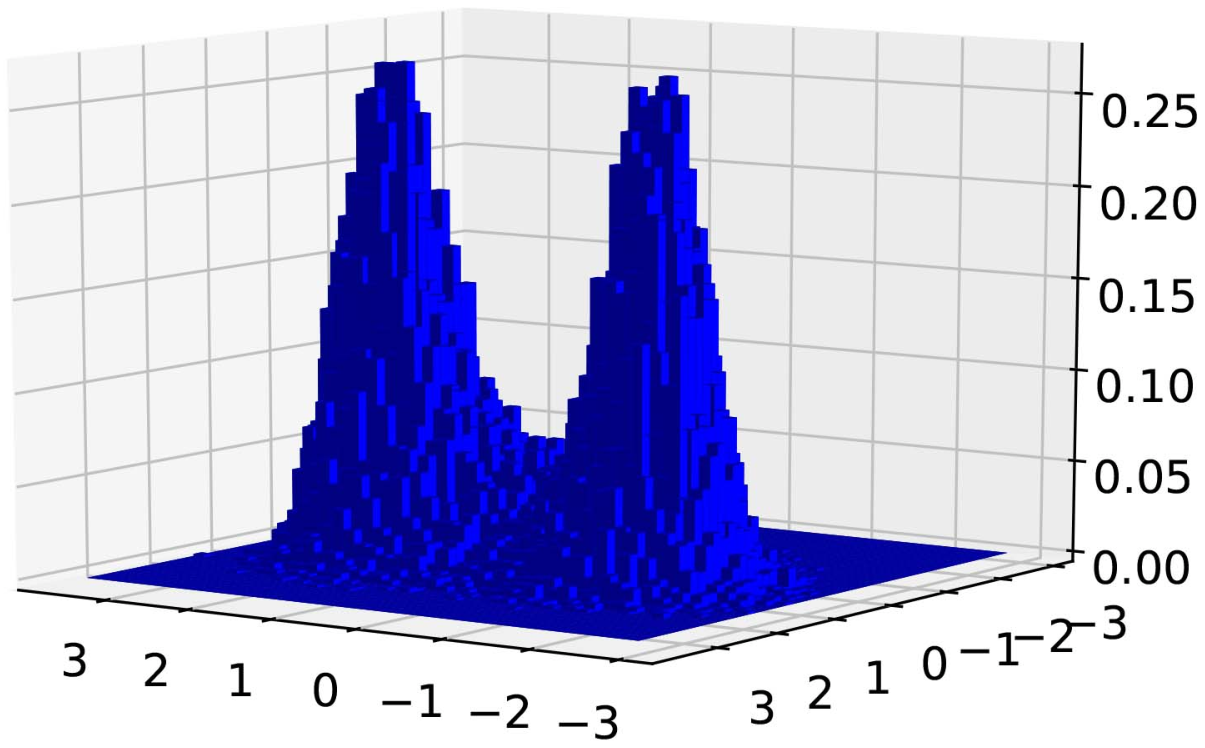} }
\subfloat[reference Gibbs measure]{\includegraphics[width=0.5\textwidth]{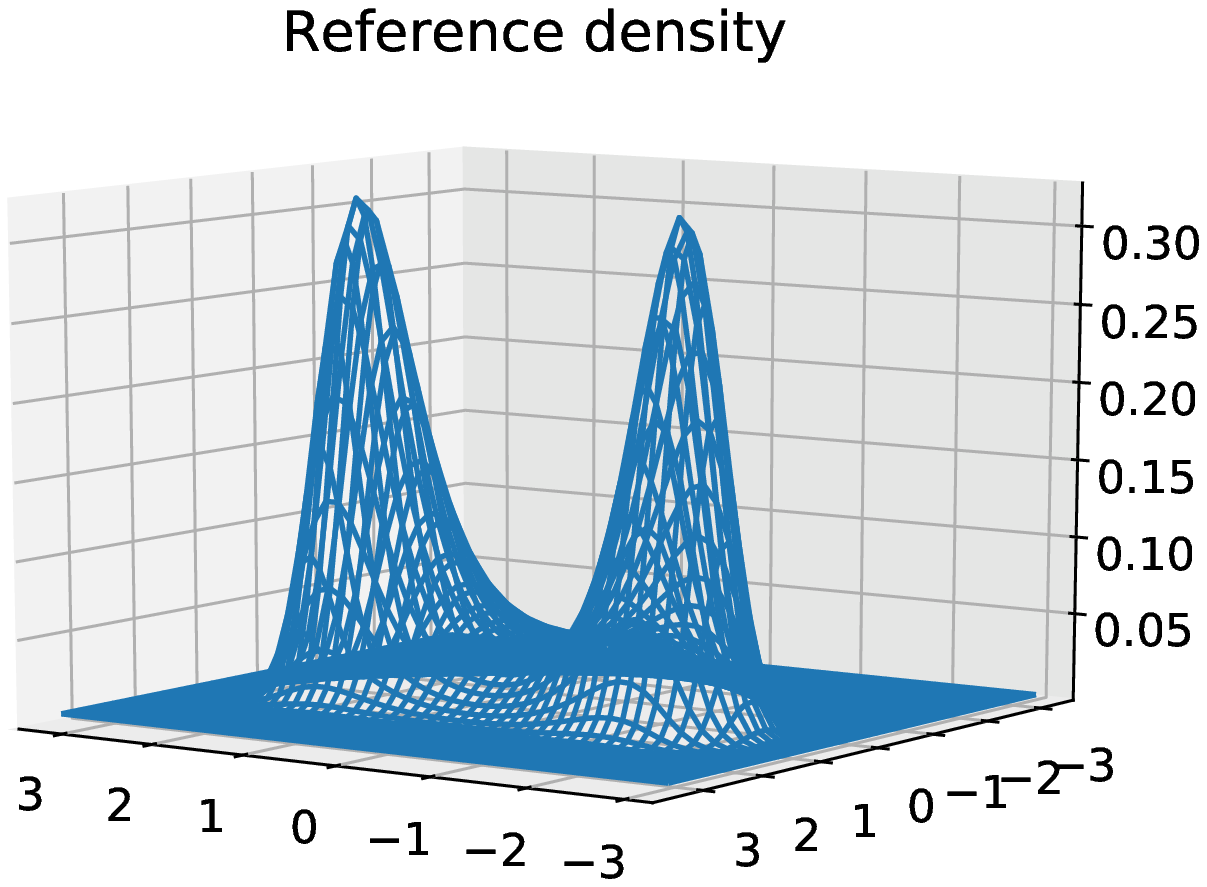} }
 \captionsetup{list=off,format=cont}
\caption{Example 3 (Section \ref{sec_ex3_2d}): the empirical distribution of $x$ in 3D histogram and 2D contour, respectively, in at various moments $t = 0, 0.25, 1, 2, 32, 512$. }
\label{ex3_fig2}
\end{figure}

\begin{figure}[htbp]
\centering
\subfloat[Mean square displacement]{\includegraphics[width=0.4\textwidth]{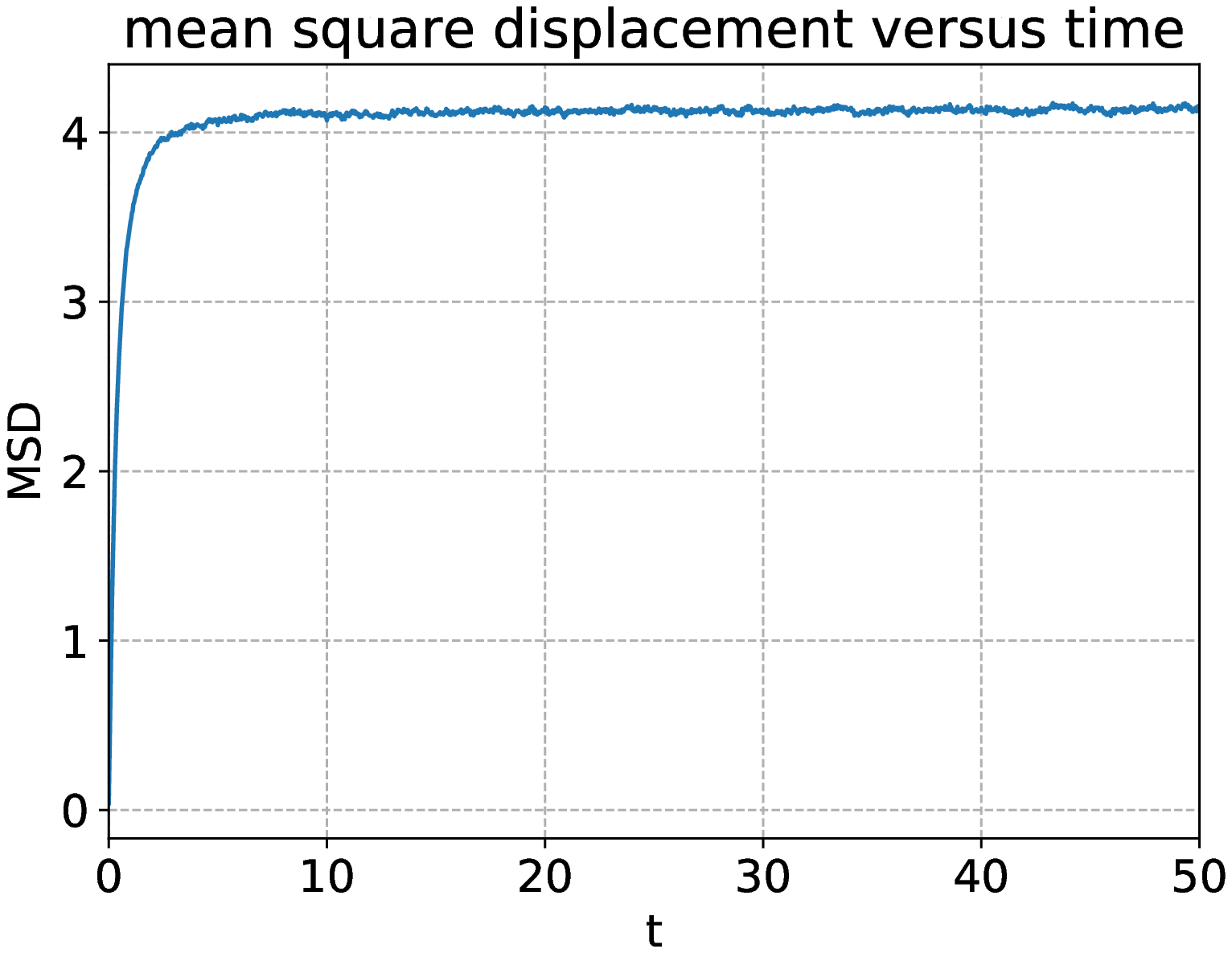}}
\subfloat[$MSD(\infty) - MSD(t)$]{\includegraphics[width=0.4\textwidth]{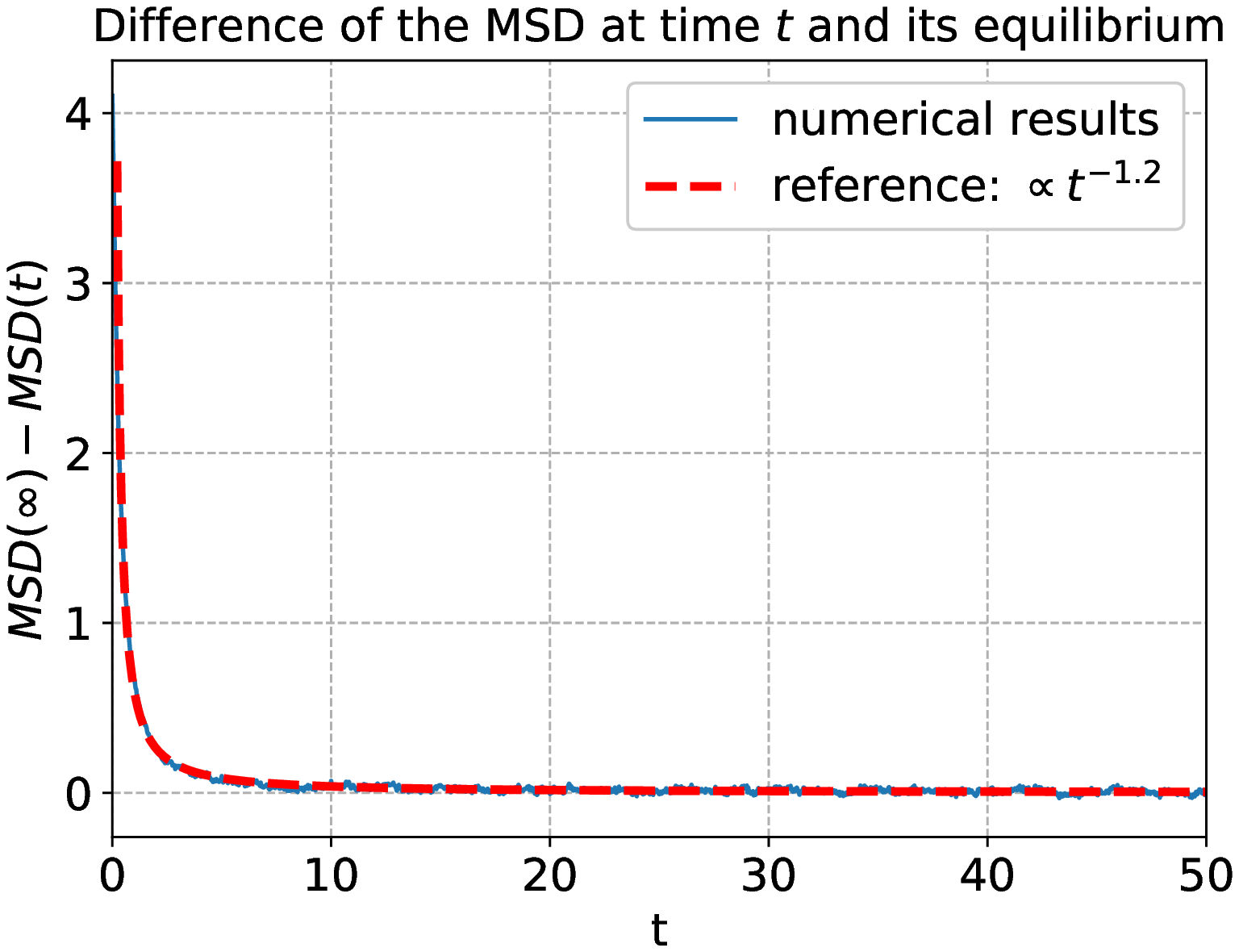}}
\caption{Example 3 (Section \ref{sec_ex3_2d}): The mean square displacement is defined as \eqref{def_msd}. It can be seen that the mean square displacement approaches an equilibrium algebraically, instead of exponentially.}
\label{ex3_msdfig}
\end{figure}

\appendix

\section{Proof of Lemma \ref{lmm:wellposedmodifiedeq}}\label{app:modifiedeq}

\begin{proof}[Proof of Lemma \ref{lmm:wellposedmodifiedeq}]
Recall the SOE approximation kernel $\gamma(t)$ in \eqref{eq:soekernel}, which is positive everywhere. Recall \eqref{eq:modifiedInt}:
\begin{gather}\label{eq:y}
y(t)=y_0+\int_0^t\gamma(t-s)f(y(s))\,ds.
\end{gather}
Assume that there are two continuous solutions $y_1(t)$ and $y_2(t)$ on some interval $I=[0, T]\cap [0, T_b)$. Let $T_1=\min(T, T_b)$.
We clearly have
\[
|y_1(t)-y_2(t)|\le \int_0^t\gamma(t-s)|f(y_1(s))-f(y_2(s))|\,ds.
\]
Assume that $y_1(t)=y_2(t)$ for all $t<t^*$ for some $t^*\in [0, T_1)$. Then, we can pick $\delta$ small enough and then $|y_i(t)|\le M$ for some $M>0$ and all $t\in [0, t^*+\delta]$. Let $L$ be the Lipschitz constant for $f$ on the interval $[0, M]$. Then,  for any $t\in [t^*, t^*+\delta]$,
\begin{gather*}
|y_1(t)-y_2(t)|\le \int_{t^*}^t\gamma(t-s)|f(y_1(s))-f(y_2(s))|\,ds
\le L\left(\sup_{s\in (t^*, t^*+\delta)}|y_1(s)-y_2(s)|\right)\sup_{t\in [t^*, t^*+\delta]}\int_{t^*}^{t}\gamma(t-s)\,ds
\end{gather*}
If $\delta$ is sufficiently small, we have
\[
\nu=L\sup_{t\in [t^*, t^*+\delta]}\int_{t^*}^{t}\gamma(t-s)\,ds<1,
\]
 then we have
\[
(1-\nu)\sup_{s\in [t^*, t^*+\delta]}|y_1(s)-y_2(s)|\le 0.
\]
This means the set of all such $t^*$ is open in $I$ with the inherited topology from $\mathbb{R}$. This set is clearly also closed under the inherited topology by the continuity of $y_i$. Hence, the set of all such $t^*$ is $I$, which means $y_1(t)=y_2(t)$ on $I$.

Now, we establish the existence result and the desired properties. Consider the following standard Picard sequence,
\[
y_{n+1}(t)=y_0+\int_0^t \gamma(t-s)f(y_n(s))\,ds,~~y_0(t)=y_0.
\]
By induction, it is not hard to see
\[
y_{n+1}(t)\ge y_n(t),~t\in [0,\infty),
\]
and for each $n$ $y_n(t)\ge y_0$ and is non-decreasing.

Clearly, $y_0\le v(t)$ for $t\in [0, T]\cap [0, T_b)$.
Assume this is true for $n$. Consider $y_{n+1}$.
For $t\in [0, k]\cap [0, T_b)$, it is easy to see
\begin{align*}
y_{n+1}(t)
&\le y_0+\frac{1}{\Gamma(\alpha)}\int_0^t(t-s)^{\alpha-1}f(y_n(s))\,ds \\
& \le y_0+\frac{1}{\Gamma(\alpha)}\int_0^t(t-s)^{\alpha-1}f(v(s))\,ds
+\frac{\epsilon T^{1-\alpha}}{\Gamma(\alpha)}\int_0^{t}(t-s)^{\alpha-1}f(v(s))\,ds \\
&=v(t)
\end{align*}
For $t\in (k, T]\cap [0, T_b)$, we have
\begin{multline*}
y_{n+1}(t)\le y_0+\frac{1}{\Gamma(\alpha)}\int_0^t(t-s)^{\alpha-1}f(y_n(s))\,ds
+\frac{\epsilon}{\Gamma(\alpha)}\int_0^{t-k}f(y_n(s))\,ds\\
\le
y_0+\frac{1}{\Gamma(\alpha)}\int_0^t(t-s)^{\alpha-1}f(y_n(s))\,ds
+\frac{\epsilon T^{1-\alpha}}{\Gamma(\alpha)}\int_0^{t}(t-s)^{\alpha-1}f(y_n(s))\,ds\\
\le y_0+\frac{1}{\Gamma(\alpha)}\int_0^t(t-s)^{\alpha-1}f(v(s))\,ds
+\frac{\epsilon T^{1-\alpha}}{\Gamma(\alpha)}\int_0^{t}(t-s)^{\alpha-1}f(v(s))\,ds=v(t).
\end{multline*}

Hence, on $[0, T]\cap [0, T_b)$, we have
\[
y_{n-1}(t)\le y_n(t)\le \ldots \le v(t).
\]
This means $y_n(t)$ increases to a non-decreasing function  $y(t)$ pointwise on $[0, T]\cap [0, T_b)$. By the monotone convergence theorem, we have
\[
y(t)=y_0+\int_0^t\gamma(t-s)f(y(s))\,ds,~\forall t\in [0, T]\cap [0, T_b).
\]
Since $y(t)\le v(t)$, we conclude that $y(t)$ must be continuous. This means that $y(\cdot)$ is a continuous solution with the desired properties.

If $f$ is Lipschitz, it is well known that $v(\cdot)$ exists globally, or  $T_b=\infty$ (see \cite{feng2017} for example) and the claim follows.
\end{proof}

\bibliographystyle{unsrt}
\bibliography{probsde}

\end{document}